\documentclass[mnsc,nonblindrev]{optonline} 

\usepackage{epstopdf}
\usepackage{bm}
\usepackage{cases}
\usepackage{paralist}
\usepackage{empheq}
\usepackage{multirow}
\usepackage{subfig}
\usepackage{subfloat}
\usepackage[table]{xcolor}
\usepackage{enumitem}
\usepackage{capt-of}
\usepackage{algorithm}
\usepackage{url}
\usepackage[noend]{algpseudocode}
\usepackage{float}
\usepackage{caption}
\usepackage{wrapfig}
\usepackage{dsfont}
\usepackage{appendix}
\usepackage[normalem]{ulem}
\usepackage{soul}

\usepackage{tikz}             
\usetikzlibrary{arrows.meta}
\usetikzlibrary{calc}

\graphicspath{{figures/}}

\usepackage{array}
\newcolumntype{L}[1]{>{\raggedright\let\newline\\\arraybackslash\hspace{0pt}}m{#1}}
\newcolumntype{C}[1]{>{\centering\let\newline\\\arraybackslash\hspace{0pt}}m{#1}}
\newcolumntype{R}[1]{>{\raggedleft\let\newline\\\arraybackslash\hspace{0pt}}m{#1}}

\newtheorem{approximation}{Approximation}

\OneAndAHalfSpacedXI
\usepackage{natbib}
 \bibpunct[, ]{(}{)}{,}{a}{}{,}%
 \def\bibfont{\small}%

\TheoremsNumberedThrough     
\ECRepeatTheorems
\EquationsNumberedThrough    
\MANUSCRIPTNO{-}

\renewcommand{\qed}{\hfill $\square$}

\newcommand{\revision}[1]{{\color{black}{#1}}}

\begin{document}

\RUNTITLE{A Planner-Trader Decomposition for Multi-Market Hydro Scheduling}
\RUNAUTHOR{Schindler, Rujeerapaiboon, Kuhn, Wiesemann}
\TITLE{A Planner-Trader Decomposition for\\ Multi-Market Hydro Scheduling}


\ARTICLEAUTHORS{
\AUTHOR{Kilian Schindler}
\AFF{\textit{Risk Analytics and Optimization Chair, \'Ecole Polytechnique F\'ed\'erale de Lausanne, Switzerland} \\
\EMAIL{kilian.schindler@epfl.ch}}

\AUTHOR{Napat Rujeerapaiboon}
\AFF{\textit{Department of Industrial Systems Engineering and Management, National University of Singapore, Singapore} \\
\EMAIL{napat.rujeerapaiboon@nus.edu.sg}}

\AUTHOR{Daniel Kuhn}
\AFF{\textit{Risk Analytics and Optimization Chair, \'Ecole Polytechnique F\'ed\'erale de Lausanne, Switzerland} \\
\EMAIL{daniel.kuhn@epfl.ch}}

\AUTHOR{Wolfram Wiesemann}
\AFF{\textit{Imperial College Business School, Imperial College London, United Kingdom} \\
\EMAIL{ww@imperial.ac.uk}}
}

\ABSTRACT{Peak/off-peak spreads on European electricity forward and spot markets are eroding due to the ongoing nuclear phaseout in Germany and the steady growth in photovoltaic capacity. The reduced profitability of peak/off-peak arbitrage forces hydropower producers to recover part of their original profitability on the reserve markets. We propose a bi-layer stochastic programming framework for the optimal operation of a fleet of interconnected hydropower plants that sells energy on both the spot and the reserve markets. The outer layer (the \emph{planner's problem}) optimizes end-of-day reservoir filling levels over one year, whereas the inner layer (the \emph{trader's problem}) selects optimal hourly market bids within each day. Using an information restriction whereby the planner prescribes the end-of-day reservoir targets one day in advance, we prove that the trader's problem simplifies from an infinite-dimensional stochastic program with 25 stages to a finite two-stage stochastic program with only two scenarios. Substituting this reformulation back into the outer layer and approximating the reservoir targets by affine decision rules allows us to simplify the planner's problem from an infinite-dimensional stochastic program with 365 stages to a two-stage stochastic program that can conveniently be solved via the sample average approximation. Numerical experiments based on a cascade in the Salzburg region of Austria demonstrate the effectiveness of the suggested framework.}


\maketitle


\section{Introduction}

For many decades, hydropower plants have offered a flexible source of emission-free electricity at low running costs. With a world-wide output of 4,246.4 TWh, hydropower accounted for 60.4\% of all renewable generation in 2019 \citep{IEA20:explorer}, and several new hydropower plants are currently being taken into service in China, Lao People’s Democratic Republic, Portugal and Turkey \citep{IEA20:hydropower}. Pumped-storage plants, which can flexibly transfer water both downstream (through generators, thereby producing energy) and upstream (through pumps, thus consuming energy), account for approximately 95\% of the world-wide grid energy storage \citep{DOE20:storage}.




Traditionally, pumped-storage hydropower plants benefit from releasing the water downstream for electricity generation at peak times and by pumping the water upstream during off-peak periods for future generation (thus adopting a `\emph{buy low and sell high}' strategy). In doing so, the generation companies can exploit the gaps between peak and off-peak electricity prices to make immediate profits. However, these price spreads have been eroding in several European markets since 2008 for two reasons \citep{Fraunhofer}: \emph{(i)} the phaseout of nuclear power plants in Germany 
has increased base load electricity prices, and \emph{(ii)} the rapid growth in photovoltaic capacity, whose supply significantly overlaps with the daytime peak hours on weekdays, has reduced the peak electricity prices \citep{Energytransition, PVfact}. As a result, hydropower producers have increasingly turned towards the reserve markets to maintain their profitability.

Electricity is traded at multiple time frames. Besides the forward markets, which trade electricity weeks or months ahead of time, the majority of electricity supply and demand is settled day-ahead on \emph{spot markets}.
In practice, however, this cannot be achieved without forecasting errors: unpredicted changes in residential and industrial electricity consumption, failures of power plants, transformers and transmission lines as well as the intermittency of renewable energy sources (such as wind and solar power) all imply that market participants may at times be unable to implement their original plans. These errors are resolved on intra-day markets, which are similar to the spot market but operate with shorter lead times, as well as the \emph{reserve markets}. While the details differ slightly by country, reserve markets are usually divided into primary (frequency containment reserve), secondary and tertiary (both frequency restoration) reserves. Both the primary reserve, which is typically activated within 30 seconds, and the secondary reserve, which is typically activated within 5 minutes, consider time spans of about 15 minutes and are activated automatically. The tertiary reserve, on the other hand, is activated manually within 15 minutes, and it considers time spans of up to one (or, in exceptional cases, several) hour(s). Primary reserves are used to swiftly restore minor grid frequency deviations. Larger deviations are managed via the secondary reserve, and---in the case of major, persistent disruptions and only when the secondary reserves fail to restore the balance between demand and supply---the tertiary reserve.
We refer to \cite{Aasgard19}, \cite{EC16:METIS} and \cite{FERNANDEZMUNOZ2020109662} for further details on the different energy markets.

Generation companies benefit from trading in the reserve markets in two ways. First, depending on the market, they may receive capacity fees regardless of whether the reserve capacities are activated. Second, they earn money proportional to the increase or decrease in their production levels when the reserves are activated. Although most thermal plants can provide all three types of reserve when they are online, their ability to do so is limited by their current power output (which must be below capacity to offer upward corrections and above the minimum stable operation limit to offer downward corrections in the energy output) as well as their ramp rate. Gas turbines are the only thermal plants that can offer tertiary reserve when they are offline. With a start-up time of a few minutes and high ramp rates, on the other hand, hydropower plants are an ideal source for any type of reserve, and they can provide secondary and tertiary reserve even when they are offline. Participation on the reserve markets should ease the pressure on the hydropower plant operators, who are struggling to recover their capital costs as well as their original profitability on the forward and spot markets because of the eroding peak/off-peak spreads.

In this paper, we develop a stochastic program that maximizes the revenues earned by a hydropower producer from simultaneously trading in both the spot and the reserve markets. The resulting optimization problem is computationally challenging for at least three reasons. First, the problem involves a large number of decision stages. Indeed, a planning horizon of one year is required to account for the seasonality of rainfall and electricity prices, while generation and pumping decisions need to be taken on an hourly basis. Second, the problem is affected by significant uncertainty in the electricity prices, the timing and amount of the water inflows ({\em i.e.}, rainfall and snowmelt) as well as the reserve activations. Finally, the problem has random recourse since the water flows between the reservoirs not only depend on the bidding decisions in the spot and reserve markets, but they also depend on the reserve activations by the transmission system operator.

To obtain a tractable solution scheme, we equivalently express the stochastic program of the hydropower producer as a bi-layer problem where the outer layer \emph{planner's problem} is a 365-stage stochastic program with daily resolution that determines end-of-day water levels for all reservoirs, and the inner layer \emph{trader's problem} is a 25-stage stochastic program with hourly resolution that uses the target water levels to determine optimal bids for the spot and reserve markets. Using an information restriction whereby the planner prescribes the target water levels one day in advance, we prove that the trader's problem can equivalently be reduced to a two-stage stochastic program with only two scenarios. By substituting this reduced trader's problem back into the planner's problem, applying an exact perfect information relaxation to all hourly bidding decisions and conservatively approximating the daily water level targets via linear decision rules, we obtain a two-stage stochastic program that can be solved efficiently via the sample average approximation.

The main contributions of this paper may be summarized as follows.
\begin{enumerate}
	\item[\emph{(i)}]
	We propose a stochastic program that maximizes the revenues of a hydropower producer who simultaneously trades in the spot and reserve markets and whose fleet of plants can have an arbitrary topology. Our model faithfully accounts for uncertainty in the electricity prices, the water inflows as well as the reserve activations. To our best knowledge, we propose the first model that accounts for the stochasticity of the reserve activations at an hourly granularity.
	\item[\emph{(ii)}]
	Under the assumption that the end-of-day water levels are chosen one day in advance, we develop a novel planner-trader decomposition that allows us to separately study the reservoir management and trading problems. We equivalently reformulate the trader's problem as a tractable linear program, and we approximate the planner's problem through a combination of sample average approximation and linear decision rules. The resulting model can be readily solved within a few minutes on standard hardware, which is in stark contrast to the vast majority of stochastic hydropower models that account for reserve market participation.
	\item[\emph{(iii)}]
	We report initial numerical results for a cascade of three connected reservoirs operating in the control area of the Austrian Power Grid AG. Our results suggest that significant revenues can be earned on the reserve markets if the stochastic reserve activations are modeled faithfully.
\end{enumerate}

The detailed modeling of uncertainty, which is essential to faithfully model the stochasticity of reserve activations, requires us to make a number of simplifying assumptions. First, we assume that the hydropower producer only participates in the spot and the secondary
reserve market. \revision{While we envision that the inclusion of the intra-day market and the primary and tertiary reserve markets is principally possible, they would result in larger optimization models that likely demand substantially longer solution times.} The inclusion of a forward market, on the other hand, appears to be fundamentally more difficult as it would require the consideration of different time scales. \revision{Second, we simplify the bidding process in both the spot and the reserve markets. Instead of directly modeling the possibility to bid at multiple price levels, we study a single-bid strategy that is finetuned out-of-sample by two hyperparameters.} Third, we omit a number of technical aspects, such as the discrete (on/off) nature of equipment operations, head effects and water flow delays. Such considerations often lead to the inclusion of integer variables and/or non-linear relationships between variables, which would significantly complicate the solution of the stochastic program.

The liberalization of energy markets world-wide has sparked significant academic interest in the optimal management of hydropower plants. Early works focus on participation in the spot markets, and they include scenario tree approaches \citep{Andrieu10, Carpentier13:sp, Aasgard14}, stochastic dynamic programming \citep{Pritchard05} and stochastic dual dynamic programming approaches \citep{Pereira91, Gjelsvik10, Philpott12a, Shapiro13, deMatos15, Philpott18} as well as robust and stochastic programs using affine decision rules \citep{Pan15, Lorca16, Delage17, Delage18}.

In departure from the aforementioned single-market studies, \cite{Wozabal13} consider a fleet of hydropower plants that participates in both the spot and the intra-day market. The problem comprises uncertain water inflows and electricity prices. The time horizon of one year is subdivided into inter-day reservoir management problems (at a daily resolution) as well as intra-day bidding problems (at an hourly resolution). The authors solve the problem by a combination of stochastic dual dynamic programming and approximate dynamic programming. The runtimes are reported to range between 4 and 40 hours.

The recent shift of the hydropower producers' attention towards the reserve markets has led to a number of multi-market models that consider both the spot and the reserve markets. Below, we distinguish between deterministic models and stochastic models using scenario trees, stochastic dynamic programming as well as stochastic dual dynamic programming.

In return for disregarding the stochasticity of the problem data, deterministic hydropower models typically offer a superior representation of the physical details of a hydropower plant, such as the consideration of discrete operation modes, head effects as well as a distinction between fixed-speed and variable-speed units. \cite{Chazarra14:opt_op, Chazarra18:econ_viab} propose daily mixed-integer programs for pumped-storage hydropower plants in Spain that model the hourly bidding on the spot and secondary reserve markets. \cite{Schillinger17}
apply a yearly linear programming reservoir management problem with weekly nonlinear bidding subproblems that comprise the spot, intra-day as well as primary, secondary and tertiary reserve markets to a hydropower plant in Switzerland. \cite{Fodstad18} propose a determinstic linear program with a one year time horizon that determines the optimal hourly bidding on the day-ahead, intra-day and reserve markets for a hydropower plant in Norway. Since deterministic models neglect the uncertainty in the water inflows, electricity prices and reserve activations, they adopt a perfect information relaxation that may lead to overly optimistic estimates of the revenues achievable by a hydropower plant.





\cite{Chazzara16} develop a two-stage stochastic program for the short-term scheduling of a fleet of hydropower plants that participates in the spot and the reserve markets. The model comprises a time horizon of one week, where the decisions of the first day are taken in the first stage and the decisions of the remaining six days are taken in the second stage, respectively. \cite{Chazarra18} propose a single-stage stochastic program for a hydropower plant that participates in the spot and secondary reserve market. The authors subdivide the planning horizon into daily problems with hourly resolution. The daily problems optimize the conditional value-at-risk of the revenues over 100 scenarios, subject to uncertainty in the spot and reserve prices. The runtime of the daily model varies between 3 and 17 minutes. 
\cite{Klaeboe19}, finally, consider a three-stage stochastic program where the hydropower producer submits the spot and reserve bids in the first and third stage, respectively, while the spot market prices and reserve market prices are observed in the second and third stage, respectively. The authors report computation times between 2 and 7 hours. While scenario-tree based stochastic programs allow to capture complex relationships between the uncertain problem parameters, the scenario discretization typically fails to provide an implementable policy, and the problem size tends to \revision{grow} exponentially with the considered time horizon. As a result, most studies consider a small number of stages and model the stochastic price and reserve activation processes at a coarser (\emph{e.g.}, daily) granularity.

\cite{abgottspon12} propose a stochastic dynamic program for a hydropower plant with a single reservoir that participates in the spot and secondary reserve markets. The authors consider a yearly problem whose weekly time stages are connected through a single state variable describing the water level of the reservoir. Each weekly problem is modelled as a two-stage stochastic program whose first stage decides upon the reserve market offering under uncertain prices and whose second stage manages the water flows. The authors report that the entire model can be solved within 20 minutes via large-scale parallelization. Similar to \cite{abgottspon12}, \cite{Helseth17} study a hydropower plant with a single reservoir that participates in the spot and secondary reserve markets. The authors also consider a yearly problem with weekly time periods, however the state variables now comprise the weekly average energy prices in addition to the reservoir level. The authors report computation times of 24 to 40 hours. While stochastic dynamic programs allow to model complex non-convex phenomena in the weekly subproblems, the curse of dimensionality implies that the runtimes increase quickly with the dimension of the state space, which makes the stochastic dynamic programs primarily suited for small fleets of plants with a coarse description of the involved stochastic processes. Moreover, the need to discretize the water levels can introduce numerical inaccuracies.



\cite{Abgottspon14} compare the performance of a stochastic dynamic and a stochastic dual dynamic program that both manage a hydropower plant with a seasonal and a daily reservoir over a yearly time horizon with weekly time periods. The weekly subproblems are modeled as two-stage stochastic programs that decide upon the spot and reserve market participation in an hourly resolution under uncertain water inflows and energy prices. The overall models record the water level of the seasonal reservoir in the state variable and optimize a nested conditional value-at-risk. Computation times are not reported. \cite{Helseth15} propose a stochastic dual dynamic program that manages a multi-reservoir fleet of hydro plants over the course of one year in weekly time stages. The state variables of the problem record the weekly water levels of the reservoirs as well as the spot prices, while the weekly subproblems manage the commitments on the spot and the primary reserve markets. The overall problem is solved in 42 hours. In a follow-up paper, \cite{Helseth16} modify the stochastic dual dynamic program of \cite{Helseth15} to incorporate stochastic reserve prices. The resulting model can be solved within 55-60 hours. While stochastic dual dynamic programs scale better than stochastic dynamic programs, they too tend to be limited to small and medium-sized fleets that are managed over coarse (typically weekly) time intervals with relatively crude approximations of the involved stochastic processes. This is particularly critical for the reserve market activations in small (\emph{e.g.}, daily) reservoirs, where a daily or weekly time granularity may be too coarse to ensure an implementable (hourly) policy.

\revision{
The remainder of the paper unfolds as follows. Section~\ref{sec:models} models the overall bidding problem of the hydropower producer. Section~\ref{sec:bi-layer} transforms this formulation to an equivalent bi-layer stochastic program. Section \ref{sec:reduction_collective} reformulates the inner layer of this stochastic program as a tractable linear program. Section~\ref{sec:num_exp_formulation} presents a tractable approximation for the outer layer of the stochastic program as well as a rolling horizon solution framework. Section~\ref{sec:case_study} applies our method to a hydropower cascade in Austria. Auxiliary results and all proofs, finally, are relegated to the appendix.
}

\noindent \textbf{Notation.} We denote by $\bm{0}$ and $\bm{1}$ the appropriately sized vectors of all zeros and all ones, respectively. The Hadamard product is denoted by $``\circ "$ and refers to the element-wise vector multiplication. For any $x \in \mathbb{R}$, we set $x^+ = \max \{x, 0 \}$ and $x^- = \max \{-x, 0 \}$ such that $x=x^+-x^-$. All random objects are defined on a probability space $(\Omega, \mathcal{F},\mathbb{P})$ consisting of a sample space $\Omega$, a $\sigma$-algebra $\mathcal{F}\subseteq 2^\Omega$ of events and a probability measure $\mathbb{P}$ on $\mathcal{F}$. For a $\sigma$-algebra $\mathcal{G} \subseteq \mathcal{F}$, we denote by $\mathcal{L}^k(\mathcal{G})$ the set of all integrable, $\mathcal{G}$-measurable functions $g : \Omega \to \mathbb{R}^k$, and if $k=1$, we simply write~$\mathcal{L}(\mathcal{G})$.

\revision{\section{Bidding Model}\label{sec:models}}

We consider a hydropower generation company that operates a cascade of reservoirs over a time horizon spanning at least one year to account for the seasonality of electricity prices and water inflows. We partition the planning horizon into days indexed by $d \in \mathcal{D} : = \{ 1,\ldots,D\}$, as well as into hours indexed by $t \in \mathcal{T} : = \{ 1,\ldots,T\}$. To simplify the exposition, we assume that $T$ is divisible by $D$, that is, the planning horizon accommodates an integral number of days. We further denote by $H=T/D$ the number of hours per day, but we emphasize that all subsequent results remain valid for any~$H\in\mathbb{N}$, {\em e.g.}, if we partition each day into 15~minute intervals. The present time is modeled by a fictitious hour $t = 0$ on day $d = 0$. It is followed by hour $t=1$ on day $d=1$, that is, the first hour of the planning horizon. No decisions are taken at time $t = 0$; its inclusion merely serves to unify our exposition so that decisions and constraints at time $t = 1$ can be expressed similar to those in any other time period. For any $t,t' \in \{ 0 \} \cup \mathcal{T}$ with $t \leq t'$, we denote the set of hours between and including $t$ and $t'$ as $[t, t'] = \{t,t+1,\ldots,t' \}$. Note in particular that $[t,t] = \{ t \}$. Furthermore, we denote by $d(t)$ the day containing hour $t$ and by $\mathcal{T}(d)$ the set of all hours in day $d$. Finally, we introduce the following functions of $d \in \mathcal{D}$ and $t \in \mathcal{T}$:
\begin{equation*}
\begin{aligned}
&\uparrow(d) = \text{ last hour of day $d$,} \qquad &&\uparrow(t) = \text{ last hour of day $d(t)$,}\\
&\downarrow(d) = \text{ first hour of day $d$,} \qquad &&\downarrow(t) = \text{ first hour of day $d(t)$,}\\
&\Downarrow(d) = \text{ last hour of day $d-1$,} \qquad &&\Downarrow(t) = \text{ last hour of day $d(t)-1$}.
\end{aligned}
\end{equation*}
By slight abuse of notation, the operators $\uparrow(\cdot)$, $\downarrow(\cdot)$ and $\Downarrow(\cdot)$ are defined both on $\mathcal D$ and $\mathcal T$. The correct interpretation of these operators will always be clear from the context. Observe also that the two $\Downarrow(\cdot)$ operators are not strictly needed because they can be expressed in terms of $\downarrow(\cdot)$ via the identities $\Downarrow(d) = \, \uparrow(d-1)$ and $\Downarrow(t) = \, \uparrow(d(t)-1)$. However, they are useful to simplify notation.

We represent the topology of the interconnected reservoirs by a directed acyclic graph with a set of nodes $\mathcal{R}$ representing the reservoirs and a set of arcs $\mathcal{A} \subseteq \mathcal{R} \times \mathcal{R}$ representing the hydraulic connections between the reservoirs. We denote the cardinalities of $\mathcal{R}$ and $\mathcal{A}$ by $R$ and $A$, respectively. A tuple $(r, r')$ is an arc in $\mathcal{A}$ if $r$ is an upstream reservoir of $r'$ (and hence $r'$ is a downstream reservoir of $r$). Without loss of generality, we assume that each arc is equipped with a generator that converts kinetic energy of water flowing downstream into electric energy and with a pump that uses electric energy to lift water upstream. Note that some of these generators and pumps may have zero capacity, which models situations in which these devices are absent. Without loss of generality, we further assume that the node $R$ is the unique sink of the graph. This dummy reservoir has infinite volume and models the discharge into a river at the end of the reservoir cascade. Accordingly, arcs discharging into this dummy reservoir may be equipped with generators, but not with pumps. In line with real hydropower plants, we assume that every reservoir (except the aforementioned dummy reservoir) has at least one outgoing arc along which water can be discharged whenever required. The topology of the cascade can be encoded conveniently through an $R\times A$ incidence matrix $\mathbf{M}$, where for each $(r,a) \in \mathcal{R} \times \mathcal{A}$ we have that
\begin{equation*}
[\mathbf{M}]_{r,a} = m_{r,a} = \left\{ \begin{array}{cl} -1 & \text{ if arc $a$ leaves reservoir $r$,} \\[2mm] +1 & \text{ if arc $a$ enters reservoir $r$,} \\[2mm] 0 & \text{ otherwise.} \end{array} \right.
\end{equation*}

Each reservoir is characterized by---potentially time-dependent---lower and upper bounds on its filling level. The time dependence of these bounds may reflect changing safety margins that are imposed to account for the seasonality of inflows, scheduled maintenance work or environmental regulations. For each hour $t \in \mathcal{T}$, we collect the lower and upper reservoir bounds in the vectors $\underline{\bm{w}}_t \in \mathbb{R}_+^R$~$[\text{m}^3]$ and $\overline{\bm{w}}_t \in (\mathbb{R}_+\cup\{+\infty\})^R$~$[\text{m}^3]$, respectively. The bounds for the dummy reservoir are $\underline{w}_{t,R} =0$ and $\overline{w}_{t,R} = +\infty$, \revision{that is, the river can absorb arbitrary amounts of water.} The generators installed along the arcs of our cascade are characterized by the upper bounds $\overline{\bm{g}}_t \in \mathbb{R}_+^A$~$[\text{m}^3]$ on the hourly water flows and the generator efficiencies $\bm{\eta}_t \in \mathbb{R}_+^A$~$[\text{MWh}/\text{m}^3]$. Similarly, the pumps are described by the upper bounds $\overline{\bm{p}}_t \in \mathbb{R}_+^A$~$[\text{m}^3]$ on the hourly water flows and the 
\emph{inverse} pump efficiencies $\bm{\zeta}_t \in \mathbb{R}_+^A$~$[\text{MWh}/\text{m}^3]$. The laws of thermodynamics dictate that $\bm{\eta}_t <\bm{\zeta}_t$ for all $t \in \mathcal{T}$.

The decision problem of the hydropower generation company is affected by three types of risk factors with finite expectations, namely fluctuating market prices, \revision{uncertain reserve allocations and reserve activations} as well as natural inflows into the reservoirs caused by unpredictable meteorological phenomena. \revision{Below, we first describe the uncertainties related to the different markets.

We assume that the hydropower pruducer participates in the spot market and in the secondary reserve markets for up- and down-regulation. In central Europe, allocations and prices on these markets are determined by separate auctions for the different hours of the day (consecutive hours could be merged to blocks). The spot market is mediated via a sealed-bid pay-as-clear auction, where demand and supply bids correspond to tuples comprising an energy quantity to be sold or bought and a price. Bids are accepted based on the spot price, which is found by matching demand and supply. The secondary reserve markets, on the other hand, are mediated via open pay-as-bid auctions, where bids correspond to tuples comprising an energy quantity offered as reserve capacity, a capacity price and an activation price. Bids are accepted in the order of their capacity prices until the regulation demand is covered. If grid imbalances require an intervention, the accepted bids are activated in the order of their energy prices until the imbalances are eliminated. This means that bids with low energy prices are more likely to be activated than bids with high energy prices.

A faithful representation of the different auction mechanisms would lead to an excessively complicated model. In order to obtain a realistic model that is amenable to numerical solution, we use the following approximations. \revision{First, we assume that the spot and reserve markets are infinitely liquid. Second,} we model the spot price in any hour~$t$ as an exogenous random variable $\pi_t^{\rm{s}} \in \mathbb{R}$~$[\$/\text{MWh}]$ that is revealed one day in advance at time~$\Downarrow(t)$, and we assume that the energy quantities traded in hour~$t$ on the spot market may adapt to~$\pi_t^{\rm{s}}$. This approximation is reasonable if the producer has no market power and places a continuum of bids corresponding to all possible spot price realizations. \revision{Third,} we model the capacity prices for up- and down-regulation in any hour~$t$ as exogenous random variables~$\pi_t^{\rm{u}} \in \mathbb{R}$~$[\$/\text{MWh}]$ and $\pi_t^{\rm{v}} \in \mathbb{R}$~$[\$/\text{MWh}]$, respectively, which are revealed one day in advance at time~$\Downarrow(t)$, and we assume that the energy quantities offered as reserve capacity in hour~$t$ may adapt to~$\pi_t^{\rm{u}}$ and~$\pi_t^{\rm{v}}$. By slight abuse of terminology, we will sometimes refer to these energy quantities simply as `bids.' We also assume that the allocation of reserve capacity in hour~$t$ is random and can be modeled through exogenous Bernoulli random variables~$\kappa^{\rm u}_t$ and~$\kappa^{\rm v}_t$ with success probabilities~$\hat\kappa^{\rm u}_t=\mathbb E[\kappa^{\rm u}_t]$ and~$\hat\kappa^{\rm v}_t=\mathbb E[\kappa^{\rm v}_t]$, respectively, which are revealed at time~$\downarrow(t)$. The allocated reserve capacity for up- or down-regulation is then obtained by multiplying the offered reserve capacity with the corresponding Bernoulli random variable. This approximation is reasonable if the producer has no market power and places only one bid on the reserve-up market with capacity price~$\pi_t^{\rm{u}}$ and only one bid in the reserve-down market with capacity price~$\pi_t^{\rm{v}}$. In reality, the prices~$\pi_t^{\rm{u}}$ and~$\pi_t^{\rm{v}}$ constitute decisions that have a diminishing impact on the allocation probabilities~$\hat\kappa^{\rm u}_t$ and~$\hat\kappa^{\rm v}_t$, respectively. Treating these quantities as exogenous parameters amounts to prescribing a fixed trade-off between the probability that a given bid is accepted and the financial compensation that is earned if the bid is accepted. In the model to be developed below we will optimize this trade-off via a procedure reminiscent of hyperparameter tuning. As the \revision{fourth} and last approximation, we model the activation prices for up- and down-regulation as exogenous random variables~$\psi_t^{\rm{u}} \in \mathbb{R}$~$[\$/\text{MWh}]$ and $\psi_t^{\rm{v}} \in \mathbb{R}$~$[\$/\text{MWh}]$, respectively, which are revealed in real time. In addition, we model the percentages of time in which the system operator activates up- and down-regulation in hour~$t$ as continuous random variables~$\rho_t^{\rm{u}}$ and~$\rho_t^{\rm{v}}$ with supports~$[0,\overline{\rho}_t^{\rm{u}}]$ and~$[0,\overline{\rho}_t^{\rm{v}}]$ and with mean values~$\hat \rho_t^{\rm{u}}=\mathbb E[\rho_t^{\rm{u}}]$ and~$\hat \rho_t^{\rm{v}}=\mathbb E[\rho_t^{\rm{v}}]$, respectively. The activations~$\rho_t^{\rm{u}}$ and~$\rho_t^{\rm{v}}$ are also revealed in real time. The amount of energy activated for up- and down-regulation in hour~$t$ is obtained by multiplying the allocated reserve capacity with~$\rho_t^{\rm{u}}$ and~$\rho_t^{\rm{v}}$, respectively. This approximation is again reasonable if the producer has no market power and places only one bid on the reserve-up market with activation price~$\psi^{\rm u}_t$ and only one bid in the reserve-down market with activation price~$\psi^{\rm v}_t$. In reality, $\psi^{\rm u}_t$ and~$\psi^{\rm v}_t$ constitute decisions that have a diminishing impact on the expected activations~$\hat \rho_t^{\rm{u}}$ and~$\hat \rho_t^{\rm{v}}$, respectively. Treating these quantities as exogenous parameters amounts to prescribing a fixed trade-off between the probability that an accepted reserve bid is activated and the financial compensation that is earned if the bid is activated. In the model to be developed below we will optimize this trade-off via a procedure reminiscent of hyperparameter tuning.
%
%
Note that, as any bidding zone may either face an over-supply or an under-supply of energy but not both at the same time, the sum of~$\rho_t^{\rm{u}}$ and~$\rho_t^{\rm{v}}$ cannot exceed~$1$. Throughout the paper we thus assume that the pairs of reserve activations on the up- and down-regulation markets have support $\{(\rho_t^{\rm{u}}, \rho_t^{\rm{v}})\in [0,\overline{\rho}^{\rm u}_t]\times [0,\overline{\rho}^{\rm v}_t] : \rho_t^{\rm{u}} + \rho_t^{\rm{v}} \leq 1\}$, \revision{where the upper bounds $\overline{\rho}^{\rm u}_t$ and $\overline{\rho}^{\rm v}_t$ can be set, for example, to quantiles of the historical hourly activation percentages.} \revision{We also assume that the reserve activations are serially independent as well as independent of all other exogenous uncertainties. This assumption may only hold approximately in practice, and we discuss its validity in Appendix~\ref{sec:independence}.}}


\revision{Finally, we denote the natural inflows into the various reservoirs by the vector $\bm{\phi}_t \in \mathbb{R}_+^R$~$[\text{m}^3]$, and we assume that they are revealed a day in advance at time~$\Downarrow(t)$ thanks to an accurate forecast.
%
%
From now on we gather all random variables that are revealed in hour $t \in \{ 0 \} \cup \mathcal{T}$ in the vector}
\begin{equation*}
\bm{\xi}_t = \left\{ \begin{array}{ll} ( \{ (\pi_{\tau}^{\rm{s}}, \pi_{\tau}^{\rm{u}}, \pi_{\tau}^{\rm{v}}, \bm{\phi}_{\tau} )\}_{\tau = 1}^{H}) & \text{ if $t = 0$,} \\[2mm] 
(\psi_t^{\rm{u}}, \psi_t^{\rm{v}}, \rho_t^{\rm{u}}, \rho_t^{\rm{v}}, \{ ( \pi_{\tau}^{\rm{s}}, \pi_{\tau}^{\rm{u}}, \pi_{\tau}^{\rm{v}}, \bm{\phi}_{\tau} ) \}_{\tau = t+1}^{t+H} ) & \text{ if $t \in \Big[ \bigcup_{\tau \in \mathcal{T}} \Downarrow(\tau) \Big] \setminus \{ 0 \}$,} 
 \\[2mm] 
\revision{ (\psi_t^{\rm{u}}, \psi_t^{\rm{v}}, \rho_t^{\rm{u}}, \rho_t^{\rm{v}}, \{ (\kappa^{\rm u}_\tau, \kappa^{\rm v}_\tau) \}_{\tau = t}^{t+H-1} )} & \revision{\text{ if $t \in \Big[ \bigcup_{\tau \in \mathcal{T}} \downarrow(\tau) \Big]$,} }\\[2mm]
 (\psi_t^{\rm{u}}, \psi_t^{\rm{v}}, \rho_t^{\rm{u}}, \rho_t^{\rm{v}}) & \text{ otherwise.} \end{array} \right.
\end{equation*}
\revision{In other words, the spot and capacity prices as well as the inflows of day $d \in \mathcal{D}$ are revealed in the last hour of the previous day, the reserve allocations of day~$d$ are revealed in the first hour of the current day, whereas the activation prices and actual activations are revealed hourly in real time.} We gather all random variables whose values are revealed within the interval $[t,t']$ in the vector $\bm{\xi}_{[t,t']} = ( \bm{\xi}_{t}, \ldots, \bm{\xi}_{t'} )$. Furthermore, we capture the information revealed during this time interval by the $\sigma$-algebra $\mathcal{F}_{[t,t']} = \sigma(\bm{\xi}_{[t,t']})$, \emph{i.e.}, the $\sigma$-algebra generated by $\bm{\xi}_{[t,t']}$. For ease of notation, we abbreviate $\bm{\xi}_{[0, t']}$ and $\mathcal{F}_{[0, t']}$ by $\bm{\xi}_{[t']}$ and $\mathcal{F}_{[t']}$, respectively.

The hydropower generation company must decide each day how much energy to offer on the different markets in every hour of the following day. This in turn necessitates a strategy for operating the reservoir system that is guaranteed to honor all market commitments. 
We denote the bids placed on the spot market, the reserve-up market and the reserve-down market by $s_t, u_t, v_t \in \mathcal{L}(\mathcal{F}_{[\Downarrow(t)]})$~$[\text{MWh}]$, respectively. \revision{Our model involves} the water flows along the arcs as well as the reservoir filling levels as real-time operational decisions. Specifically, we denote the generation flows, \emph{i.e.}, the amounts of water running through the turbines in each hour, by $\bm{g}_t \in \mathcal{L}^A(\mathcal{F}_{[t]})$~$[\text{m}^3]$ and the consumption flows, \emph{i.e.}, the amounts of water running through the pumps in each hour, by $\bm{p}_t \in \mathcal{L}^A(\mathcal{F}_{[t]})$~$[\text{m}^3]$. \revision{Furthermore, we ignore concerns of flooding related to the dummy reservoir $R$ and assume that one can spill unlimited amounts of water $\bm{z}_t \in \mathcal{L}^A(\mathcal{F}_{[t]})$~$[\text{m}^3]$ along all arcs from the upstream reservoirs to the downstream reservoirs/rivers.} The reservoir filling levels at the end of each hour are denoted by $\bm{w}_t \in \mathcal{L}^R(\mathcal{F}_{[t]})$~$[\text{m}^3]$. The initial filling levels $\bm{w}_0 \in \mathbb{R}^R$~$[\text{m}^3]$ are deterministic input parameters of our model. 
We assume that the hydropower generation company is risk-neutral but aims to satisfy all constraints robustly---especially those related to reserve market commitments. This appears justified as non-compliance would incur high penalty fees or even lead to an exclusion from the reserve markets. Thus, the company maximizes the expected cumulative revenues across the entire planning horizon while maintaining its ability to respond to all possible call-offs on the reserve markets. The objective function thus takes the form
\begin{equation*}
\sum_{t \in \mathcal{T}} \mathbb{E}\big[ \pi_t^{\rm{s}} s_t + \revision{\kappa^{\rm u}_t}(\pi_t^{\rm{u}} + \rho_t^{\rm{u}} \psi_t^{\rm{u}}) u_t + \revision{\kappa^{\rm v}_t}(\pi_t^{\rm{v}} + \rho_t^{\rm{v}} \psi_t^{\rm{v}}) v_t \big],
\end{equation*}
and the non-anticipativity constraints to be respected are
\begin{equation*}
s_t, u_t, v_t \in \mathcal{L}(\mathcal{F}_{[\Downarrow(t)]}), ~\bm{g}_t, \bm{p}_t, \bm{z}_t \in \mathcal{L}^A(\mathcal{F}_{[t]}), ~\bm{w}_t \in \mathcal{L}^R(\mathcal{F}_{[t]}) \qquad \forall t \in \mathcal{T}.
\end{equation*}
While spot market bids have no sign restrictions (\emph{i.e.}, the company can both buy or sell electricity in the spot market), bids in the reserve markets as well as the operational flow decisions must be non-negative. Also, the flow decisions must obey the capacity limits of the respective generators and pumps, and the reservoir levels must stay within the respective bounds. Thus, we require
\begin{equation*}
{0} \leq {u}_t, ~{0} \leq {v}_t, ~\bm{0} \leq \bm{g}_t \leq \overline{\bm{g}}_t, ~\bm{0} \leq \bm{p}_t \leq \overline{\bm{p}}_t, ~\bm{0} \leq \bm{z}_t,~\underline{\bm{w}}_t \leq \bm{w}_t \leq \overline{\bm{w}}_t \qquad \forall t \in \mathcal{T},~\text{$\mathbb{P}$-a.s.}
\end{equation*}
In addition, the hydropower company is obliged to produce and/or consume energy according to its market commitments. If the same system operator is responsible for both the spot and the reserve markets, then the market commitments are enforced through the following constraint
\begin{equation*}
{s}_t + \revision{\kappa^{\rm u}_t}\rho_t^{\rm{u}} {u}_t - \revision{\kappa^{\rm v}_t}\rho_t^{\rm{v}} {v}_t = \bm{\eta}_t^\top \bm{g}_t - \bm{\zeta}_t^\top \bm{p}_t \qquad \forall t \in \mathcal{T},~\text{$\mathbb{P}$-a.s.}
\end{equation*}
Taking into account the natural inflows, the reservoir filling levels obey the dynamic equation
\begin{equation*}
\bm{w}_t = \bm{w}_{t-1} + \bm{\phi}_t + \mathbf{M}(\bm{g}_t - \bm{p}_t + \bm{z}_t) \qquad \forall t \in \mathcal{T},~\text{$\mathbb{P}$-a.s.}
\end{equation*}
\revision{In summary, the company aims to solve the stochastic optimization problem}
\begin{equation*}
\label{opt:collective_bidding_model}
\tag{C}
\begin{aligned}
    &\text{sup} && \sum_{t \in \mathcal{T}} \mathbb{E}\big[ \pi_t^{\rm{s}} s_t + \revision{\kappa^{\rm u}_t}(\pi_t^{\rm{u}} + \rho_t^{\rm{u}} \psi_t^{\rm{u}}) u_t + \revision{\kappa^{\rm v}_t}(\pi_t^{\rm{v}} + \rho_t^{\rm{v}} \psi_t^{\rm{v}}) v_t \big]\\
    &\text{s.t.} && s_t, u_t, v_t \in \mathcal{L}(\mathcal{F}_{[\Downarrow(t)]}), ~\bm{g}_t, \bm{p}_t, \bm{z}_t \in \mathcal{L}^A(\mathcal{F}_{[t]}), ~\bm{w}_t \in \mathcal{L}^R(\mathcal{F}_{[t]}) && \forall t \in \mathcal{T}\\
    &	  && 0 \leq u_t, ~0 \leq v_t, ~\bm{0} \leq \bm{g}_t \leq \overline{\bm{g}}_t, ~\bm{0} \leq \bm{p}_t \leq \overline{\bm{p}}_t, ~\bm{0} \leq \bm{z}_t && \forall t \in \mathcal{T},~\text{$\mathbb{P}$-a.s.}\\
    &	  && s_t + \revision{\kappa^{\rm u}_t}\rho_t^{\rm{u}} u_t - \revision{\kappa^{\rm v}_t}\rho_t^{\rm{v}} v_t = \bm{\eta}_t^\top \bm{g}_t - \bm{\zeta}_t^\top \bm{p}_t && \forall t \in \mathcal{T},~\text{$\mathbb{P}$-a.s.}\\
    &	  && \bm{w}_t = \bm{w}_{t-1} + \bm{\phi}_t + \mathbf{M}(\bm{g}_t - \bm{p}_t + \bm{z}_t) && \forall t \in \mathcal{T},~\text{$\mathbb{P}$-a.s.}\\
    &	  && \underline{\bm{w}}_t \leq \bm{w}_t \leq \overline{\bm{w}}_t && \forall t \in \mathcal{T},~\text{$\mathbb{P}$-a.s.}
\end{aligned}
\end{equation*}

\section{Planner-Trader Decomposition}
\label{sec:bi-layer}
%
%


In this section we will demonstrate that problem~\eqref{opt:collective_bidding_model} can be decomposed into subproblems that lend themselves to further simplification. 
\revision{Specifically, we first express~\eqref{opt:collective_bidding_model} as a bi-layer stochastic program}, where the outer layer (the {\em planner's problem}) optimizes over the reservoir filling levels with a daily granularity, \revision{whereas} the inner layer (the {\em trader's problem}) optimizes over the bidding decisions~with an hourly granularity. To this end, recall that the decision variables $\{ \bm{w}_t \}_{t \in \mathcal{T}}$ in problem~\eqref{opt:collective_bidding_model} represent the end-of-hour filling levels of the reservoirs. In order to reformulate~\eqref{opt:collective_bidding_model}, we will overload notation and denote by $\bm{w}_d = \bm{w}_{\uparrow(d)}$, $\underline{\bm{w}}_d = \bm{\underline{w}}_{\uparrow(d)}$ and $\bm{\overline{w}}_d = \overline{\bm{w}}_{\uparrow(d)}$ the end-of-day reservoir filling levels and their lower and upper bounds for day $d \in \mathcal{D}$, respectively. Next, we use the reservoir balance constraints to re-express all end-of-hour filling levels~$\{ \bm{w}_t \}_{t \in \mathcal{T}}$ in terms of the end-of-day filling levels~$\{ \bm{w}_d \}_{d \in \mathcal{D}}$, the hourly flow decisions~$\{\bm{g}_t, \bm{p}_t, \bm{z}_t\}_{t\in\mathcal{T}}$ and the natural inflows~$\{\bm{\phi}_t\}_{t\in\mathcal{T}}$: 
\begin{equation*}
	\bm{w}_t = \bm{w}_{d-1} + \sum_{\tau=\downarrow(d)}^t \bm{\phi}_\tau + \mathbf{M}(\bm{g}_\tau - \bm{p}_\tau + \bm{z}_\tau)
	\qquad \forall d \in \mathcal{D}, \; \forall t \in \mathcal{T}(d).
\end{equation*}
Using this relation, the bidding model~\eqref{opt:collective_bidding_model} can be recast equivalently as
\begin{equation}
\label{opt:collective_bidding_model_dailywater}
\begin{aligned}
    &\text{sup} && \displaystyle \sum_{t \in \mathcal{T}} \mathbb{E}\left[ \pi_t^{\rm{s}} s_t + \revision{\kappa^{\rm u}_t}(\pi_t^{\rm{u}} + \rho_t^{\rm{u}} \psi_t^{\rm{u}}) u_t + \revision{\kappa^{\rm v}_t}(\pi_t^{\rm{v}} + \rho_t^{\rm{v}} \psi_t^{\rm{v}}) v_t \right]\\
    &\text{s.t.} && s_t, u_t, v_t \in \mathcal{L}(\mathcal{F}_{[\Downarrow(t)]}), ~\bm{g}_t, \bm{p}_t, \bm{z}_t \in \mathcal{L}^A(\mathcal{F}_{[t]}), ~\bm{w}_d \in \mathcal{L}^R(\mathcal{F}_{[\uparrow(d)]}) && \forall d \in \mathcal{D}, ~\forall t \in \mathcal{T}(d)\\
    &	  && 0 \leq u_t, ~0 \leq v_t, ~\bm{0} \leq \bm{g}_t \leq \overline{\bm{g}}_t, ~\bm{0} \leq \bm{p}_t \leq \overline{\bm{p}}_t, ~\bm{0} \leq \bm{z}_t && \hspace{0mm} \forall d \in \mathcal{D}, ~\forall t \in \mathcal{T}(d),~\text{$\mathbb{P}$-a.s.}\\
    &	  && s_t + \revision{\kappa^{\rm u}_t}\rho_t^{\rm{u}} u_t- \revision{\kappa^{\rm v}_t}\rho_t^{\rm{v}} v_t = \bm{\eta}_t^\top \bm{g}_t - \bm{\zeta}_t^\top \bm{p}_t && \hspace{0mm} \forall d \in \mathcal{D}, ~\forall t \in \mathcal{T}(d),~\text{$\mathbb{P}$-a.s.}\\
    &	  && \underline{\bm{w}}_t  \leq \bm{w}_{d-1} + \sum_{\tau=\downarrow(d)}^t \bm{\phi}_\tau + \mathbf{M}(\bm{g}_\tau - \bm{p}_\tau + \bm{z}_\tau) \leq \overline{\bm{w}}_t && \hspace{0mm} \forall d \in \mathcal{D}, ~\forall t \in \mathcal{T}(d),~\text{$\mathbb{P}$-a.s.}\\
    &	  && \bm{w}_d = \bm{w}_{d-1} + \sum_{\tau \in \mathcal{T}(d)} \bm{\phi}_\tau + \mathbf{M}(\bm{g}_\tau - \bm{p}_\tau + \bm{z}_\tau) && \hspace{0mm} \forall d \in \mathcal{D},~\text{$\mathbb{P}$-a.s.}\\
    &	  && \underline{\bm{w}}_d \leq \bm{w}_d \leq \overline{\bm{w}}_d && \forall d \in \mathcal{D},~\hspace{0mm} \text{$\mathbb{P}$-a.s.},
\end{aligned}
\end{equation}
Note that the end-of-day reservoir bounds in the last line are implied by the end-of-hour reservoir bounds two lines above and are thus redundant. However, they will be instrumental for proving the tightness of a relaxation of problem~\eqref{opt:collective_bidding_model_dailywater} to be derived below. Next, we introduce another auxiliary problem that imposes the daily reservoir balance constraints as inequalities rather than equalities.
\begin{equation}
\label{opt:collective_bidding_model_dailywater2}
\begin{aligned}
    &\text{sup} && \sum_{t \in \mathcal{T}} \mathbb{E}\left[ \pi_t^{\rm{s}} s_t + \revision{\kappa^{\rm u}_t}(\pi_t^{\rm{u}} + \rho_t^{\rm{u}} \psi_t^{\rm{u}}) u_t + \revision{\kappa^{\rm v}_t}(\pi_t^{\rm{v}} + \rho_t^{\rm{v}} \psi_t^{\rm{v}}) v_t \right]\\
    &\text{s.t.} && s_t, u_t, v_t \in \mathcal{L}(\mathcal{F}_{[\Downarrow(t)]}), ~\bm{g}_t, \bm{p}_t, \bm{z}_t \in \mathcal{L}^A(\mathcal{F}_{[t]}), ~\bm{w}_d \in \mathcal{L}^R(\mathcal{F}_{[\uparrow(d)]}) && \forall d \in \mathcal{D}, ~\forall t \in \mathcal{T}(d)\\
    &	  && 0 \leq u_t, ~0 \leq v_t, ~\bm{0} \leq \bm{g}_t \leq \overline{\bm{g}}_t, ~\bm{0} \leq \bm{p}_t \leq \overline{\bm{p}}_t, ~\bm{0} \leq \bm{z}_t && \hspace{0mm} \forall d \in \mathcal{D}, ~\forall t \in \mathcal{T}(d),~\text{$\mathbb{P}$-a.s.}\\
    &	  && s_t + \revision{\kappa^{\rm u}_t}\rho_t^{\rm{u}} u_t - \revision{\kappa^{\rm v}_t}\rho_t^{\rm{v}} v_t = \bm{\eta}_t^\top \bm{g}_t - \bm{\zeta}_t^\top \bm{p}_t && \hspace{0mm} \forall d \in \mathcal{D}, ~\forall t \in \mathcal{T}(d),~\text{$\mathbb{P}$-a.s.}\\
    &	  && \underline{\bm{w}}_t  \leq \bm{w}_{d-1} + \sum_{\tau=\downarrow(d)}^t \bm{\phi}_\tau + \mathbf{M}(\bm{g}_\tau - \bm{p}_\tau + \bm{z}_\tau) \leq \overline{\bm{w}}_t && \hspace{0mm} \forall d \in \mathcal{D}, ~\forall t \in \mathcal{T}(d),~\text{$\mathbb{P}$-a.s.}\\
    &	  && \bm{w}_d \leq \bm{w}_{d-1} + \sum_{\tau \in \mathcal{T}(d)} \bm{\phi}_\tau + \mathbf{M}(\bm{g}_\tau - \bm{p}_\tau + \bm{z}_\tau) && \hspace{0mm} \forall d \in \mathcal{D},~\text{$\mathbb{P}$-a.s.}\\
    &	  && \underline{\bm{w}}_d \leq \bm{w}_d \leq \overline{\bm{w}}_d && \forall d \in \mathcal{D},~\hspace{0mm} \text{$\mathbb{P}$-a.s.}\\
\end{aligned}
\end{equation}
The next proposition asserts that the optimization problems \eqref{opt:collective_bidding_model_dailywater} and \eqref{opt:collective_bidding_model_dailywater2} are indeed equivalent. \revision{This proposition relies on the observation that any slacks between the left- and right-hand sides of the daily reservoir balance constraints can systematically be eliminated by suitable spillages, which are unrestricted in magnitude by our earlier assumptions.} 


\begin{proposition}
\label{prop:collective_inequality_equivalence}
The optimal values of problems \eqref{opt:collective_bidding_model_dailywater} and \eqref{opt:collective_bidding_model_dailywater2} are equal.
\end{proposition}

We are now ready to decompose the bidding problem~\eqref{opt:collective_bidding_model}, which jointly maximizes over all hourly and daily decisions, into planning and trading subproblems that maximize only over the daily and the hourly decisions, respectively. The proposed decomposition assumes that a fictitious planner and a fictitious trader collaborate to solve problem~\eqref{opt:collective_bidding_model} in the following manner. The planner determines the end-of-day reservoir levels by solving the stochastic program 
\begin{equation*}
\label{opt:collective_bidding_model_dailywater_sep}
\tag{CP}
\begin{array}{c@{~~~}l@{\,}l}
    \text{sup} & \displaystyle \sum_{d \in \mathcal{D}}  \mathbb{E} \big[ \Pi_d( \bm w_{d-1}, \bm w_{d}, & \bm\xi_{[\Downarrow(d)]}) \big] \\[1ex]
    \text{s.t.} & \bm{w}_d \in \mathcal{L}^R(\mathcal{F}_{[\uparrow(d)]}) & \forall d \in \mathcal{D} \\
    & \underline{\bm{w}}_d \leq \bm{w}_d \leq \overline{\bm{w}}_d & \forall d \in \mathcal{D},~ \text{$\mathbb{P}$-a.s.},
    \end{array}
\end{equation*}
which maximizes the expected daily profits accrued over the entire planning horizon. Here, the function $\Pi_d( \bm w_{d-1}, \bm w_{d}, \bm\xi_{[\Downarrow(d)]})$ represents the expected profit earned by the trader on day~$d$, conditional on the information $ \bm\xi_{[\Downarrow(d)]}$ available at the beginning of the day and conditional on the initial and terminal reservoir levels $\bm w_{d-1}$ and $\bm w_{d}$, respectively, imposed by the planner. The function $\Pi_d( \bm w_{d-1}, \bm w_{d}, \bm\xi_{[\Downarrow(d)]})$ evaluates the optimal value of the stochastic program
\begin{equation*}
\tag{CT}
\label{opt:collective_trading}
\begin{array}{c@{~~~}ll} \text{sup} & \displaystyle \sum_{t \in \mathcal{T}(d)}\mathbb{E}\big[ \pi_t^{\rm{s}} s_t + \revision{\kappa^{\rm u}_t}(\pi_t^{\rm{u}} + \rho_t^{\rm{u}} \psi_t^{\rm{u}}) u_t + \revision{\kappa^{\rm v}_t}(\pi_t^{\rm{v}} + \rho_t^{\rm{v}} \psi_t^{\rm{v}}) v_t \big\vert \bm\xi_{[\Downarrow(d)]}\big] \\[3mm]
    \text{s.t.} & s_t, u_t, v_t \in \mathbb{R}, ~\bm{g}_t, \bm{p}_t, \bm{z}_t \in \mathcal{L}^A(\mathcal{F}_{[\downarrow(d),t]} ) & \hspace{-3mm} \forall t \in \mathcal{T}(d) \\[3mm]
    & 0 \leq u_t, ~0 \leq v_t, ~\bm{0} \leq \bm{g}_t \leq \overline{\bm{g}}_t, ~\bm{0} \leq \bm{p}_t \leq \overline{\bm{p}}_t, ~\bm{0} \leq \bm{z}_t & \hspace{-3mm} \forall t \in \mathcal{T}(d),~\text{$\mathbb{P}_{\vert \bm\xi_{[\Downarrow(d)]}}$-a.s.}\\[3mm]
    & s_t + \revision{\kappa^{\rm u}_t}\rho_t^{\rm{u}} u_t - \revision{\kappa^{\rm v}_t}\rho_t^{\rm{v}} v_t = \bm{\eta}_t^\top \bm{g}_t - \bm{\zeta}_t^\top \bm{p}_t & \hspace{-3mm} \forall t \in \mathcal{T}(d),~\text{$\mathbb{P}_{\vert \bm\xi_{[\Downarrow(d)]}}$-a.s.}\\
    & \displaystyle \underline{\bm{w}}_t \leq \bm{w}_{d-1} + \sum_{\tau=\downarrow(d)}^t \bm{\phi}_\tau + \mathbf{M}(\bm{g}_\tau - \bm{p}_\tau + \bm{z}_\tau) \leq \overline{\bm{w}}_t & \hspace{-3mm} \forall t \in \mathcal{T}(d),~\text{$\mathbb{P}_{\vert \bm\xi_{[\Downarrow(d)]}}$-a.s.}\\[3mm]
    & \displaystyle \bm{w}_d \leq \bm{w}_{d-1} + \sum_{\tau \in \mathcal{T}(d)} \bm{\phi}_\tau + \mathbf{M}(\bm{g}_\tau - \bm{p}_\tau + \bm{z}_\tau) & \hspace{-3mm} \text{$\mathbb{P}_{\vert \bm\xi_{[\Downarrow(d)]}}$-a.s.}\\[3mm]
    \end{array}
\end{equation*}
solved by the trader, where $\mathbb{P}_{\vert \bm\xi_{[\Downarrow(d)]}}$ denotes the probability measure~$\mathbb P$ conditioned on~$\bm\xi_{[\Downarrow(d)]}$, and $\mathbb{E} [ \cdot \vert \bm\xi_{[\Downarrow(d)]}]$ denotes the expectation with respect to~$\mathbb{P}_{\vert \bm\xi_{[\Downarrow(d)]}}$. Note that the trader's problem~\eqref{opt:collective_trading} may be infeasible for some reservoir targets $\bm w_{d-1}$ and $\bm w_{d}$, in which case $\Pi_d( \bm w_{d-1}, \bm w_{d}, \bm\xi_{[\Downarrow(d)]})$ evaluates to $-\infty$ and thus introduces implicit constraints in the planner's problem~\eqref{opt:collective_bidding_model_dailywater_sep}. 
We also highlight that problem~\eqref{opt:collective_trading} constitutes a stochastic program with contextual information \citep{BanRudin19, ref:kallus-20}, where the contextual covariates $\bm\xi_{[\Downarrow(d)]}$ impact the (conditional) distribution of the random variables in the objective and the constraints.

The next proposition asserts that the planner-trader decomposition incurs no loss of optimality.

\begin{proposition}
\label{prop:collective_bilevel_separation}
The optimal values of problems \eqref{opt:collective_bidding_model_dailywater2} and \eqref{opt:collective_bidding_model_dailywater_sep} are equal.
\end{proposition}

In conclusion, we emphasize that the proposed planner-trader decomposition results in an exact reformulation of the underlying bidding model~\eqref{opt:collective_bidding_model} and that they involve no approximations. This decomposition is useful because the trader's subproblem can be reduced to a tractable~linear~program under a mild information restriction. 

\section{Reduction of the \revision{Trader's Problem}}
\label{sec:reduction_collective}

The trader's problem~\eqref{opt:collective_trading} constitutes an infinite-dimensional stochastic program because it optimizes over functional flow decisions. Such problems are known to be intractable in general \citep{dyer:06, Grani15}. In this section we will show, however, that the stochastic program \eqref{opt:collective_trading} is equivalent to a finite-dimensional linear program if the planner chooses the reservoir filling levels one day ahead of time. Note that selecting $\bm{w}_d$ at the end of day $d-1$ rather than the end of day $d$ amounts to an information restriction whereby the planner sacrifices potentially useful information revealed during day $d$ and thus foregoes some of the achievable expected profit. In other words, requiring that the reservoir filling levels be pre-committed a day in advance leads to a conservative approximation of the planning problem \eqref{opt:collective_bidding_model_dailywater_sep}. For later reference, we formally state this approximation below. 

\begin{approximation}[Information restriction]
\label{apx:day_ahead_water_level}
\revision{For every $d\in\mathcal D$, the end-of-day reservoir filling level $\bm{w}_d$}
is restricted to $\mathcal{L}^R(\mathcal{F}_{[\uparrow(d-1)]}) = \mathcal{L}^R(\mathcal{F}_{[\Downarrow(d)]})\subseteq \mathcal{L}^R(\mathcal{F}_{[\uparrow(d)]})$.
\end{approximation}

Under Approximation~\ref{apx:day_ahead_water_level}, $\bm{w}_d$ becomes a measurable function of~$\bm\xi_{[\Downarrow(d)]}$, which captures the information available to the trader at the beginning of day~$d$. 
Similarly, the natural inflows $\{ \bm{\phi}_t \}_{t \in \mathcal{T}(d)}$ of day $d$ were assumed to be predictable at the beginning of the day and can therefore also be expressed as measurable functions of~$\bm\xi_{[\Downarrow(d)]}$. Finally, the coefficients
\[
	 \tilde \pi_t^{\rm{s}} = \mathbb{E}\big[ \pi_t^{\rm{s}} \,\big\vert\, \bm\xi_{[\Downarrow(d)]}\big], \quad 
	 \tilde \pi_t^{\rm{u}} = \mathbb{E}\big[ \revision{\kappa^{\rm u}_t}(\pi_t^{\rm{u}} + \rho_t^{\rm{u}} \psi_t^{\rm{u}}) \,\big\vert\, \bm\xi_{[\Downarrow(d)]}\big] \quad \text{and} \quad
	 \tilde \pi_t^{\rm{v}} = \mathbb{E}\big[ \revision{\kappa^{\rm v}_t}(\pi_t^{\rm{v}} + \rho_t^{\rm{v}} \psi_t^{\rm{v}}) \,\big\vert\, \bm\xi_{[\Downarrow(d)]}\big] \quad \forall t \in \mathcal{T}(d)
\]
of the (here-and-now) bidding decisions $\{ (s_t, u_t, v_t) \}_{t \in \mathcal{T}(d)}$ in the objective function of problem~\eqref{opt:collective_trading} constitute measurable functions of~$\bm\xi_{[\Downarrow(d)]}$ thanks to the properties of conditional expectations. We may thus conclude that, once a particular realization of~$\bm\xi_{[\Downarrow(d)]}$ is fixed, the individual trader's problem~\eqref{opt:collective_trading} reduces to a multistage stochastic program with $H+1$ decision stages. \revision{Here, the reserve allocations $\{ (\kappa_t^{\rm{u}},\kappa_t^{\rm{v}}) \}_{t \in \mathcal{T}(d)}$, which are revealed simultaneously at the beginning of hour~1, and the reserve activations $\{ (\rho_t^{\rm{u}},\rho_t^{\rm{v}}) \}_{t \in \mathcal{T}(d)}$, which are revealed sequentially at the beginning of the respective hours,} are the only exogenous uncertain parameters affecting the constraints, and the objective function is independent of the (wait-and-see) flow decisions $\{ (\bm{g}_t, \bm{p}_t, \bm{z}_t) \}_{t \in \mathcal{T}(d)}$. 

\revision{Approximation~\ref{apx:day_ahead_water_level} is crucial for our approach to solving the bidding problem~\eqref{opt:collective_bidding_model}, yet it 
is likely to sacrifice optimality. Indeed, selecting the water levels $\bm{w}_d$ already at the end of day $d - 1$ implies that $\bm{w}_d$ cannot adapt to the sequence of reserve allocations $\{ (\kappa_t^\text{u}, \kappa_t^\text{v}) \}_{t \in \mathcal{T} (d)}$  and reserve activations $\{ (\rho_t^\text{u}, \rho_t^\text{v}) \}_{t \in \mathcal{T} (d)}$ encountered during day~$d$. Thus, the water levels $\bm{w}_d$ need to be sufficiently low so as to accommodate for any possible allocation and activation sequence throughout day $d$, and any water saved due to the absence of reserve-up or the presence of reserve-down activations will not be accounted for on subsequent days. While this may discourage bidding on the reserve markets, our case study in Section~\ref{sec:case_study} will show that significant fractions of the overall revenues are nevertheless generated from participation on the reserve markets.}

We will now show that the trader's problem~\eqref{opt:collective_trading} collapses to a tractable linear program if the reservoir targets imposed by the planner are chosen a day in advance. We thus subject the planner's problem~\eqref{opt:collective_bidding_model_dailywater_sep} to the information restriction of Approximation~\ref{apx:day_ahead_water_level}, which reduces the achievable expected revenues but makes the feasible set of the trader's problem~\eqref{opt:collective_trading} independent of all exogenous uncertainties except for the \revision{reserve allocations $\{ (\kappa_t^{\rm{u}}, \kappa_t^{\rm{v}}) \}_{t \in \mathcal{T}(d)}$ and the reserve activations $\{ (\rho_t^{\rm{u}}, \rho_t^{\rm{v}}) \}_{t \in \mathcal{T}(d)}$.} As the objective function of~\eqref{opt:collective_trading} depends only on the bidding decisions $\{ (s_t, u_t, v_t) \}_{t \in \mathcal{T}(d)}$, we may restrict the operational decisions $\{ (\bm{g}_t, \bm{p}_t, \bm{z}_t) \}_{t \in \mathcal{T}(d)}$ to depend only on the reserve activations without sacrificing optimality. Doing so results in the following variant of the trader's problem,
\begin{equation} 
\label{opt:collective_trading_rho}
\begin{aligned}
    &\text{sup} && \sum_{t \in \mathcal{T}(d)} \tilde{\pi}_t^{\rm{s}} s_t + \tilde{\pi}_t^{\rm{u}} u_t + \tilde{\pi}_t^{\rm{v}} v_t \\
    &\text{s.t.} && s_t, u_t, v_t \in \mathbb{R}, ~\bm{g}_t, \bm{p}_t, \bm{z}_t \in \mathcal{L}^A(\mathcal{F}_{[\downarrow(d),t]}^\rho) && \forall t \in \mathcal{T}(d)\\
    &	  && 0 \leq u_t, ~0 \leq v_t, ~\bm{0} \leq \bm{g}_t \leq \overline{\bm{g}}_t, ~\bm{0} \leq \bm{p}_t \leq \overline{\bm{p}}_t, ~\bm{0} \leq \bm{z}_t && \forall t \in \mathcal{T}(d),~\text{$\mathbb{P}_{\vert \bm\xi_{[\Downarrow(d)]}}$-a.s.}\\
    &	  && s_t + \revision{\kappa^{\rm{u}}_t}\rho_t^{\rm{u}} u_t - \revision{\kappa^{\rm{v}}_t}\rho_t^{\rm{v}} v_t = \bm{\eta}_t^\top \bm{g}_t - \bm{\zeta}_t^\top \bm{p}_t && \forall t \in \mathcal{T}(d),~\text{$\mathbb{P}_{\vert \bm\xi_{[\Downarrow(d)]}}$-a.s.}\\
    &	  && \underline{\bm{w}}_t \leq \bm{w}_{d-1} + \sum_{\tau = \downarrow(d)}^t \bm{\phi}_\tau + \mathbf{M} ( \bm{g}_\tau - \bm{p}_\tau + \bm{z}_\tau ) \leq \overline{\bm{w}}_t && \forall t \in \mathcal{T}(d),~\text{$\mathbb{P}_{\vert \bm\xi_{[\Downarrow(d)]}}$-a.s.}\\
    &	  && \bm{w}_d \leq \bm{w}_{d-1} + \sum_{\tau \in \mathcal{T}(d)} \bm{\phi}_\tau + \mathbf{M} ( \bm{g}_\tau - \bm{p}_\tau + \bm{z}_\tau ) && \text{$\mathbb{P}_{\vert \bm\xi_{[\Downarrow(d)]}}$-a.s.},
\end{aligned}
\end{equation}
where $\mathcal{F}_{[\downarrow(d),t]}^\rho=\sigma( \revision{ \{(\kappa_\tau^{\rm{u}}, \kappa_\tau^{\rm{v}}) \}_{\tau \in\mathcal T(d)}},  (\rho_\tau^{\rm{u}}, \rho_\tau^{\rm{v}}) \}_{\tau = \downarrow(d)}^t)$ denotes the $\sigma$-algebra generated by \revision{all reserve allocations of day~$d$ as well as all reserve activations within the interval $[\downarrow(d),t]$.} Even though~\eqref{opt:collective_trading_rho} constitutes a restriction of the trader's problem~\eqref{opt:collective_trading}, their optimal values can be shown to coincide. 

\begin{proposition}
\label{prop:rho_restriction_collective}
Under Approximation~\ref{apx:day_ahead_water_level}, the optimal values of~\eqref{opt:collective_trading} and \eqref{opt:collective_trading_rho} are equal.
\end{proposition}

\revision{Approximation~\ref{apx:day_ahead_water_level} thus reduces the trader's problem~\eqref{opt:collective_trading} to the stochastic program~\eqref{opt:collective_trading_rho}, which accommodates only~$H+1$ decision stages and whose wait-and-see decisions depend only on the reserve allocations and activations. However,~\eqref{opt:collective_trading_rho} still constitutes an infinite-dimensional linear program. We will now show that problem~\eqref{opt:collective_trading_rho} is indeed equivalent to an efficiently solvable linear program. 
As a first step towards this goal, we derive a family of valid inequalities that may be appended to problem~\eqref{opt:collective_trading_rho} without affecting its feasible set.} 

\begin{proposition}
\label{prop:down_cut_collective}
Any feasible solution of \eqref{opt:collective_trading_rho} satisfies $s_t - \revision{\overline{\rho}^{\rm{v}}_t} v_t \geq - \bm{\zeta}_t^\top \overline{\bm{p}}_t$ for all $t \in \mathcal{T}(d)$.
\end{proposition}

The valid inequalities of Proposition~\ref{prop:down_cut_collective} characterize the maximum bids on the reserve-down market (collectively across all arcs of the reservoir system) that can be honored under all possible realizations of the reserve activations. \revision{To see this, note that $\zeta_{t,a} \overline{p}_{t,a}$ represents the maximum amount of energy that can be absorbed on arc~$a$ by pumping. In case of a call-off on the reserve-down market, the energy production on all arcs, combined, can therefore be reduced at most by $s_t + \sum_{a \in \mathcal{A}}\zeta_{t,a} \overline{p}_a$. Appending these valid inequalities to problem~\eqref{opt:collective_trading_rho} yields}
\begin{equation}
\label{opt:collective_trading_cuts}
\begin{aligned}
    &\sup && \sum_{t \in \mathcal{T}(d)} \tilde{\pi}_t^{\rm{s}} s_t + \tilde{\pi}_t^{\rm{u}} u_t + \tilde{\pi}_t^{\rm{v}} v_t \\
    &\text{s.t.} && s_t, u_t, v_t \in \mathbb{R}, ~\bm{g}_t, \bm{p}_t, \bm{z}_t \in \mathcal{L}^A(\mathcal{F}_{[\downarrow(d),t]}^\rho) && \forall t \in \mathcal{T}(d)\\
    &	  && 0 \leq u_t, ~0 \leq v_t, ~\bm{0} \leq \bm{g}_t \leq \overline{\bm{g}}_t, ~\bm{0} \leq \bm{p}_t \leq \overline{\bm{p}}_t, ~\bm{0} \leq \bm{z}_t && \forall t \in \mathcal{T}(d),~\text{$\mathbb{P}_{\vert \bm\xi_{[\Downarrow(d)]}}$-a.s.}\\
    &	  && s_t + \revision{\kappa^{\rm u}_t}\rho_t^{\rm{u}} u_t - \revision{\kappa^{\rm v}_t}\rho_t^{\rm{v}} v_t = \bm{\eta}_t^\top \bm{g}_t - \bm{\zeta}_t^\top \bm{p}_t && \forall t \in \mathcal{T}(d),~\text{$\mathbb{P}_{\vert \bm\xi_{[\Downarrow(d)]}}$-a.s.}\\
    &	  && s_t - \revision{\overline{\rho}^{\rm{v}}_t} v_t \geq - \bm{\zeta}_t^\top \overline{\bm{p}}_t && \forall t \in \mathcal{T}(d) \\
    &	  && \underline{\bm{w}}_t \leq \bm{w}_{d-1} + \sum_{\tau = \downarrow(d)}^t \bm{\phi}_\tau + \mathbf{M} ( \bm{g}_\tau - \bm{p}_\tau + \bm{z}_\tau ) \leq \overline{\bm{w}}_t && \forall t \in \mathcal{T}(d),~\text{$\mathbb{P}_{\vert \bm\xi_{[\Downarrow(d)]}}$-a.s.}\\
    &	  && \bm{w}_d \leq \bm{w}_{d-1} + \sum_{\tau \in \mathcal{T}(d)} \bm{\phi}_\tau + \mathbf{M} ( \bm{g}_\tau - \bm{p}_\tau + \bm{z}_\tau ) && \text{$\mathbb{P}_{\vert \bm\xi_{[\Downarrow(d)]}}$-a.s.},
\end{aligned}
\end{equation}
which is equivalent to problem~\eqref{opt:collective_trading_rho} by Proposition~\ref{prop:down_cut_collective}.

Next, we show that problem~\eqref{opt:collective_trading_cuts} is equivalent to the reduced stochastic program
\begin{equation*} \tag{CT${}^{\rm{r}}$}
\label{opt:reduced_collective_trading}
\begin{aligned}
    &\sup && \sum_{t \in \mathcal{T}(d)} \tilde{\pi}_t^{\rm{s}} s_t + \tilde{\pi}_t^{\rm{u}} u_t + \tilde{\pi}_t^{\rm{v}} v_t \\
    &\text{s.t.} && s_t, u_t, v_t \in \mathbb{R}, ~\bm{g}_t, \bm{p}_t, \bm{z}_t \in \mathbb{R}^A && \forall t \in \mathcal{T}(d)\\
    &	  && 0 \leq u_t, ~0 \leq v_t, ~\bm{0} \leq \bm{g}_t \leq \overline{\bm{g}}_t, ~\bm{0} \leq \bm{p}_t \leq \overline{\bm{p}}_t, ~\bm{0} \leq \bm{z}_t && \forall t \in \mathcal{T}(d)\\
    &	  && s_t + \revision{\overline{\rho}^{\rm{u}}_t} u_t = \bm{\eta}_t^\top \bm{g}_t - \bm{\zeta}_t^\top \bm{p}_t && \forall t \in \mathcal{T}(d)\\
    &	  && s_t - \revision{\overline{\rho}^{\rm{v}}_t} v_t \geq - \bm{\zeta}_t^\top \overline{\bm{p}}_t && \forall t \in \mathcal{T}(d) \\
    &	  && \underline{\bm{w}}_t \leq \bm{w}_{d-1} + \sum_{\tau = \downarrow(d)}^t \bm{\phi}_\tau + \mathbf{M} ( \bm{g}_\tau - \bm{p}_\tau + \bm{z}_\tau ) \leq \overline{\bm{w}}_t && \forall t \in \mathcal{T}(d)\\
    &	  && \bm{w}_d \leq \bm{w}_{d-1} + \sum_{\tau \in \mathcal{T}(d)} \bm{\phi}_\tau + \mathbf{M} ( \bm{g}_\tau - \bm{p}_\tau + \bm{z}_\tau ).
\end{aligned}
\end{equation*}
Problem~\eqref{opt:reduced_collective_trading} involves only here-and-now decisions, and its constraints are deterministic conditional on $\bm\xi_{[\Downarrow(d)]}$. In fact, it constitutes a linear program of size $\mathcal O(A\cdot H)$.



\begin{proposition}
\label{prop:robusttodeterministic_collective1}
Under Approximation~\ref{apx:day_ahead_water_level}, the optimal value of problem~\eqref{opt:reduced_collective_trading} is non-inferior to that of problem~\eqref{opt:collective_trading_cuts}.
\end{proposition}

\revision{
Note that Proposition~\ref{prop:robusttodeterministic_collective1} does not prove the equivalence of the problems~\eqref{opt:reduced_collective_trading} and~\eqref{opt:collective_trading_cuts}. Indeed, proving this equivalence directly appears to be difficult. Instead, Appendix~\ref{sec:reduction_individual} proves that under Approximation~\ref{apx:day_ahead_water_level},  problem~\eqref{opt:collective_trading} is a restriction of problem~\eqref{opt:reduced_collective_trading}. Since Propositions~\ref{prop:rho_restriction_collective}--\ref{prop:robusttodeterministic_collective1} also imply that~\eqref{opt:reduced_collective_trading} is a restriction of~\eqref{opt:collective_trading}, both problems indeed share the same optimal value.
The relations between the different variants of the trader's problem are illustrated in Figure~\ref{fig:proofplan}.
Every arc encodes a relation $A \leq B$, where $A$ and $B$ represent the optimal values of the problems at the arc's tail and head, respectively. Dashed arcs indicate trivial relaxations, and solid arcs represent non-trivial implications proved in the referenced propositions. Interpreting Figure~\ref{fig:proofplan} as a directed graph, we note that Proposition~\ref{prop:robusttodeterministic_collective1} and Appendix~\ref{sec:reduction_individual} complete a counter-clockwise loop that visits each node, which implies that the optimal values of {\em all} optimization problems are in fact equal. 

%
{\color{black}
\begin{figure}[h]
\begin{center}
\begin{tikzpicture}[scale=0.9, every node/.style={scale=0.9}, squarednode/.style={rectangle, draw=black, fill=gray!5, very thick, minimum width=15mm, minimum height=10mm}, node distance=4.7cm]
	\node[squarednode] (C1) {\eqref{opt:collective_trading}};
	\node[squarednode] (C2) [right of=C1] {\eqref{opt:collective_trading_rho}};
	\node[squarednode] (C3) [right of=C2] {\eqref{opt:collective_trading_cuts}};
	\node[squarednode] (C4) [right of=C3] {\eqref{opt:reduced_collective_trading}};

	\draw[-{Latex[length=2mm]}] ($(C1.east)+(0mm,1mm)$) -- node[above]{Proposition~\ref{prop:rho_restriction_collective}} ($(C2.west)+(0mm,1mm)$);
	
	\draw[-{Latex[length=2mm]}] ($(C2.east)+(0mm,1mm)$) -- node[above]{Proposition~\ref{prop:down_cut_collective}} ($(C3.west)+(0mm,1mm)$);
	
	\draw[-{Latex[length=2mm]}] ($(C3.east)+(0mm,0mm)$) -- node[above]{Proposition~\ref{prop:robusttodeterministic_collective1}} ($(C4.west)+(0mm,0mm)$);
	
	\draw[{Latex[length=2mm]}-,dashed] ($(C1.east)+(0mm,-1mm)$) --  ($(C2.west)+(0mm,-1mm)$);
	
	\draw[{Latex[length=2mm]}-,dashed] ($(C2.east)+(0mm,-1mm)$) --  ($(C3.west)+(0mm,-1mm)$);	
	
    \draw[-] ($(C4.south)+(0mm,0mm)$) --  ($(C4.south)+(0mm,-40mm)$);
	
	\draw[-{Latex[length=2mm]}] ($(C1.south)+(0mm,-40mm)$) --  ($(C1.south)+(0mm,0mm)$);
	
	\draw[-] ($(C4.south)+(0mm,-40mm)$) -- node[above]{Appendix~\ref{sec:reduction_individual}} ($(C1.south)+(0mm,-40mm)$);
	
\end{tikzpicture}
\end{center}
	\caption{Illustration of the relations between different variants of the trader's problem. Dashed arcs represent trivial relaxations, and solid arcs represent non-trivial implications proved in the referenced propositions.} 
	\label{fig:proofplan}
\end{figure}
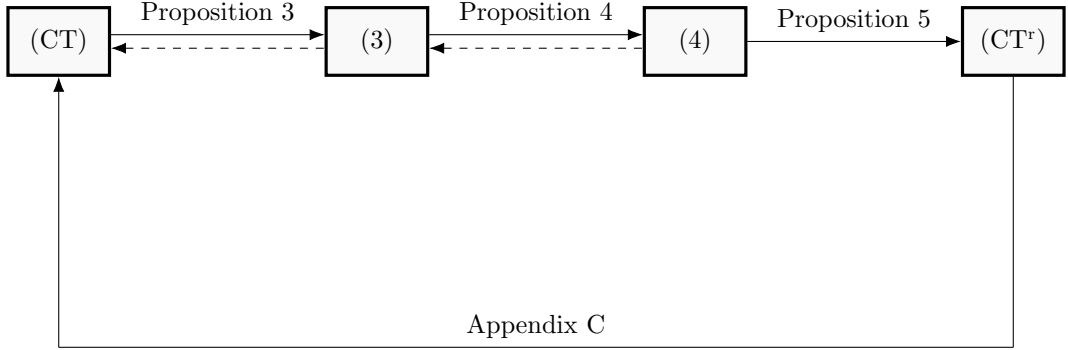
}

The results of Section~\ref{sec:reduction_collective} as visualized in Figure~\ref{fig:proofplan} culminate in the following main theorem. }


\begin{theorem}
\label{thm:all_collective}
Under Approximation~\ref{apx:day_ahead_water_level}, the optimal values of~\eqref{opt:collective_trading} and~\eqref{opt:reduced_collective_trading} are equal.
\end{theorem}

\revision{In the remainder of this section we leverage Theorem~\ref{thm:all_collective} to construct a restriction of the bidding model~\eqref{opt:collective_bidding_model} that is susceptible to further approximations and numerical solution. To this end, note first that Approximation~\ref{apx:day_ahead_water_level} restricts the planner's problem~\eqref{opt:collective_bidding_model_dailywater_sep} to}
\begin{equation}
\label{opt:reduced_collective_planning}
 \tag{CP${}^{\rm{r}}$}
\begin{array}{c@{~~~}l@{\,}l}
    \text{sup} & \displaystyle \sum_{d \in \mathcal{D}}  \mathbb{E} \big[ \Pi_d( \bm w_{d-1}, \bm w_{d}, & \bm\xi_{[\Downarrow(d)]}) \big] \\[1ex]
    \text{s.t.} & \bm{w}_d \in \mathcal{L}^R(\mathcal{F}_{[\Downarrow(d)]}) & \forall d \in \mathcal{D} \\
    & \underline{\bm{w}}_d \leq \bm{w}_d \leq \overline{\bm{w}}_d & \forall d \in \mathcal{D},~ \text{$\mathbb{P}$-a.s.}
    \end{array}
\end{equation}
\revision{where the end-of-day reservoir levels $\bm w_d$ are chosen one day in advance and thus adapt to information available at time $\Downarrow(d)$. Under this information restriction, Theorem~\ref{thm:all_collective} allows us to compute the optimal value $\Pi_d( \bm w_{d-1}, \bm w_{d}, \bm\xi_{[\Downarrow(d)]})$ of the trader's problem~\eqref{opt:collective_trading} by solving the linear program~\eqref{opt:reduced_collective_trading}, which enables us to prove that the planner's reduced problem~\eqref{opt:reduced_collective_planning} is equivalent~to}
\begin{equation}
\label{opt:collective_bidding_model_dailywater_sep_apx}
\tag{C${}^{\rm{r}}$}
\begin{aligned}
    &\hspace{-1mm}\text{sup} && \sum_{t \in \mathcal{T}} \mathbb{E}\left[ \tilde \pi_t^{\rm{s}}\,s_t + \tilde \pi_t^{\rm{u}} \,u_t + \tilde \pi_t^{\rm{v}} \, v_t 
\right]\\
    &\hspace{-1mm}\text{s.t.} &&s_t, u_t, v_t \in \mathcal{L}(\mathcal{F}_{[T]}), ~\bm{g}_t, \bm{p}_t, \bm{z}_t \in \mathcal{L}^A(\mathcal{F}_{[T]}), ~\bm{w}_d \in  \mathcal{L}^R(\mathcal{F}_{[\Downarrow(d)]}) && \forall d \in \mathcal{D}, ~\forall t \in \mathcal{T}(d)\\
    &\hspace{-1mm}	  && 0 \leq u_t, ~0 \leq v_t, ~\bm{0} \leq \bm{g}_t \leq \overline{\bm{g}}_t, ~\bm{0} \leq \bm{p}_t \leq \overline{\bm{p}}_t, ~\bm{0} \leq \bm{z}_t && \hspace{0mm} \forall d \in \mathcal{D}, ~\forall t \in \mathcal{T}(d),~\text{$\mathbb{P}$-a.s.}\\
    &\hspace{-1mm}	  && s_t + \revision{\overline{\rho}^{\rm{u}}_t} u_t = \bm{\eta}_t^\top \bm{g}_t - \bm{\zeta}_t^\top \bm{p}_t, ~s_t - \revision{\overline{\rho}^{\rm{v}}_t}v_t \geq - \bm{\zeta}_t^\top \overline{\bm{p}}_t && \hspace{0mm} \forall d \in \mathcal{D}, ~\forall t \in \mathcal{T}(d),~\text{$\mathbb{P}$-a.s.}\\
    &\hspace{-1mm}	  && \underline{\bm{w}}_t  \leq \bm{w}_{d-1} + \sum_{\tau=\downarrow(d)}^t \bm{\phi}_\tau + \mathbf{M}(\bm{g}_\tau - \bm{p}_\tau + \bm{z}_\tau) \leq \overline{\bm{w}}_t && \hspace{0mm} \forall d \in \mathcal{D}, ~\forall t \in \mathcal{T}(d),~\text{$\mathbb{P}$-a.s.}\\
    &\hspace{-1mm}	  && \bm{w}_d \leq \bm{w}_{d-1} + \sum_{\tau \in \mathcal{T}(d)}\bm{\phi}_\tau + \mathbf{M}(\bm{g}_\tau - \bm{p}_\tau + \bm{z}_\tau) && \hspace{0mm} \forall d \in \mathcal{D},~\text{$\mathbb{P}$-a.s.}\\
    &\hspace{-1mm}	  && \underline{\bm{w}}_d \leq \bm{w}_d \leq \overline{\bm{w}}_d  && \forall d \in \mathcal{D},~\hspace{0mm} \text{$\mathbb{P}$-a.s.}
\end{aligned}
\end{equation}
Problem~\eqref{opt:collective_bidding_model_dailywater_sep_apx} can be viewed as a reduction of the original bidding model~\eqref{opt:collective_bidding_model}.

\begin{theorem}
\label{thm:all_collective_bidding}
The optimal values of the problems~\eqref{opt:reduced_collective_planning} and \eqref{opt:collective_bidding_model_dailywater_sep_apx} are equal, and they are smaller than or equal to the optimal value of problem~\eqref{opt:collective_bidding_model}.
\end{theorem}

\revision{The bidding model~\eqref{opt:collective_bidding_model} and its reduction~\eqref{opt:collective_bidding_model_dailywater_sep_apx} differ in the following aspects.
\begin{itemize}
\item In problem~\eqref{opt:collective_bidding_model} the market decisions $({s}_t, {u}_t, {v}_t)$ are adapted to information that is available at the {\em beginning of day $d(t)$}, while the operational decisions $(\bm{g}_t, \bm{p}_t, \bm{z}_t)$ are adapted to {\em real-time information}. In contrast, in~\eqref{opt:collective_bidding_model_dailywater_sep_apx} these decisions are taken under {\em perfect information}. 
\item Problem~\eqref{opt:collective_bidding_model} has {\em random recourse} because the reserve market bids in the energy delivery constraints are multiplied by the uncertain reserve activations. In contrast, problem~\eqref{opt:collective_bidding_model_dailywater_sep_apx} has {\em fixed recourse} because all reserve activations were eliminated from the constraints.
\item Problem~\eqref{opt:collective_bidding_model_dailywater_sep_apx} accommodates the {\em valid cuts} derived in Proposition~\ref{prop:down_cut_collective}, which are absent in~\eqref{opt:collective_bidding_model}.
\end{itemize}

We highlight again that problem~\eqref{opt:collective_bidding_model_dailywater_sep_apx} was obtained from the original bidding model~\eqref{opt:collective_bidding_model} by applying a single information restriction, whereby the end-of-day reservoir levels must be chosen one day in advance. No other approximations have been applied. Finally, the equivalence between~\eqref{opt:reduced_collective_planning} and \eqref{opt:collective_bidding_model_dailywater_sep_apx} established in Theorem~\ref{thm:all_collective_bidding} is in terms of the problems' optimal values and not necessarily of their optimal solutions. In essence, the proof shows that the optimal market decisions ${s}_t$,  ${u}_t$ and ${v}_t$ as well as the optimal operational decisions $\bm{g}_t$, $\bm{p}_t$ and $\bm{z}_t$ of problem~\eqref{opt:collective_bidding_model_dailywater_sep_apx}, which may fail to be non-anticipative, can be converted to non-anticipative decisions
\begin{equation*}
\begin{aligned}
    &{s}'_t = \mathbb{E} \left[ {s}_t \vert \bm\xi_{[\Downarrow (t)]} \right], \quad 
    {u}'_t = \mathbb{E} \left[ {u}_t \vert \bm\xi_{[\Downarrow (t)]} \right], \quad
    {v}'_t = \mathbb{E} \left[ {v}_t \vert \bm\xi_{[\Downarrow (t)]} \right] \\
    &\bm{g}'_t = \mathbb{E} \left[ \bm{g}_t \vert \bm\xi_{[\Downarrow (t)]} \right], \quad 
    \bm{p}'_t = \mathbb{E} \left[ \bm{p}_t \vert \bm\xi_{[\Downarrow (t)]} \right], \quad
    \bm{z}'_t = \mathbb{E} \left[ \bm{z}_t \vert \bm\xi_{[\Downarrow (t)]} \right]
\end{aligned}
\end{equation*}
by taking conditional expectations. Because the objective and constraint functions are all linear in these decision variables, the new solution remains feasible and---in fact---optimal in problem~\eqref{opt:collective_bidding_model_dailywater_sep_apx}. This argument is only valid because Approximation~\ref{apx:day_ahead_water_level} ensures that problem~\eqref{opt:collective_bidding_model_dailywater_sep_apx} has fixed recourse.
}

\section{Numerical Solution of the Reduced Bidding Model}
\label{sec:num_exp_formulation}

\revision{Although Approximation~\ref{apx:day_ahead_water_level} is at the core of our reduction of the bidding model~\eqref{opt:collective_bidding_model} to~\eqref{opt:collective_bidding_model_dailywater_sep_apx}, it is not entirely innocent 
as it may discourage bidding on the reserve markets (\emph{cf.}~Section~\ref{sec:reduction_collective}).
%
Moreover, the reduced planner's problem~\eqref{opt:reduced_collective_planning} constitutes a large-scale stochastic program that is challenging to solve. To simultaneously alleviate both shortcomings, this section develops variants of the reduced planner's and trader's problems~\eqref{opt:reduced_collective_planning} and~\eqref{opt:reduced_collective_trading} with the following characteristics: \emph{(i)} the reduced planner's problem employs an affine decision rule approximation to  result in a two-stage stochastic program that can be solved efficiently using the well-known sample average approximation; \emph{(ii)} the reduced trader's problem accounts not only for the worst possible reserve activation sequence but also for the nominal activation sequence; and \emph{(iii)} the reduced trader's problem contains a valuation of the nominal end-of-day reservoir levels in the objective function. We also incorporate both problems into a rolling horizon framework that repeatedly solves the planner's and trader's problems and always only implements the decisions of the first day. Under these changes, the equivalences of Theorems~\ref{thm:all_collective} and~\ref{thm:all_collective_bidding} no longer hold, but we expect the resulting variants of problems~\eqref{opt:reduced_collective_planning} and~\eqref{opt:reduced_collective_trading} to perform better in actual case studies (\emph{cf.}~Section~\ref{sec:case_study}).

In the following, we first describe how to efficiently solve and extract water values from problem~\eqref{opt:collective_bidding_model_dailywater_sep_apx}. Afterwards, we discuss how these water values can be used to compute the first-day decisions by solving a variant of the trader's problem that accounts both for the water target levels and for valuations of the end-of-day water levels. 
}


In order to compute water values, 
we consider an equivalent reformulation of the reduced planner's problem~\eqref{opt:reduced_collective_planning} that contains the superfluous decision variables $\bm{w}_d^-$, $d \in \mathcal{D}$, to denote the water levels at the \emph{beginning} of day $d \in \mathcal{D}$:
\begin{equation}\label{eq:collective_planer_revised}
\begin{array}{c@{~~~}l@{\quad}l}
    \text{sup} & \displaystyle \sum_{d \in \mathcal{D}}  \mathbb{E} \big[ \Pi_d( \bm w_d^-, \bm w_{d}, \bm\xi_{[\Downarrow(d)]}) \big] \\[1ex]
    \text{s.t.} & \bm{w}_d^- \in \mathcal{L}^R(\mathcal{F}_{[\Downarrow(d - 1)]}),~ \bm{w}_d \in \mathcal{L}^R(\mathcal{F}_{[\Downarrow(d)]}) & \forall d \in \mathcal{D} \\
    & \underline{\bm{w}}_d \leq \bm{w}_d \leq \overline{\bm{w}}_d & \forall d \in \mathcal{D},~ \text{$\mathbb{P}$-a.s.} \\
    & \bm{w}_d^- = \bm{w}_{d-1} & \forall d \in \mathcal{D},~ \text{$\mathbb{P}$-a.s.} 
    \end{array}
\end{equation}
The shadow prices of the newly introduced constraints $\bm{w}_d^- = \bm{w}_{d-1}$ can be interpreted as water values at the end of day $d - 1$, and they will be used later on in our variant of the trader's problem.

While the reduced bidding model~\eqref{eq:collective_planer_revised} has fixed recourse and takes all market bids and operational decisions under perfect information, it still constitutes an infinite-dimensional linear program over a space of measurable functions, and thus there is little hope to solve it {\em exactly}. To make this problem amenable to numerical solution, we further reduce it to a two-stage stochastic program by restricting the reservoir filling levels to parsimonious {\em affine decision rules} that depend on the observable random parameters only through a few judiciously chosen features.

\begin{approximation}[Affine Decision Rule Restriction]
\label{apx:decision_rule}
For all $d \in \mathcal{D}$, the end-of-day reservoir filling levels are representable as
\begin{equation*}
\textstyle \bm{w}_d = \bm{\lambda}_d + \bm \Phi_d \left( \sum_{\tau \in \mathcal{T}(d)} \bm{\phi}_\tau \right) + \mathbf{\Psi}_d \left( \sum_{\tau = 1}^{\uparrow(d-1)} \bm{\phi}_\tau \right) + \bm{\mu}_d \left( \frac{1}{H}\sum_{\tau \in \mathcal{T}(d)} \pi_\tau^{\rm{s}} \right)
\end{equation*}
for some fixed vectors $\bm\lambda_d, \bm\mu_d \in \mathbb{R}^R$ and matrices $\bm\Phi_d, \bm \Psi_d \in \mathbb{R}^{R \times R}$. 
\end{approximation}

\revision{In principle, the decision rule for $\bm{w}_d$ can depend on any linear transformation of $\bm{\xi}_{[\Downarrow(d)]}$. Generally, richer feature vectors lead to more adaptive (and hence better) decisions at a higher computational cost. We do not include as features the reserve \revision{allocations and} activations since they are assumed to be serially independent as well as independent of all other exogenous uncertainties. Instead, Approximation~\ref{apx:decision_rule} restricts the reservoir filling levels at the end of day~$d$ to affine functions of the following features: ($i$)~the cumulative natural inflows into the reservoirs across day~$d$, ($ii$)~the cumulative natural inflows into the reservoirs across the planning horizon until the end of day $d-1$ and ($iii$)~the average spot price on day~$d$. The proposed affine decision rules are parsimonious as they compress the history of all observations into a few relevant features, but they are flexible enough to allow the planner to set different reservoir targets for wet and dry days, for wet and dry seasons as well as for high- and low-price days.} All of the proposed features are observable at the beginning of day $d$, and therefore the non-anticipativity conditions $\bm{w}_d \in \mathcal{L}^R (\mathcal{F}_{[\Downarrow(d)]})$, $d \in \mathcal{D}$, imposed by Approximation~\ref{apx:day_ahead_water_level} are automatically satisfied. Affine decision rule approximations are routinely used in multistage robust optimization and stochastic programming problems; we refer to the reviews by \cite{BenTal09}, \cite{Iancu2015a} and \cite{YGH18:aro_survey}.


Note that Approximation~\ref{apx:decision_rule} restricts the functional form of the reservoir targets, and thus it results in a conservative lower bound on the optimal value of problem~\eqref{opt:collective_bidding_model_dailywater_sep_apx}, which itself underestimates the optimal value of the original bidding model~\eqref{opt:collective_bidding_model} by virtue of Theorem~\ref{thm:all_collective_bidding}. We further emphasize that Approximation~\ref{apx:decision_rule} reduces problem~\eqref{eq:collective_planer_revised} to a two-stage stochastic program with here-and-now decisions $\bm\lambda_d, \bm\mu_d \in \mathbb{R}^R$ and $\bm\Phi_d, \bm \Psi_d \in \mathbb{R}^{R \times R}$, $d\in\mathcal D$, which are chosen without any information about $\bm \xi_{[T]}$, and with wait-and-see decisions $s_t, u_t, v_t\in \mathcal{L}(\mathcal{F}_{[T]})$ and $\bm{g}_t, \bm{p}_t, \bm{z}_t \in \mathcal{L}^A(\mathcal{F}_{[T]})$, $t \in \mathcal{T}$, which are chosen under perfect information about $\bm \xi_{[T]}$. The emerging two-stage stochastic program can then be solved with the popular sample average approximation \citep{Shapiro09}.

\begin{approximation}[Sample Average Approximation]
\label{apx:saa}
The original probability measure $\mathbb P$ is replaced with a discrete empirical measure $\mathbb{\widehat P} = \frac{1}{N}\sum_{n=1}^N \delta_{\omega^{(n)}}$, where $\delta_{\omega^{(n)}}$ stands for the Dirac point mass at $\omega^{(n)}$, and where $\omega^{(n)}\in\Omega$, $n=1,\ldots, N$, constitute independent samples from $\mathbb P$.
\end{approximation}

The combination of affine decision rules and the sample average approximation is not new and has been explored, among others, by \cite{Vayanos12} and \cite{Bodur18} in the context of robust optimization and stochastic programming, respectively. We emphasize that, even though we use an empirical probability measure $\mathbb{\widehat P}$ when solving the decision rule approximation of problem~\eqref{opt:collective_bidding_model_dailywater_sep_apx}, the conditional expectations $\tilde \pi_t^{\rm{s}}$, $\tilde \pi_t^{\rm{u}}$ and $\tilde \pi_t^{\rm{v}}$ in the objective function of~\eqref{opt:collective_bidding_model_dailywater_sep_apx} are pre-computed under the original measure $\mathbb P$. We solve the linear program obtained by applying Approximations~\ref{apx:decision_rule} and~\ref{apx:saa} to problem~\eqref{eq:collective_planer_revised} and record, \revision{in addition to the water target levels $\bm{w}_1$}, the water values $\bm{\vartheta}_1 \in \mathbb{R}^R$ corresponding to the water preservation constraints $\bm{w}_2^- = \bm{w}_1$. Note that these water values are deterministic since $\bm w_1$ and $\bm{w}_2^-$ are deterministic by virtue of Approximation~\ref{apx:day_ahead_water_level}, and that $\vartheta_{1,R} = 0$ by the properties of the dummy reservoir $R$.

To compute the market bids $\{ (s_t, u_t, v_t) \}_{t \in \mathcal{T}(1)}$ for day $d = 1$, we solve the following variant of the trader's problem,
\begin{equation}\label{eq:revised_reduced_trader_problem}
\begin{aligned}
    &\sup && \sum_{t \in \mathcal{T}(1)} \tilde{\pi}_t^{\rm{s}} s_t + \tilde{\pi}_t^{\rm u} u_t + \tilde{\pi}_t^{\rm v} v_t + \bm{\vartheta}_1^\top [ \bm\phi_t + \mathbf{M} ( \hat{\bm{g}}_t - \hat{\bm{p}}_t + \hat{\bm{z}}_t ) ] \\
    &\text{s.t.} && s_t, u_t, v_t \in \mathbb{R}, ~\bm{g}_t, \bm{p}_t, \bm{z}_t \in \mathbb{R}^A,~\hat{\bm{g}}_t, \hat{\bm{p}}_t, \hat{\bm{z}}_t \in \mathbb{R}^A 
    && \forall t \in \mathcal{T}(1) \\
    & && 0 \leq {u}_t, ~{0} \leq {v}_t, ~\bm{0} \leq \bm{g}_t,\hat{\bm{g}}_t \leq \overline{\bm{g}}_t, ~\bm{0} \leq \bm{p}_t,\hat{\bm{p}}_t \leq \overline{\bm{p}}_t, ~\bm{0} \leq \bm{z}_t,\hat{\bm{z}}_t
    && \forall t \in \mathcal{T}(1)\\
    & && s_t + \revision{\overline{\rho}^{\rm{u}}_t} u_t = \bm{\eta}_t^\top \bm{g}_t - \bm{\zeta}_t^\top \bm{p}_t,~s_t - \revision{\overline{\rho}^{\rm{v}}_t} v_t \geq -\bm\zeta_t^\top \overline{\bm{p}}_t 
    && \forall t \in \mathcal{T}(1)\\
    & && s_t + \revision{\hat{\kappa}^{\rm u}_t}\revision{\hat{\rho}_t^{\rm{u}}} u_t- \revision{\hat{\kappa}^{\rm v}_t}\revision{\hat{\rho}_t^{\rm{v}}} v_t = \bm{\eta}_t^\top \hat{\bm{g}}_t - \bm{\zeta}_t^\top \hat{\bm{p}}_t
    && \forall t \in \mathcal{T}(1)\\
    & && \underline{\bm{w}}_t \leq \bm{w}_{0} + \sum_{\tau=\downarrow({\color{black}1})}^t \bm{\phi}_\tau + \mathbf{M} ( \bm{g}_\tau - \bm{p}_\tau + \bm{z}_\tau ) \leq \overline{\bm{w}}_t
    && \forall t \in \mathcal{T}(1) \\
    & && \underline{\bm{w}}_t \leq \bm{w}_{0} + \sum_{\tau=\downarrow({\color{black}1})}^t \bm{\phi}_\tau + \mathbf{M} ( \hat{\bm{g}}_\tau - \hat{\bm{p}}_\tau + \hat{\bm{z}}_\tau ) \leq \overline{\bm{w}}_t
    && \forall t \in \mathcal{T}(1) \\
    & && \revision{\bm{w}_1 \leq \bm{w}_{0} + \sum_{\tau \in \mathcal{T}(1)} \bm{\phi}_\tau + \mathbf{M} ( \bm{g}_\tau - \bm{p}_\tau + \bm{z}_\tau ),}
\end{aligned}
\end{equation}
\revision{where $\hat{\kappa}_t^{\rm{u}} = \mathbb{E} \left[ \kappa_t^{\rm{u}} \right]$, $\hat{\kappa}_t^{\rm{v}} = \mathbb{E} \left[ \kappa_t^{\rm{v}} \right]$,
$\hat{\rho}_t^{\rm{u}} = \mathbb{E} \left[ \rho_t^{\rm{u}} \right]$ and $\hat{\rho}_t^{\rm{v}} = \mathbb{E} \left[\rho_t^{\rm{v}} \right]$ for all $t \in \mathcal{T} (1)$.} Problem~\eqref{eq:revised_reduced_trader_problem} differs from the day-1 reduced trader's problem~\eqref{opt:reduced_collective_trading} of Section~\ref{sec:reduction_collective} in two aspects. Firstly, the operational decisions $\{ (\bm{g}_t, \bm{p}_t, \bm{z}_t) \}_{t \in \mathcal{T}(1)}$ guaranteeing the satisfaction of the reservoir bounds under any possible reserve activation sequence are complemented with the \emph{nominal} operational decisions $\{ (\hat{\bm{g}}_t, \hat{\bm{p}}_t, \hat{\bm{z}}_t) \}_{t \in \mathcal{T}(1)}$ that account for the water dynamics under the \revision{expected scenario, where the products of the reserve allocations and activations are replaced with their means $\mathbb E[\kappa_t^{\rm{u}}\rho_t^{\rm{u}}]=\hat{\kappa}_t^{\rm{u}}\hat{\rho}_t^{\rm{u}}$ and $\mathbb E[\kappa_t^{\rm{v}}\rho_t^{\rm{v}}]=\hat{\kappa}_t^{\rm{v}}\hat{\rho}_t^{\rm{v}}$ for all $t \in \mathcal{T} (1)$.} Secondly, problem~\eqref{eq:revised_reduced_trader_problem} \revision{incorporates both the water target level constraints and a valuation of the nominal end-of-day reservoir levels in the objective function}. These changes alleviate the conservatism introduced by Approximation~\ref{apx:day_ahead_water_level}, \revision{but they imply that the equivalence of Theorem~\ref{thm:all_collective} no longer holds for the revised trader's problem~\eqref{eq:revised_reduced_trader_problem}.}


To obtain the operational decisions $(\bm{g}_\theta, \bm{p}_\theta, \bm{z}_\theta)$ for each hour $\theta \in \mathcal T(1)$ that are compatible with the market bids $(s_t, u_t, v_t)$ and feasible in view of the revealed reserve market activations, finally, we solve the truncated problems
\begin{equation*} 
\begin{aligned}
&\sup && \bm{\vartheta}_1^\top \mathbf{M} \left( {\bm{g}}_\theta - {\bm{p}}_\theta + {\bm{z}}_\theta + \sum_{t=\theta+1}^H (\hat{\bm{g}}_t - \hat{\bm{p}}_t + \hat{\bm{z}}_t) \right) \\
&\text{s.t.} && \bm{g}_t, \bm{p}_t, \bm{z}_t \in \mathbb{R}^A, ~\hat{\bm{g}}_t, \hat{\bm{p}}_t, \hat{\bm{z}_t} \in \mathbb{R}^A && \hspace{0mm} \forall t \in [\theta,H]\\
&	  && \bm{0} \leq \bm{g}_t,\hat{\bm{g}}_t \leq \overline{\bm{g}}_t, ~\bm{0} \leq \bm{p}_t,\hat{\bm{p}}_t \leq \overline{\bm{p}}_t, ~\bm{0} \leq \bm{z}_t,\hat{\bm{z}}_t && \hspace{0mm} \forall t \in [\theta,H]\\
&     && s_\theta + \revision{\kappa^{\rm u}_\theta}\rho^{\rm u}_\theta u_\theta - \revision{\kappa^{\rm v}_\theta}\rho^{\rm v}_\theta v_\theta = \bm{\eta}_\theta^\top \bm{g}_\theta - \bm{\zeta}_\theta^\top \bm{p}_\theta  \\  
&	  && s_t + \revision{\overline{\rho}^{\rm{u}}_t} u_t = \bm{\eta}_t^\top \bm{g}_t - \bm{\zeta}_t^\top \bm{p}_t,
&& \hspace{0mm} \forall t \in [\theta+1,H]\\
&     && s_t + \revision{\kappa^{\rm u}_t}\revision{\hat{\rho}_t^{\rm{u}}} u_t- \revision{\kappa^{\rm v}_t}\revision{\hat{\rho}_t^{\rm{v}}} v_t = \bm{\eta}_t^\top \hat{\bm{g}}_t - \bm{\zeta}_t^\top \hat{\bm{p}}_t && \hspace{0mm} \forall t \in [\theta+1,H]\\
&	  && \underline{\bm{w}}_t \leq \bm{w}_{0} + \sum_{\tau = 1}^t \bm{\phi}_\tau + \mathbf{M} ( \bm{g}_\tau - \bm{p}_\tau + \bm{z}_\tau ) \leq \overline{\bm{w}}_t && \hspace{0mm} \forall t \in [\theta,H] \\
&	  && \underline{\bm{w}}_t \leq \bm{w}_{0} + \sum_{\tau = 1}^t \bm{\phi}_\tau + \mathbf{M} ( \hat{\bm{g}}_\tau - \hat{\bm{p}}_\tau + \hat{\bm{z}}_\tau ) \leq \overline{\bm{w}}_t && \hspace{0mm} \forall t \in [\theta,H] \\
&     && \revision{\bm{w}_1 \leq \bm{w}_{0} + \sum_{\tau \in \mathcal{T}(1)} \bm{\phi}_\tau + \mathbf{M} ( \bm{g}_\tau - \bm{p}_\tau + \bm{z}_\tau ).}
\end{aligned}
\end{equation*}
Each of these problems differs from problem~\eqref{eq:revised_reduced_trader_problem} in two aspects. Firstly, the truncated problems disregard the impact of the (sunk) market bids $\{ (s_t, u_t, v_t) \}_{t \in \mathcal{T}(1)}$ on the objective function. \revision{Secondly, the nominal reserve allocations and activations $\hat{\kappa}_{t}^{\rm{u}}$, $\hat{\kappa}_{t}^{\rm{v}}$, $\hat{\rho}_{\theta}^{\rm{u}}$ and $\hat{\rho}_{\theta}^{\rm{v}}$ in hour $\theta$ are replaced with their actual realizations $\kappa_{t}^{\rm{u}}$, $\kappa_{t}^{\rm{v}}$, $\rho_{\theta}^{\rm{u}}$ and $\rho_{\theta}^{\rm{v}}$, respectively, for all $t \in [\theta+1,H]$.}

Apart from reducing the conservatism introduced by Approximation~\ref{apx:day_ahead_water_level}, our rolling horizon implementation enjoys several advantages: ($i$) The true state of the world~$\omega$ may differ from all discretization points $\omega^{(n)}$, $n=1,\ldots, N$, of the empirical measure~$\mathbb{\widehat P}$, and thus the sample average approximation does not provide any recourse decisions corresponding to~$\omega$. ($ii$) Problem~\eqref{opt:collective_bidding_model} only models a finite time window of the perpetual operation of the reservoir system, which adversely affects the decisions towards the end of the planning horizon~$\mathcal D$. ($iii$) Resolving problem~\eqref{opt:collective_bidding_model} on a daily basis ensures that the end-of-day reservoir targets adapt to all available information, and it reduces the conservatism of the affine decision rule approximation.


\section{Case Study: Gasteiner Tal Cascade}
\label{sec:case_study}

We apply our planner-trader decomposition to a hydropower cascade located in the Gasteiner Tal, Austria. We describe the problem instance in Section~\ref{subsec:parameters}, and we present our results in Section~\ref{subsec:results}. 

\subsection{Problem Parameters}
\label{subsec:parameters}

The Gasteiner Tal cascade comprises three reservoirs: the Bockhartsee annual reservoir with a capacity of 18,500,000 m$^3$, the Nassfeld daily reservoir with a capacity of 230,000 m$^3$ and the Remsach compensation reservoir with a capacity of 4,000 m$^3$. The Bockhartsee reservoir is connected to the Nassfeld reservoir by the Nassfeld pumped-storage plant that has a generating capacity of 40,600 m$^3$/h and a generation efficiency of 6.68 $\times$ $10^{-4}$ MWh/m$^3$, as well as a pumping capacity of 28,500 m$^3$/h and an inverse pumping efficiency of 9.35 $\times$ $10^{-4}$ MWh/m$^3$, respectively. The Nassfeld reservoir is connected to the Remsach reservoir by the B\"ockstein plant with a generating capacity of 41,400 m$^3$/h and an efficiency of 10.3 $\times$ $10^{-4}$ MWh/m$^3$, and the Remsach reservoir is connected to the Gasteiner Ache river by the Remsach  plant with a generating capacity of 50,400 m$^3$/h and an efficiency of 5.41 $\times$ $10^{-4}$ MWh/m$^3$. In total, the cascade produces \revision{up to} 264,000 MWh of electricity per year and covers the demand of approximately 75,000 households.\footnote{Further details of the cascade can be found online at \texttt{https://www.salzburg-ag.at/content/dam/web18/doku}- \texttt{mente/unternehmen/erzeugung/Kraftwerke-GasteinerTal.pdf}.}

\revision{
We model the hourly natural inflows into the three reservoirs based on historical daily inflow data for the years 2002-2012. In absence of more detailed data, we make the simplifying assumption that each hourly inflow within a day amounts to 1/24-th of the corresponding daily inflow. We use the inflows of the years 2002-2011 as in-sample scenarios and the inflows of the year 2012 as out-of-sample scenarios, respectively.

For the spot and reserve prices, we use the Balancing Statistics dataset for the year 2020 that is publicly available on the website of the Austrian Power Grid AG.\footnote{Website: \url{www.apg.at}.} We use the hourly averages of the EXAA and EPEX spot prices as the out-of-sample spot prices, whereas our in-sample spot prices are generated from a linear regression model of the out-of-sample prices with indicator variables for the month of the year, the day of the week and the hour of the day. \revision{In contrast, the capacity fees and the activation prices on the secondary reserve markets are reported in four-hourly and 15-minute blocks, respectively. We use the} four-hourly capacity fees for the secondary reserve-up and reserve-down market as the out-of-sample capacity fees, while our in-sample capacity fees are derived from a logarithmic regression model of the corresponding realized fees with indicator variables for the month of the year, the day of the week and the four-hour block of the day. \revision{Likewise, we use the realized activation prices on the reserve markets at 15-minute granularity as out-of-sample data, whereas} the in-sample prices are generated from a linear regression model of the out-of-sample prices with indicator variables for the month of the year, the day of the week and the hour of the day. While our regression models appear to generally perform reasonably well, we note that the capacity fees appear to be challenging to predict to a high accuracy. We also note that despite the availability of further historical prices, we decided to use the year 2020 as the basis for both the in-sample and the out-of-sample data as the energy prices exhibited significant non-stationarity in recent years due to the changing political climate. Note, however, that the in-sample prices differ from the out-of-sample prices due to our use of the aforementioned regression models.}

\subsection{Results}
\label{subsec:results}

\revision{We now apply our decomposition scheme to the case study of Section~\ref{subsec:parameters}. Our experiments comprise a planning horizon of one year ($52$ weeks $\times$ $7$ days, resulting in $D = 364$ days). To alleviate end-of-horizon effects, we solve our master problem as a rolling-horizon problem where the in-sample and out-of-sample data is re-used in a circular fashion, that is, we re-use the data of year 2020 \emph{in lieu} of the data for the subsequent year 2021. To account for the idiosyncracies of  the energy market of our case study, we also incorporate constraints into our planner and trader problems that require the hourly capacity fee and activation price bids to be constant across 4-hour blocks. Our planner and trader problems are controlled by two hyper-parameters that specify the target probability with which our capacity bids are being accepted as well as the target frequency with which our energy price bids are being activated, respectively. The actual capacity bids and energy prices vary across the four-hour blocks, and they are constructed from the hyper-parameters using our regression models for the capacity fees and activation prices, respectively. Below, we present results for the hyper-parameter combinations $\{ 0.0, 0.25, 0.5, 0.75, 1.0 \} \times \{ 0.0, 0.25, 0.5, 0.75, 1.0 \}$. Our decomposition scheme was implemented in C++ using Gurobi 9.0.1 and run on an 3.2 GHz Intel Xeon processor in single-core mode. All planner problems were solved within less than two minutes, and the trader problems were all solved within fractions of a second. For the sake of exposition, all water levels and revenues are denominated in 1,000s of m$^3$ and 1,000s of EUR, respectively. Our datasets and detailed results, together with the source codes of all algorithms, can be found online.\footnote{Website: \url{https://github.com/napat-rujeerapaiboon/hydro-scheduling}.}

\begin{table}[tb]
    \centering
    \begin{tabular}{r|c|c|c|c|c|}
        \multicolumn{1}{c}{} & \multicolumn{1}{c}{\textbf{0.0}} & \multicolumn{1}{c}{\textbf{0.25}} & \multicolumn{1}{c}{\textbf{0.5}} & \multicolumn{1}{c}{\textbf{0.75}} & \multicolumn{1}{c}{\textbf{1.0}} \\ \cline{2-6}
        \textbf{capacity} & (28.7\%, 26.9\%) & (39.4\%, 32.8\%) & (42.4\%, 38.7\%) & (67.6\%, 46.7\%) & (70.7\%, 47.1\%) \\ \cline{2-6}
        \textbf{energy} & (0.9\%, 1.6\%) & (24.4\%, 26.8\%) & (50.9\%, 51.8\%) & (78.0\%, 77.9\%) & (99.3\%, 99.1\%) \\ \cline{2-6}
    \end{tabular}
    \caption{\revision{Rates at which capacity bids (first row) and energy price bids (second row) are accepted for different target probabilities (columns) on the reserve-up (first entry in each cell) and the reserve-down (second entry in each cell) markets. \label{tab:hyper_params:accuracy}}}
\end{table}

\begin{table}[tb]
	$\mspace{-50mu}$
    \footnotesize
    \begin{tabular}{r|ccccc|}
        \multicolumn{1}{c}{} & \textbf{0.0} & \textbf{0.25} & \textbf{0.5} & \textbf{0.75} & \multicolumn{1}{c}{\textbf{1.0}} \\ \cline{2-6}
        \textbf{0.0} & \cellcolor[gray]{0.95} 10,755.29 & \cellcolor[gray]{0.95} 10,761.46 & \cellcolor[gray]{0.95} 10,769.91 & \cellcolor[gray]{0.95} 10,771.26 & \cellcolor[gray]{0.95} 10,767.04 \\
        & \cellcolor[gray]{0.95} (100.0\%, 0.0\%, 0.0\%) & \cellcolor[gray]{0.95} (99.9\%, 0.1\%, 0.0\%) & \cellcolor[gray]{0.95} (99.7\%, 0.2\%, 0.0\%) & \cellcolor[gray]{0.95} (99.8\%, 0.2\%, 0.0\%) & \cellcolor[gray]{0.95} (99.8\%, 0.2\%, 0.0\%) \\ \cline{2-6}
        \textbf{0.25} & \cellcolor[gray]{0.8} 11,660.82 & \cellcolor[gray]{0.9} 11,065.15 & \cellcolor[gray]{0.82} 11,507.15 & \cellcolor[gray]{0.87} 11,256.31 & \cellcolor[gray]{0.94} 10,842.81 \\
        & \cellcolor[gray]{0.8} (99.3\%, 0.0\%, 0.6\%) & \cellcolor[gray]{0.9} (94.6\%, 1.5\%, 3.9\%) & \cellcolor[gray]{0.82} (94.0\%, 3.1\%, 2.9\%) & \cellcolor[gray]{0.87} (95.2\%, 4.0\%, 0.8\%) & \cellcolor[gray]{0.94} (103.3\%, 2.0\%, -5.2\%) \\ \cline{2-6}
        \textbf{0.5} & \cellcolor[gray]{1.0} 10,479.21 & \cellcolor[gray]{0.96} 10,722.82 & \cellcolor[gray]{0.77} 11,800.98 & \cellcolor[gray]{0.84} 11,424.93 & \cellcolor[gray]{0.99} 10,539.76 \\
        & \cellcolor[gray]{1.0} (97.7\%, 0.4\%, 1.8\%) & \cellcolor[gray]{0.96} (95.2\%, 0.0\%, 4.8\%) & \cellcolor[gray]{0.77} (87.4\%, 9.4\%, 3.2\%) & \cellcolor[gray]{0.84} (89.8\%, 10.6\%, -0.4\%) & \cellcolor[gray]{0.99} (95.5\%, 10.7\%, -6.2\%) \\ \cline{2-6}
        \textbf{0.75} & \cellcolor[gray]{0.92} 10,932.94 & \cellcolor[gray]{0.7} 12,237.71 & \cellcolor[gray]{0.6} 12,836.11 & \cellcolor[gray]{0.62} 12,700.91 & \cellcolor[gray]{0.9} 11,077.83 \\
        & \cellcolor[gray]{0.92} (96.9\%, 0.4\%, 2.7\%) & \cellcolor[gray]{0.7} (83.1\%, 11.9\%, 5.0\%) & \cellcolor[gray]{0.6} (74.5\%, 22.6\%, 2.9\%) & \cellcolor[gray]{0.62} (63.9\%, 38.0\%, -1.9\%) & \cellcolor[gray]{0.9} (67.2\%, 41.9\%, -9.1\%) \\ \cline{2-6}
        \textbf{1.0} & \cellcolor[gray]{0.62} 12,717.43 & \cellcolor[gray]{0.72} 12,130.75 & \cellcolor[gray]{0.5} 13,396.20 & \cellcolor[gray]{0.52} 13,268.93 & \cellcolor[gray]{0.54} 13,141.65 \\
        & \cellcolor[gray]{0.62} (97.6\%, 0.5\%, 2.0\%) & \cellcolor[gray]{0.72} (79.2\%, 16.0\%, 4.7\%) & \cellcolor[gray]{0.5} (53.0\%, 44.8\%, 2.2\%) & \cellcolor[gray]{0.52} (93.9\%, 6.5\%, -0.4\%) & \cellcolor[gray]{0.54} (67.2\%, 40.8\%, -8.0\%) \\ \cline{2-6}
    \end{tabular}
    \caption{\revision{Revenues (top entry) and revenue contributions by market (bottom entry; from left to right: spot, reserve-up and reserve-down market) for different combinations of target probabilities for the capacity bids (rows) and energy price bids (columns). \label{tab:hyper_params:revenues}}}
\end{table}

We first consider the choice of the two hyper-parameters. To this end, Table~\ref{tab:hyper_params:accuracy} compares the target acceptance probabilities for our capacity bids with the actual acceptance frequencies (first row) as well as the target activation frequencies for our energy price bids with the actual activation frequencies (second row) on the reserve-up (first entry) and reserve-down (second entry) markets. The table shows that while our predictions of the energy price activation frequencies are remarkably accurate, the actual capacity bid acceptance frequencies can deviate quite significantly from the target acceptance probabilities. We attribute this to the relatively poor fit of the underlying vanilla regression models described in the previous section. Table~\ref{tab:hyper_params:revenues} compares the out-of-sample revenues generated by different combinations of both hyper-parameters over the course of the year, and it also elucidates how the spot and the two reserve markets contribute towards those revenues. The table reveals that overall, targeting a 100\% probability of having the capacity bids accepted (fifth row) and targeting for a 50\% frequency of having the energy price bids accepted (third row), respectively, performs best. Note that while almost half of the revenues generated under this hyper-parameter combination stem from the reserve-up market, the reserve-down market appears to contribute only a negligible fraction to the overall revenues. This turns out to be misleading, however, as successful bids on the reserve-down market lead to water being pumped to higher reservoirs, which implies further revenue uplifts at later times that we do not account for in this calculation. This is further evidenced by separate numerical experiments (not reported here) which reveal that the inclusion of either reserve market to a spot-only model leads to a revenue uplift of 20\%-25\%, whereas the additional inclusion of the other reserve market only provides a further uplift of 0.1\%-5\%, independent of the order in which the two reserve markets are added.

\begin{figure}[tb]
	$\mspace{-30mu}$
	\begin{tabular}{c@{}c@{}c}
		\includegraphics[scale=0.4]{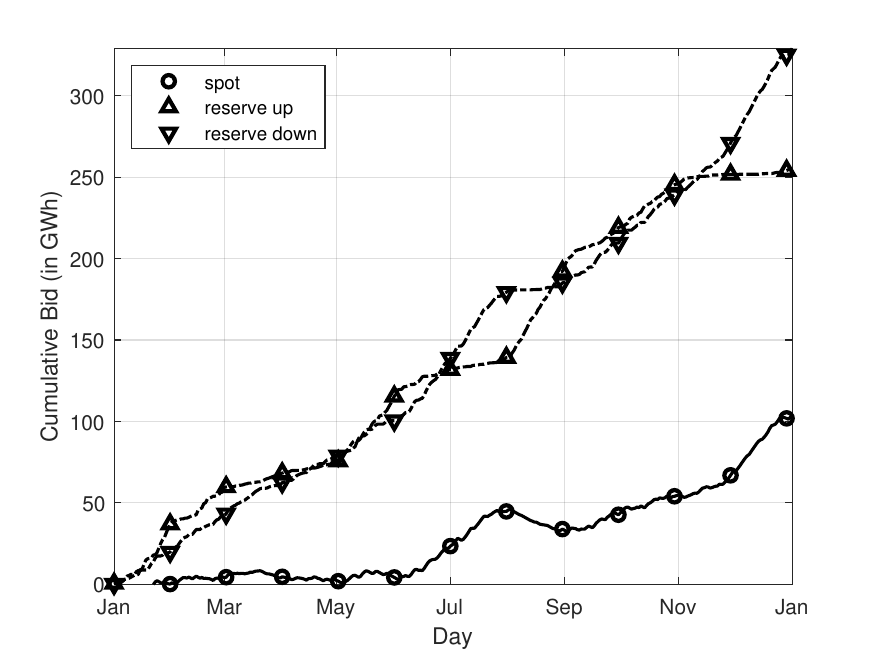} &
		\includegraphics[scale=0.4]{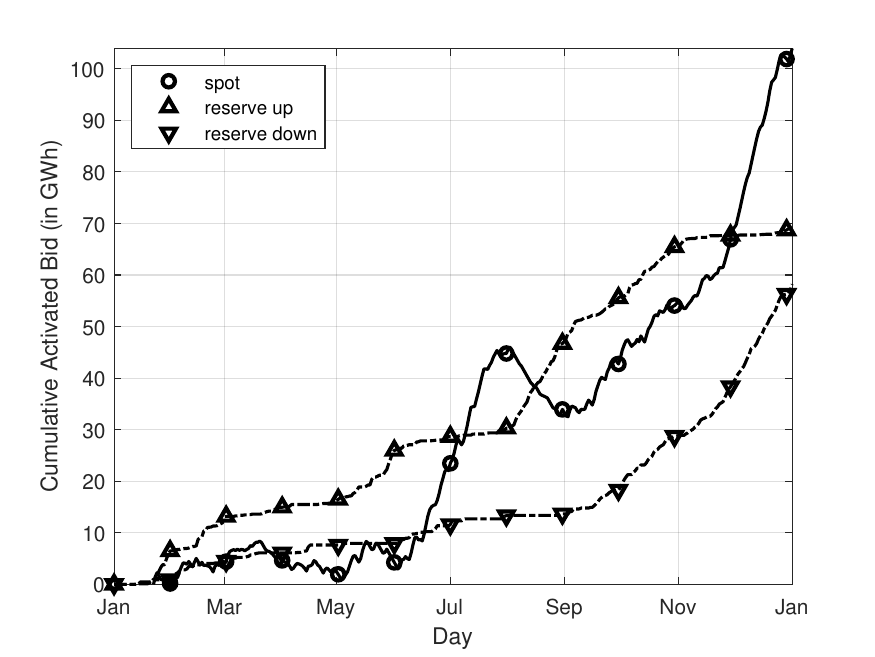} &
		\includegraphics[scale=0.4]{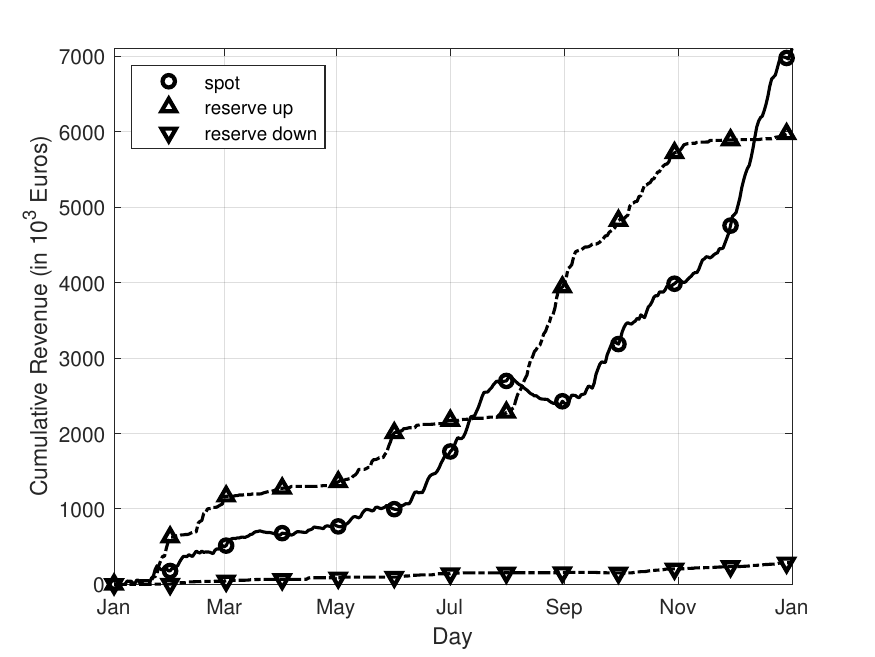}
	\end{tabular}
	\caption{\revision{Cumulative bids and revenue contributions of the spot and the reserve markets. From left to right, the graphs show the cumulative bids, the activated bids and the generated revenues on each of the three markets. \label{fig:revenue_drilldown}}}
\end{figure}

\begin{figure}[tb]
	$\mspace{-30mu}$
	\begin{tabular}{c@{}c@{}c}
		\includegraphics[scale=0.4]{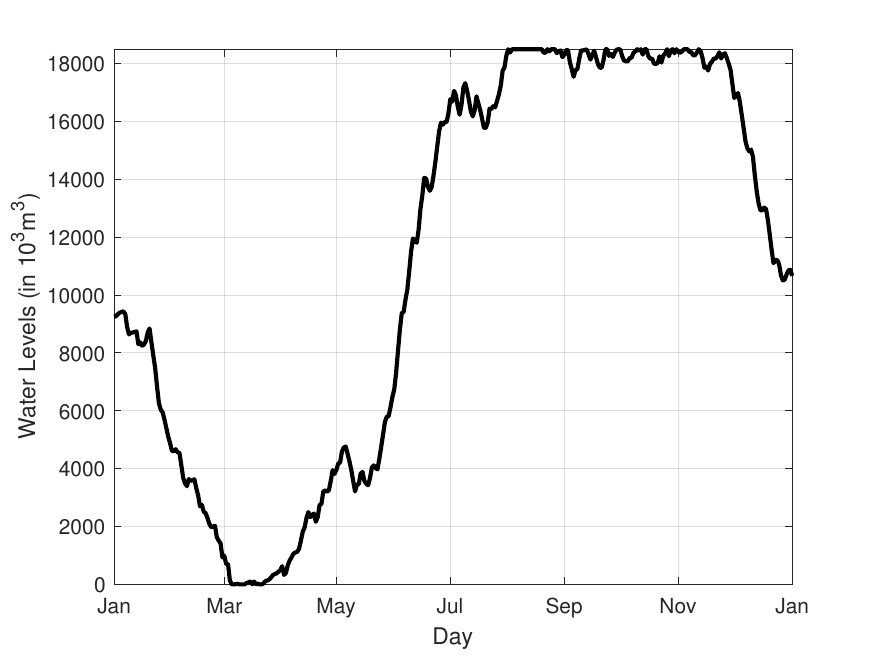} &
		\includegraphics[scale=0.4]{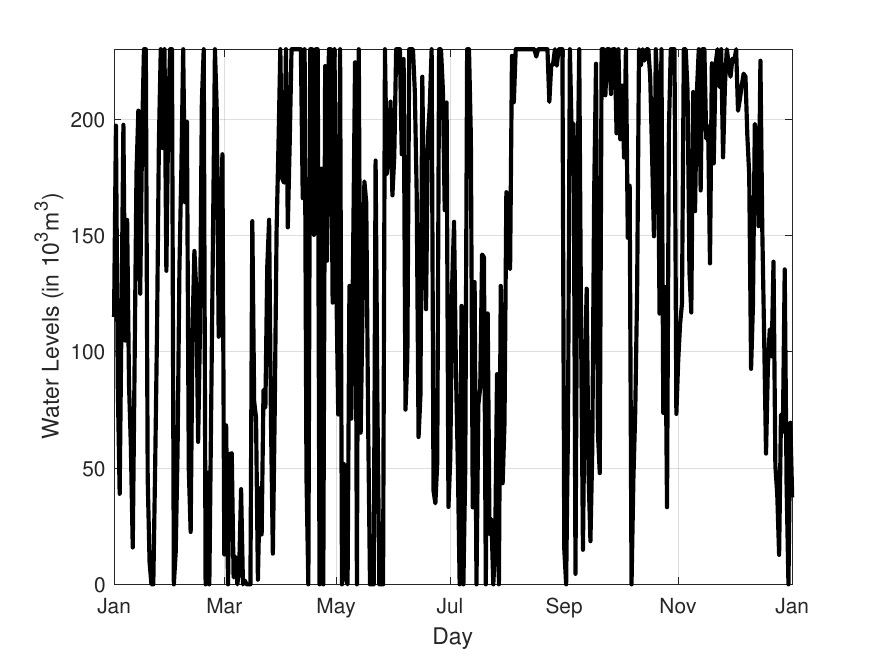} &
		\includegraphics[scale=0.4]{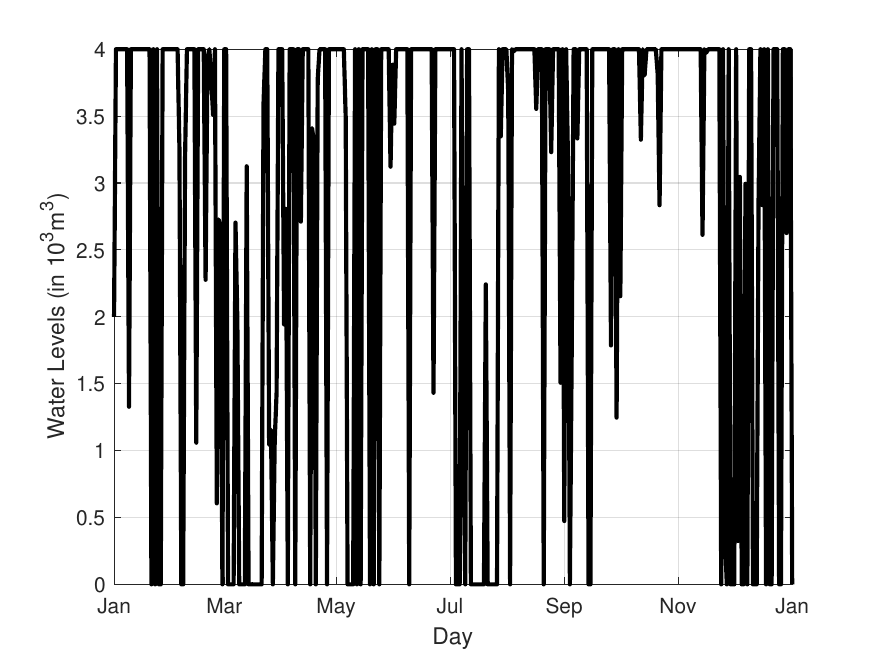}
	\end{tabular}
	\caption{\revision{Out-of-sample reservoir levels for the Bockhartsee (left), Nassfeld (middle) and Remsach (right) reservoirs. \label{fig:detailed_water_levels}}}
\end{figure}

\begin{figure}[tb]
    \centering
    \includegraphics[scale=0.65]{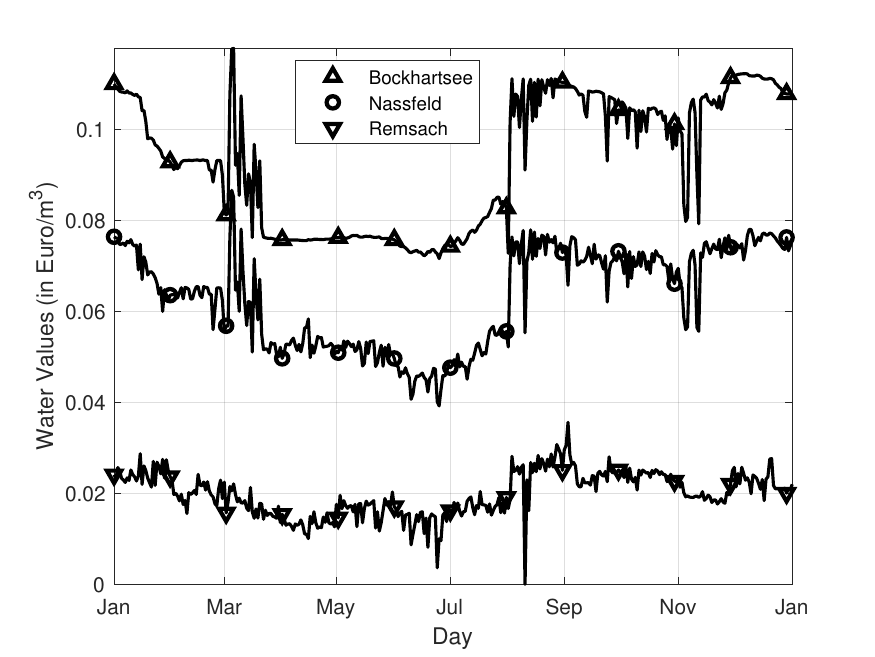}
    \caption{\revision{Water values for the Bockhartsee (left), Nassfeld (middle) and Remsach (right) reservoirs. \label{fig:detailed_water_values}}}
\end{figure}

\begin{figure}[tb]
	\begin{center}
	    \begin{tabular}{c@{}c}
		    \includegraphics[scale=0.5]{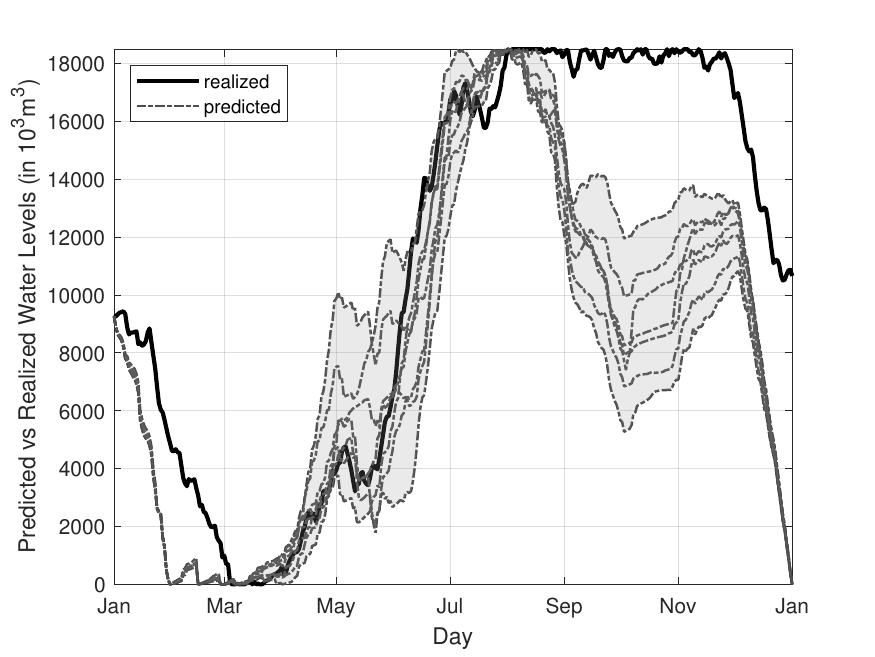} &
		    \includegraphics[scale=0.5]{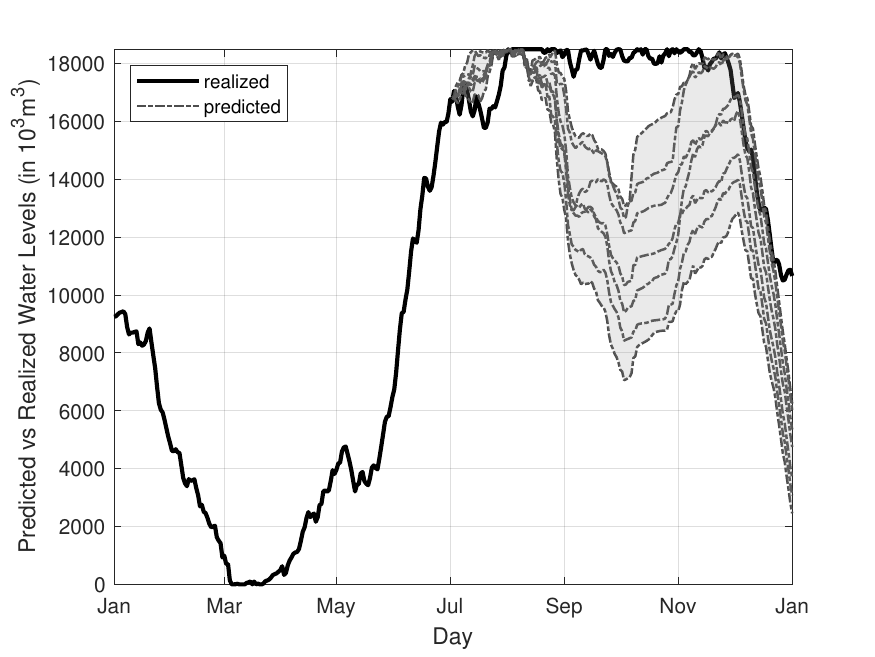}
	    \end{tabular}
	\end{center}
	\caption{\revision{Predictive accuracy of the planner's problem. The realized out-of-sample water levels of the Bockhartsee reservoir are printed in bold, whereas the shaded regions illustrate the in-sample water level predictions on day 1 (left graph) and on day 182 (right graph). \label{fig:predictive_accuracy}}}
\end{figure}

We now take a closer look at the best-performing strategy from Table~\ref{tab:hyper_params:revenues}. Figure~\ref{fig:revenue_drilldown} visualizes the cumulative contributions of the spot and reserve markets towards the energy bids (left), the energy provided (middle) and the revenues generated (right). The figure shows that while the spot and the reserve-up market both generate about half of the overall revenues, the amounts of energy exchanged on all three markets are of the same order of magnitude. Figures~\ref{fig:detailed_water_levels} and~\ref{fig:detailed_water_values} present the water levels and water values for the three reservoirs of our case study, respectively. The figures show that our model correctly captures the different time scales of the reservoir dynamics. While the filling level of the Bockhartsee annual reservoir changes slowly throughout the year, the Nassfeld and Remsach reservoirs operate at a much faster pace. The seasonality of the water levels also implies a seasonality in the water values, which are highest ahead of the depletion phase of the top-level reservoir. Note that in our experiments, all lower reservoir bounds are set to zero. In practice, one typically imposes strictly positive lower reservoir bounds to protect the aquatic life as well as to avoid a negative impact on tourism. Figure~\ref{fig:predictive_accuracy} investigates the predictive accuracy of our bidding model. It visualizes the out-of-sample water levels of the Bockhartsee reservoir across the time horizon, together with the in-sample prediction of these water levels on day $1$ and day $182$. We observe that the out-of-sample water levels tend to stay close the predicted intervals for a while, but that our predictions become increasingly inaccurate as we look further into the future. We emphasize that the water level predictions are updated every day, rather than twice during the time horizon as shown in this figure.

\begin{figure}[tb]
    \centering
    \includegraphics[scale=0.65]{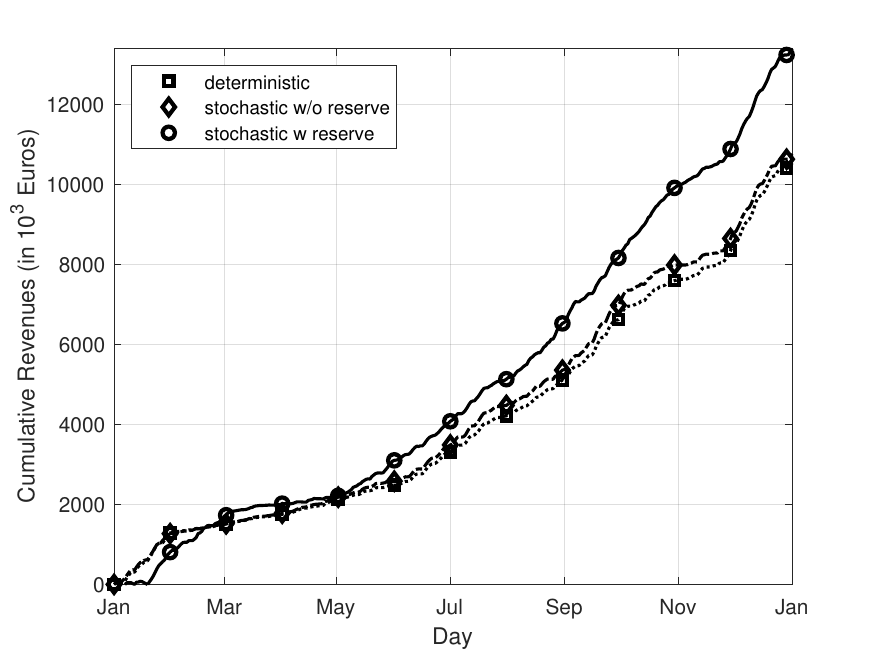}
    \caption{\revision{Cumulative revenues generated by the deterministic model, the stochastic model without reserve markets and the stochastic model including reserve markets. \label{fig:comparison_to_competitors}}}
\end{figure}

In our final experiment, we compare our bidding problem with a variant of the problem that precludes participation on the reserve markets. We also compare our model against a deterministic bidding formulation that replaces all stochastic parameters with their expected values. In particular, this formulation replaces the stochastic reserve market activations with deterministic fractional activations, which implies that the hydropower producer may not be able to honor her reserve market commitments out-of-sample. In such cases, we revert to a spot-only trading strategy for the affected days. We observe empirically that the cumulative out-of-sample revenues of the deterministic formulation are maximized when the total reserve market participation is restricted to not exceed the spot market participation in any given hour. This ad hoc restriction avoids overly aggressive reserve market participation that underperforms out-of-sample due to frequent switches to the aforementioned spot-only strategy. Note that by retrospectively switching to a spot-only strategy based on the observed out-of-sample reserve activations, the deterministic strategy has an advantage and is not implementable in reality. Figure~\ref{fig:comparison_to_competitors} visualizes the cumulative out-of-sample revenues generated by the three models. The figure highlights the benefits of a stochastic strategy that participates in both the spot and the reserve markets: our strategy outperforms both the spot-only stochastic model and the deterministic model \mbox{by about 25\% in terms of cumulative revenues.}}

\section*{Acknowledgments}

This publication is posthumously dedicated to the first author, Kilian Schindler (29.10.1989– 16.5.2020), who has led the research in collaboration with the second author. We are grateful to Georg Ostermaier from Decision Trees GmbH and Dirk Lauinger for valuable discussions and advice, as well as the associate editor and the referees for their suggestions that significantly improved the manuscript. This research was supported by the Ministry of Education, Singapore, under its 2019 Academic Research Fund Tier 3 grant MOE-2019-T3-1-010 and the Swiss National Science Foundation grant BSCGI0\_157733 as well as the EPSRC grant EP/R045518/1.

\bibliography{bibliography}

\begin{thebibliography}{48}
\expandafter\ifx\csname natexlab\endcsname\relax\def\natexlab#1{#1}\fi
\expandafter\ifx\csname url\endcsname\relax
  \def\url#1{{\tt #1}}\fi
\expandafter\ifx\csname urlprefix\endcsname\relax\def\urlprefix{URL }\fi
\expandafter\ifx\csname urlstyle\endcsname\relax
  \expandafter\ifx\csname doi\endcsname\relax
  \def\doi#1{doi:\discretionary{}{}{}#1}\fi \else
  \expandafter\ifx\csname doi\endcsname\relax
  \def\doi{doi:\discretionary{}{}{}\begingroup \urlstyle{rm}\Url}\fi \fi

\bibitem[{Aasg\r{a}rd et~al.(2014)Aasg\r{a}rd, Andersen, Fleten, and
  Haugstvedt}]{Aasgard14}
Aasg\r{a}rd, E., G.~Andersen, S.~Fleten, D.~Haugstvedt. 2014.
\newblock Evaluating a stochastic-programming-based bidding model for a
  multireservoir system.
\newblock {\it IEEE Transactions on Power Systems\/} {\bf 29}(4) 1748--1757.

\bibitem[{Aasg\r{a}rd et~al.(2019)Aasg\r{a}rd, Fleten, Kaut, Midthun, and
  Perez-Valdes}]{Aasgard19}
Aasg\r{a}rd, E., S.~Fleten, M.~Kaut, K.~Midthun, G.~Perez-Valdes. 2019.
\newblock Hydropower bidding in a multi-market setting.
\newblock {\it Energy Systems\/} {\bf 10} 543--565.

\bibitem[{Abgottspon and Andersson(2012)}]{abgottspon12}
Abgottspon, H., G.~Andersson. 2012.
\newblock Approach of integrating ancillary services into a medium-term hydro
  optimization.
\newblock {\it Proceedings of the 12th Symposium of Specialists in Electric
  Operational and Expansion Planning\/}. 1--10.

\bibitem[{Abgottspon et~al.(2014)Abgottspon, Njálsson, Bucher, and
  Andersson}]{Abgottspon14}
Abgottspon, H., K.~Njálsson, M.~A. Bucher, G.~Andersson. 2014.
\newblock Risk-averse medium-term hydro optimization considering provision of
  spinning reserves.
\newblock {\it Proceedings of the 2014 International Conference on
  Probabilistic Methods Applied to Power Systems\/}. 1--6.

\bibitem[{Andrieu et~al.(2010)Andrieu, Henrion, and R\"omisch}]{Andrieu10}
Andrieu, L., R.~Henrion, W.~R\"omisch. 2010.
\newblock A model for dynamic chance constraints in hydro power reservoir
  management.
\newblock {\it European Journal of Operational Research\/} {\bf 207}(2)
  579--589.

\bibitem[{Ban and Rudin(2019)}]{BanRudin19}
Ban, G.-Y., C.~Rudin. 2019.
\newblock The big data newsvendor: {P}ractical insights from machine learning.
\newblock {\it Operations Research\/} {\bf 67}(1) 90--108.

\bibitem[{Ben-Tal et~al.(2009)Ben-Tal, El~Ghaoui, and Nemirovski}]{BenTal09}
Ben-Tal, A., L.~El~Ghaoui, A.~Nemirovski. 2009.
\newblock {\it Robust Optimization\/}.
\newblock Princeton University Press.

\bibitem[{Bertsimas and Kallus(2020)}]{ref:kallus-20}
Bertsimas, D., N.~Kallus. 2020.
\newblock From predictive to prescriptive analytics.
\newblock {\it Management Science\/} {\bf 66}(3) 1025--1044.

\bibitem[{Bodur and Luedtke(2022)}]{Bodur18}
Bodur, M., J.~R. Luedtke. 2022.
\newblock Two-stage linear decision rules for multi-stage stochastic
  programming.
\newblock {\it Mathematical Programming\/} {\bf 191}(1) 347--380.

\bibitem[{Carpentier et~al.(2013)Carpentier, Gendreau, and
  Bastin}]{Carpentier13:sp}
Carpentier, P.‐L., M.~Gendreau, F.~Bastin. 2013.
\newblock Long‐term management of a hydroelectric multireservoir system under
  uncertainty using the progressive hedging algorithm.
\newblock {\it Water Resources Research\/} {\bf 49}(5) 2812--2827.

\bibitem[{Chazarra et~al.(2016)Chazarra, Garc\'{i}a-González,
  P\'{e}rez-D\'{i}az, and Arteseros}]{Chazzara16}
Chazarra, M., J.~Garc\'{i}a-González, J.~P\'{e}rez-D\'{i}az, M.~Arteseros.
  2016.
\newblock Stochastic optimization model for the weekly scheduling of a
  hydropower system in day-ahead and secondary regulation reserve markets.
\newblock {\it Electric Power Systems Research\/} {\bf 130} 67--77.

\bibitem[{{Chazarra} et~al.(2018){Chazarra}, {P\'{e}rez-D\'{i}az}, and
  {Garc\'{i}a-Gonz\'{a}lez}}]{Chazarra18}
{Chazarra}, M., J.~I. {P\'{e}rez-D\'{i}az}, J.~{Garc\'{i}a-Gonz\'{a}lez}. 2018.
\newblock Optimal joint energy and secondary regulation reserve hourly
  scheduling of variable speed pumped storage hydropower plants.
\newblock {\it IEEE Transactions on Power Systems\/} {\bf 33}(1) 103--115.

\bibitem[{Chazarra et~al.(2014)Chazarra, Pérez-Díaz, and
  García-González}]{Chazarra14:opt_op}
Chazarra, M., J.~I. Pérez-Díaz, J.~García-González. 2014.
\newblock Optimal operation of variable speed pumped storage hydropower plants
  participating in secondary regulation reserve markets.
\newblock {\it Proceedings of the 11th International Conference on the European
  Energy Market (EEM14)\/}. 1--5.

\bibitem[{Chazarra et~al.(2018)Chazarra, Pérez-Díaz, García-González, and
  Praus}]{Chazarra18:econ_viab}
Chazarra, M., J.~I. Pérez-Díaz, J.~García-González, R.~Praus. 2018.
\newblock Economic viability of pumped-storage power plants participating in
  the secondary regulation service.
\newblock {\it Applied Energy\/} {\bf 216} 224--233.

\bibitem[{de~Matos et~al.(2015)de~Matos, Philpott, and Finardi}]{deMatos15}
de~Matos, V., A.~Philpott, E.~Finardi. 2015.
\newblock Improving the performance of stochastic dual dynamic programming.
\newblock {\it Journal of Computational and Applied Mathematics\/} {\bf 290}
  196--208.

\bibitem[{Delage and Iancu(2015)}]{Iancu2015a}
Delage, E., D.~Iancu. 2015.
\newblock Robust multistage decision making.
\newblock {\it INFORMS Tutorials in Operations Research\/}  19--46.

\bibitem[{{DOE}(2020)}]{DOE20:storage}
{DOE}. 2020.
\newblock {DOE} {OE} global energy storage database.
\newblock \url{https://www.sandia.gov/ess-ssl/global-energy-}
  \url{storage-database/}.
\newblock Accessed on 1st March 2021.

\bibitem[{Dyer and Stougie(2006)}]{dyer:06}
Dyer, M., L.~Stougie. 2006.
\newblock Computational complexity of stochastic programming problems.
\newblock {\it Mathematical Programming\/} {\bf 106}(3) 423--432.

\bibitem[{{European Commission}(2020)}]{EC16:METIS}
{European Commission}. 2020.
\newblock European {C}ommission {METIS} technical note {T4}: Overview of
  {E}uropean electricity markets.
\newblock
  \url{https://ec.europa.eu/energy/sites/default/files/documents/overview_of_european_electricity_markets.pdf}.
\newblock Accessed on 1st March 2021.

\bibitem[{Fernández-Muñoz et~al.(2020)Fernández-Muñoz, Pérez-Díaz,
  Guisández, Chazarra, and Fernández-Espina}]{FERNANDEZMUNOZ2020109662}
Fernández-Muñoz, D., J.~I. Pérez-Díaz, I.~Guisández, M.~Chazarra, Á.
  Fernández-Espina. 2020.
\newblock Fast frequency control ancillary services: An international review.
\newblock {\it Renewable and Sustainable Energy Reviews\/} {\bf 120} 109662.

\bibitem[{Fodstad et~al.(2018)Fodstad, Aarlott, and Midthun}]{Fodstad18}
Fodstad, M., M.~Aarlott, K.~Midthun. 2018.
\newblock Value-creation potential from multi-market trading for a hydropower
  producer.
\newblock {\it Energies\/} {\bf 11}(1) 1--15.

\bibitem[{Gauvin et~al.(2017)Gauvin, Delage, and Gendreau}]{Delage17}
Gauvin, C., E.~Delage, M.~Gendreau. 2017.
\newblock Decision rule approximations for the risk averse reservoir management
  problem.
\newblock {\it European Journal of Operational Research\/} {\bf 261}(1)
  317--336.

\bibitem[{Gauvin et~al.(2018)Gauvin, Delage, and Gendreau}]{Delage18}
Gauvin, C., E.~Delage, M.~Gendreau. 2018.
\newblock A stochastic program with time series and affine decision rules for
  the reservoir management problem.
\newblock {\it European Journal of Operational Research\/} {\bf 267}(2)
  716--732.

\bibitem[{Gjelsvik et~al.(2010)Gjelsvik, Mo, and Haugstad}]{Gjelsvik10}
Gjelsvik, A., B.~Mo, A.~Haugstad. 2010.
\newblock Long- and medium-term operations planning and stochastic modelling in
  hydro-dominated power systems based on stochastic dual dynamic programming.
\newblock P.. Pardalos, S.~Rebennack, M.~Pereira, N.~Iliadis, eds., {\it
  Handbook of Power Systems I\/}. Springer, 33--55.

\bibitem[{Hanasusanto et~al.(2016)Hanasusanto, Kuhn, and Wiesemann}]{Grani15}
Hanasusanto, G., D.~Kuhn, W.~Wiesemann. 2016.
\newblock A comment on ``{C}omputational complexity of stochastic programming
  problems''.
\newblock {\it Mathematical Programming\/} {\bf 159} 557--569.

\bibitem[{Helseth et~al.(2017)Helseth, Fodstad, Askeland, Mo, {Bjarte Nilsen},
  Pérez‐Díaz, Chazarra, and Guisández}]{Helseth17}
Helseth, A., M.~Fodstad, M.~Askeland, B.~Mo, O.~{Bjarte Nilsen}, J.~I.
  Pérez‐Díaz, M.~Chazarra, I.~Guisández. 2017.
\newblock Assessing hydropower operational profitability considering energy and
  reserve markets.
\newblock {\it {IET} Renewable Power Generation\/} {\bf 11}(13) 1640--1647.

\bibitem[{Helseth et~al.(2016)Helseth, Fodstad, and Mo}]{Helseth16}
Helseth, A., M.~Fodstad, B.~Mo. 2016.
\newblock Optimal medium-term hydropower scheduling considering energy and
  reserve capacity markets.
\newblock {\it IEEE Transactions on Sustainable Energy\/} {\bf 7}(3) 934--942.

\bibitem[{Helseth et~al.(2015)Helseth, Mo, Fodstad, and Hjelmeland}]{Helseth15}
Helseth, A., B.~Mo, M.~Fodstad, M.~Hjelmeland. 2015.
\newblock Co-optimizing sales of energy and capacity in a hydropower scheduling
  model.
\newblock {\it Proceedings of IEEE Eindhoven PowerTech\/}. 1--6.

\bibitem[{{IEA}(2020{\natexlab{a}})}]{IEA20:explorer}
{IEA}. 2020{\natexlab{a}}.
\newblock Renewables 2020 data explorer.
\newblock \url{https://www.iea.org/articles/renewables-2020-data-}
  \url{explorer?mode=market&region=World&product=Hydro}.
\newblock Accessed on 1st March 2021.

\bibitem[{{IEA}(2020{\natexlab{b}})}]{IEA20:hydropower}
{IEA}. 2020{\natexlab{b}}.
\newblock Renewables 2020 fuel report.
\newblock \url{https://www.iea.org/reports/renewables-2020/hydropower-}
  \url{bioenergy-csp-and-geothermal#hydropower}.
\newblock Accessed on 1st March 2021.

\bibitem[{Klæboe et~al.(2019)Klæboe, Braathen, Eriksrud, and
  Fleten}]{Klaeboe19}
Klæboe, G., J.~Braathen, A.~Eriksrud, S.-E. Fleten. 2019.
\newblock Day-ahead market bidding taking the balancing power market into
  account.
\newblock {\it Available on {SSRN} \hspace{-2mm}\/} .

\bibitem[{L{\"o}hndorf et~al.(2013)L{\"o}hndorf, Wozabal, and
  Minner}]{Wozabal13}
L{\"o}hndorf, N., D.~Wozabal, S.~Minner. 2013.
\newblock Optimizing trading decisions for hydro storage systems using
  approximate dual dynamic programming.
\newblock {\it Operations Research\/} {\bf 61}(4) 810--823.

\bibitem[{Lorca et~al.(2016)Lorca, Sun, Litvinov, and Zheng}]{Lorca16}
Lorca, A., X.~A. Sun, E.~Litvinov, T.~Zheng. 2016.
\newblock Multistage adaptive robust optimization for the unit commitment
  problem.
\newblock {\it Operations Research\/} {\bf 64}(1) 32--51.

\bibitem[{Mayer(2014)}]{Fraunhofer}
Mayer, J. 2014.
\newblock Electricity production and spot-prices in {G}ermany 2014.
\newblock Fraunhofer Institute for Solar Energy Systems.

\bibitem[{Morris and Pehnt(2015)}]{Energytransition}
Morris, C., M.~Pehnt. 2015.
\newblock Energy transition: The {G}erman {E}nergiewende.
\newblock Heinrich B\"{o}ll Foundation.

\bibitem[{Pan et~al.(2015)Pan, Housh, Liu, Cai, and Chen}]{Pan15}
Pan, L., M.~Housh, P.~Liu, X.~Cai, X.~Chen. 2015.
\newblock Robust stochastic optimization for reservoir operation.
\newblock {\it Water Resources Research\/} {\bf 51}(1) 409--429.

\bibitem[{Pereira and Pinto(1991)}]{Pereira91}
Pereira, M., L.~Pinto. 1991.
\newblock Multi-stage stochastic optimization applied to energy planning.
\newblock {\it Mathematical Programming\/} {\bf 52} 359--375.

\bibitem[{Philpott and {de Matos}(2012)}]{Philpott12a}
Philpott, A., V.~{de Matos}. 2012.
\newblock Dynamic sampling algorithms for multi-stage stochastic programs with
  risk aversion.
\newblock {\it European Journal of Operational Research\/} {\bf 218}(2)
  470--483.

\bibitem[{Philpott et~al.(2018)Philpott, de~Matos, and Kapelevich}]{Philpott18}
Philpott, A., V.~de~Matos, L.~Kapelevich. 2018.
\newblock Distributionally robust {SDDP}.
\newblock {\it Computational Management Science\/} {\bf 15} 431--454.

\bibitem[{Press et~al.(2007)Press, Teukolsky, Vetterling, and
  Flannery}]{Press07}
Press, W., S.~Teukolsky, W.~Vetterling, B.~Flannery. 2007.
\newblock {\it Numerical Recipes 3rd Edition: The Art of Scientific
  Computing\/}.
\newblock Cambridge University Press.

\bibitem[{Pritchard et~al.(2005)Pritchard, Philpott, and Neame}]{Pritchard05}
Pritchard, G., A.~Philpott, P.~Neame. 2005.
\newblock Hydroelectric reservoir optimization in a pool market.
\newblock {\it Mathematical Programming\/} {\bf 103}(3) 445--461.

\bibitem[{Rockafellar and Wets(2010)}]{ref:Rockafellar-10}
Rockafellar, R.~T., Roger J.-B. Wets. 2010.
\newblock {\it Variational Analysis\/}.
\newblock Springer.

\bibitem[{{Schillinger} et~al.(2017){Schillinger}, {Weigt}, {Barry}, and
  {Schumann}}]{Schillinger17}
{Schillinger}, M., H.~{Weigt}, M.~{Barry}, R.~{Schumann}. 2017.
\newblock Hydropower operation in a changing environment -- a {S}wiss case
  study.
\newblock {\it Proceedings of the 14th International Conference on the European
  Energy Market\/}. 1--6.

\bibitem[{Shapiro et~al.(2009)Shapiro, Dentcheva, and Ruszczynski}]{Shapiro09}
Shapiro, A., D.~Dentcheva, A.~Ruszczynski. 2009.
\newblock {\it {Lectures on Stochastic Programming: Modeling and Theory}\/}.
\newblock SIAM.

\bibitem[{Shapiro et~al.(2013)Shapiro, Tekaya, Soares, and
  da~Costa}]{Shapiro13}
Shapiro, A., W.~Tekaya, M.~Soares, J.~da~Costa. 2013.
\newblock Worst-case-expectation approach to optimization under uncertainty.
\newblock {\it Operations Research\/} {\bf 61}(6) 1435--1449.

\bibitem[{Vayanos et~al.(2012)Vayanos, Kuhn, and Rustem}]{Vayanos12}
Vayanos, P., D.~Kuhn, B.~Rustem. 2012.
\newblock A constraint sampling approach for multi-stage robust optimization.
\newblock {\it Automatica\/} {\bf 48}(3) 459--471.

\bibitem[{Wirth(2016)}]{PVfact}
Wirth, H. 2016.
\newblock Recent facts about photovoltaics in {G}ermany.
\newblock Fraunhofer Institute for Solar Energy Systems.

\bibitem[{Yan{\i}ko\u{g}lu et~al.(2019)Yan{\i}ko\u{g}lu, Gorissen, and
  {den~Hertog}}]{YGH18:aro_survey}
Yan{\i}ko\u{g}lu, I., B.~L. Gorissen, D.~{den~Hertog}. 2019.
\newblock A survey of adjustable robust optimization.
\newblock {\it European Journal of Operational Research\/} {\bf 277}(3)
  799--813.

\end{thebibliography}
\bibliographystyle{ormsv080}

\clearpage

\begin{APPENDIX}{}
\section{Serial Independence of Reserve Activations}
\label{sec:independence}
\revision{
Throughout the paper we assume that the uncertain reserve activations~$\rho_t^{\rm{u}}$ and~$\rho_t^{\rm{v}}$ are serially independent across different hours~$t\in\mathcal T$. To further investigate this assumption, recall that secondary reserve is dispatched by the grid operator in five-second intervals. It is therefore expedient to model the reserve activations as hourly averages of more fine-grained reserve activation processes~$\rho^{\rm u}(s)$ and~$\rho^{\rm v}(s)$, where~$s\in\mathcal S$ indexes all five-second intervals within the planning horizon. Formally, we thus set
\begin{equation}
    \label{eq:averaged-rhos}
    \rho_t^{\rm u} = \frac{1}{\vert\mathcal{S}(t)\vert}\sum_{s \in \mathcal{S}(t)} \rho^{\rm u}(s) 
    \quad \text{and} \quad 
    \rho_t^{\rm v} = \frac{1}{\vert\mathcal{S}(t)\vert}\sum_{s \in \mathcal{S}(t)} \rho^{\rm v}(s) \quad \forall t \in \mathcal{T},
\end{equation}
where~$\mathcal S(t)$ denotes the set of all 720 five-second intervals in hour~$t$. In any given interval~$s\in\mathcal S$, there cannot be a simultaneous call-off on the reserve-up and reserve-down markets. This prompts us to model~$\rho^{\rm u}(s)$ and~$\rho^{\rm v}(s)$ as Bernoulli random variables, which evaluate to~0 if there is no call-off and to~1 if there is a call-off on the reserve-up or the reserve-down market, respectively, and to assume that~$\rho^{\rm u}(s)$ and~$\rho^{\rm v}(s)$ cannot evaluate to~1 at the same time. More precisely, we model the stochastic process~$\{(\rho^{\rm u}(s),\rho^{\rm u}(s))\}_{s\in\mathcal S}$ as a time-homogeneous Markov chain with three states: no call-off (state 1: $\rho^{\rm u}(s) = \rho^{\rm v}(s) = 0$), a call-off on the reserve-up market (state 2: $\rho^{\rm u}(s) = 1$ and $\rho^{\rm v}(s) = 0$), or a call-off on the reserve-down market (state 3: $\rho^{\rm u}(s) = 0$ and $\rho^{\rm v}(s) = 1$).
If we assume that the system cannot transition from a reserve-up call-off to a reserve-down call-off (or vice versa) in just five seconds and that call-offs on the reserve-up and the reserve-down markets are equally likely, then the transition probability matrix of the Markov chain can be represented as 
%
%
\begin{equation*}
    \mathbf{P} = \left[ \begin{array}{ccc}
         1-2p & p & p  \\
         q & 1-q & 0 \\
         q & 0 & 1-q
    \end{array}  \right],
\end{equation*}
where the rows and columns refer to the three states in order of appearance. Here, $p \in (0,\frac{1}{2})$ denotes the probability of a call-off on the reserve-up or the reserve-down market in the next \revision{five seconds} if there is no call-off at present. Similarly, $q \in (0,1)$ denotes the probability that an ongoing call-off on either market terminates over the next five seconds.
\revision{The imposed time homogeneity makes the simplifying assumption that the probability of entering a new reserve activation (\emph{i.e.}, the probability $p$ leaving state 1) as well as the probability of terminating a reserve activation (\emph{i.e.}, the probability $q$ of entering state 1) are independent of the recent history of reserve activations. Using this Markov chain, we will investigate below to which degree the averaged hourly reserve activations~$\rho_t^{\rm{u}}$ and $\rho_t^{\rm{v}}$ can be considered serially independent.} To estimate~$p$ and~$q$, we first note that the unique invariant distribution of the Markov chain encoded by~$\mathbf{P}$ is given by $\bm\pi = \frac{1}{q+2p}(q,p,p)$. Thus, the long-run fraction of time spent in reserve-up call-offs (or equivalently, reserve-down call-offs) is given by~$\varrho = \frac{p}{q+2p}$. In addition, an elementary calculation reveals that~$1/q$ coincides with the expected duration~$h$ of a call-off. Both~$\varrho$ and~$h$ \revision{can be} estimated from data, and they uniquely determine~$p$ and~$q$. In the following we assume that the initial state distribution is given by~$\bm\pi$. Thus, the Markov chain is in the stationary regime, and~$\mathbb E[\rho^{\rm u}(s)]=\mathbb E[\rho^{\rm v}(s)] = \varrho$ for all~$s\in\mathcal S$, which in turn implies via~\eqref{eq:averaged-rhos} that~$\hat\rho^{\rm u}_t =\mathbb E[\rho^{\rm u}_t] = \varrho$ and~$\hat\rho^{\rm v}_t =\mathbb E[\rho^{\rm v}_t] = \varrho$ for all~$t\in\mathcal T$. In this case, the dependence between the states of the Markov chain at time~$1$ and at time~$s$ can be measured by the uncertainty coefficient~$U(s)=I((\rho^{\rm u}(s),\rho^{\rm u}(s));(\rho^{\rm u}(1),\rho^{\rm u}(1)))/H((\rho^{\rm u}(s),\rho^{\rm u}(s)))$, where~$I(\cdot;\cdot)$ denotes mutual information, and~$H(\cdot)$ denotes entropy. The uncertainty coefficient ranges over~$[0,1]$ and characterizes the fraction of the bits of~$(\rho^{\rm u}(s),\rho^{\rm u}(s))$ that can be predicted from~$(\rho^{\rm u}(1),\rho^{\rm u}(1))$.
The smaller~$U(s)$, the more {\em in}dependent are the reserve activations at times~1 and~$s$; see \emph{e.g.} \citet[\S 14.7.4]{Press07}. Figure~\ref{fig:inddependence-of-activations} shows that~$U(s)$ decays rapidly with~$s$ for~$\varrho=1\%$ and for all values of~$h\in\{1\,\text{min}, 5\,\text{min}, 10\,\text{min}, 15\,\text{min}, 20\,\text{min} \}$, which implies that the reserve activations become increasingly independent as the lag~$s$ grows. In particular, for lags exceeding one hour ($s \geq 720$), the uncertainty coefficient~$U(s)$ virtually vanishes for all considered values of~$h$. Note that $15$ minutes is a natural upper bound on~$h$ \mbox{as longer-lasting imbalances trigger an activation of tertiary reserves. }


\begin{figure}[tb]
\begin{center}
    \includegraphics[width=0.55\paperwidth]{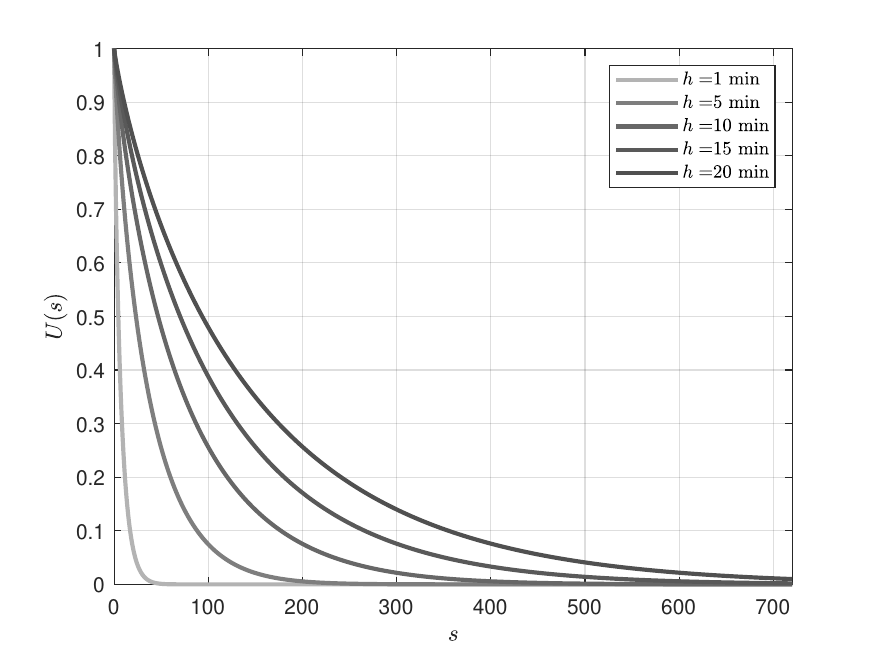}
\end{center}
	\caption{\revision{The uncertainty coefficient $U(s)$ as a function of time~$s$ measured in five-second intervals for different expected call durations $h$ when the frequency of reserve-up call-offs equals $\varrho = 1\%$. 
	}}
	\label{fig:inddependence-of-activations}
\end{figure}

Table~\ref{tab:mutual-information} reports the uncertainty coefficient between the hourly reserve activations~$(\rho^{\rm u}_1, \rho^{\rm v}_1)$ and $(\rho^{\rm u}_{t}, \rho^{\rm v}_{t})$ constructed as in~\eqref{eq:averaged-rhos} for~$t \in\{2,3,4\}$. Here, the uncertainty coefficient is computed via simulation using $10^6$~sample trajectories. Hence, consecutive hourly reserve activations are indeed nearly independent.


\begin{table}[tb]
\begin{center}
\revision{
\begin{tabular}{c|ccccc}
    & $h = 1$ min & $h = 5$ mins & $h = 10$ mins & $h = 15$ mins & $h = 20$ mins \\ \hline 
    $t = 2$ & $0.01\%$ & $0.18\%$  & $1.26\%$  & $2.99\%$ & $5.24\%$ \\
    $t = 3$ & $7 \times 10^{-3}\%$ & $0.03\%$ & $0.06\%$  & $0.08\%$ & $0.11\%$ \\
    $t = 4$ & $7 \times 10^{-3}\%$ & $0.02\%$ & $0.06\%$ & $0.07\%$ & $0.08\%$ \\ \hline \hline
\end{tabular}}
\end{center}
	\caption{\revision{The uncertainty coefficient between $(\rho^{\rm u}_1, \rho^{\rm v}_1)$ and $(\rho^{\rm u}_{t}, \rho^{\rm v}_{t})$ for different hourly time lags $t$ and expected call durations $h$ when the frequency of reserve-up call-offs equals $\varrho = 1\%$.}}
	\label{tab:mutual-information}
\end{table}}


\newpage

\newpage

\section{\revision{Proofs and Auxiliary Technical Results}}

\begin{proof}{Proof of Proposition~\ref{prop:collective_inequality_equivalence}}
As problem~\eqref{opt:collective_bidding_model_dailywater2} is a relaxation of problem~\eqref{opt:collective_bidding_model_dailywater}, it suffices to show that any feasible solution of problem~\eqref{opt:collective_bidding_model_dailywater2} corresponds to a feasible solution of problem~\eqref{opt:collective_bidding_model_dailywater} that attains the same objective function value. \revision{To this end, select any $\{ ( {s}_t, {u}_t, {v}_t, \bm{g}_t, \bm{p}_t, \bm{z}_t, \bm{w}_d)\}_{t,d}$ feasible in~\eqref{opt:collective_bidding_model_dailywater2}. As spillage is neither restricted nor penalized and as the reservoir topology represents a directed acyclic graph, any amount of water equal to the slack of the reformulated reservoir balance constraints in~\eqref{opt:collective_bidding_model_dailywater2} can be spilled through to node~$R$ sequentially for all reservoirs in the partial order induced by the topology graph and sequentially for all hours in the planning horizon. Thus, one can systematically replace the spillage decision~$\bm{z}_t$ with another $\mathcal{F}_{[t]}$-measurable random vector~$\bm{z}'_t$ for each~$t\in\mathcal T$ such that~$\{ ( {s}_t, {u}_t, {v}_t, \bm{g}_t, \bm{p}_t, \bm{z}'_t, \bm{w}_d)\}_{t,d}$ is feasible in problem~\eqref{opt:collective_bidding_model_dailywater}. 
The claim now follows because the spillage decisions do not enter the objective function.
}
\qed
\end{proof}

~\\[-10mm]

\revision{The proof of Proposition~\ref{prop:collective_bilevel_separation} relies on an auxiliary result which we state and prove first. To this end, consider} a generic multistage stochastic program of the form
\begin{equation}
\label{eq:msp}
\begin{aligned}
    &\sup && \mathbb{E} \left[ f(\bm{x}_0, \hdots, \bm{x}_T, \bm\xi_{[T]}) \right] \\
    &\mspace{3mu} \text{s.t.} &&  \bm{x}_0 \in \mathcal{L}^{n_0}(\mathcal{F}_0),~ \bm{x}_t \in \mathcal{L}^{n_t}(\mathcal{F}_{[t]}) && \forall t = 1, \hdots, T, 
\end{aligned}
\end{equation}
where the expectation is taken with respect to $\mathbb P$, the joint distribution of the random vectors $\bm \xi_t\in \mathbb{R}^{k_t}$, $t=0,\ldots,T$. Here, $f$ constitutes a normal integrand in the sense of \citet[Definition~14.27]{ref:Rockafellar-10}, that is, its epigraph $\text{epi}(f(\cdot,\bm\xi_{[T]}))$ is a closed-valued measurable multifunction of $\bm\xi_{[T]}$. 
We assume that $f$ may adopt the value $-\infty$ and thereby encode $\mathbb P$-almost sure constraints on the decision variables. We also retain the notational conventions for combined random vectors and $\sigma$-algebras introduced in Section~\ref{sec:models}. \revision{The following lemma relies on Theorem~14.60 by \citet{ref:Rockafellar-10} and the properties of conditional expectations.}

\begin{lemma}
\label{lem:maxexp2expmax}
The optimal value of the multistage stochastic program~\eqref{eq:msp} equals
\begin{equation}
\label{eq:msp-conditional}
\mathbb{E}
\left[
\begin{aligned}
    &\sup && \mathbb{E} \left[ f(\bm{x}_0, \hdots, \bm{x}_T, \bm\xi_{[T]}) \,\big\vert\, \bm\xi_0 \right] \\
    &\mspace{2mu} \emph{s.t.}&& \bm{x}_0 \in \mathbb{R}^{n_0},~ \bm{x}_t \in \mathcal{L}^{n_t}(\mathcal{F}_{[1,t]}) && \forall t = 1, \hdots, T 
\end{aligned}
\right].
\end{equation}
\end{lemma}

\begin{proof}{Proof} As joint optimization over the stage-wise decision variables is equivalent to sequential optimization, we can rewrite the multistage stochastic program~\eqref{eq:msp} as
\begin{equation*}
    \sup_{\bm{x}_0 \in \mathcal{L}^{n_0}(\mathcal{F}_0)} \ \sup_{\bm{x}_1 \in \mathcal{L}^{n_1}(\mathcal{F}_{[1]})} \ \cdots \ \sup_{\bm{x}_{T-1} \in \mathcal{L}^{n_{T-1}}(\mathcal{F}_{[T-1]}) } \ \sup_{\bm{x}_T \in \mathcal{L}^{n_T}(\mathcal{F}_{[T]})} \ \mathbb{E} \left[ f(\bm{x}_0, \ldots, \bm{x}_T, \bm\xi_{[T]}) \right].
\end{equation*}
By \citet[Theorem~14.60]{ref:Rockafellar-10}, the above problem is equivalent to
\begin{equation*}
    \sup_{\bm{x}_0 \in \mathcal{L}^{n_0}(\mathcal{F}_0)} \ \sup_{\bm{x}_1 \in \mathcal{L}^{n_1}(\mathcal{F}_{[1]})} \ \cdots \ \sup_{\bm{x}_{T-1} \in \mathcal{L}^{n_{T-1}}(\mathcal{F}_{[T-1]})} \ \mathbb{E} \left[ \sup_{\bm{x}_T \in \mathbb{R}^{n_T}} f(\bm{x}_0, \ldots, \bm{x}_T, \bm\xi_{[T]}) \right].
\end{equation*}
Recalling the tower property of conditional expectations, we can rewrite the resulting problem as
\begin{equation*}
    \sup_{\bm{x}_0 \in \mathcal{L}^{n_0}(\mathcal{F}_0)} \ \sup_{\bm{x}_1 \in \mathcal{L}^{n_1}(\mathcal{F}_{[1]})} \ \cdots \ \sup_{\bm{x}_{T-1} \in \mathcal{L}^{n_{T-1}}(\mathcal{F}_{[T-1]})} \ \mathbb{E} \left[ \mathbb{E} \left[ \sup_{\bm{x}_T \in \mathbb{R}^{n_T}} f(\bm{x}_0, \ldots, \bm{x}_T, \bm\xi_{[T]}) \,\Big\vert\, \bm\xi_{[T-1]} \right] \right]
\end{equation*}
and invoke the interchangeability theorem of \citet{ref:Rockafellar-10} once again to prove its equivalence to
\begin{equation*}
    \sup_{\bm{x}_0 \in \mathcal{L}^{n_0}(\mathcal{F}_0)} \ \sup_{\bm{x}_1 \in \mathcal{L}^{n_1}(\mathcal{F}_{[1]})} \ \cdots \ \mathbb{E}  \left[ \sup_{\bm{x}_{T-1} \in \mathbb{R}^{n_{T-1}}} \mathbb{E} \left[ \sup_{\bm{x}_T \in \mathbb{R}^{n_T}} f(\bm{x}_0, \ldots, \bm{x}_T, \bm\xi_{[T]}) \,\Big\vert\, \bm\xi_{[T-1]} \right] \right]
\end{equation*}
Repeating this argument, it is possible to move all supremum operators inside the appropriate conditional expectation layers such as to obtain the equivalent problem
\begin{equation*}
    \mathbb{E} \left[ \sup_{\bm{x}_0 \in \mathbb{R}^{n_0}} \mathbb{E} \left[ \sup_{\bm{x}_1 \in \mathbb{R}^{n_1}} \mathbb{E} \left[ \sup_{\bm{x}_2 \in \mathbb{R}^{n_2}} \cdots \ \mathbb{E} \left[\sup_{\bm{x}_T \in \mathbb{R}^{n_T}} f(\bm{x}_0, \ldots, \bm{x}_T, \bm\xi_{[T]}) \,\Big\vert\, \bm\xi_{[T-1]} \right] \cdots \,\Big\vert\, \bm\xi_{[1]} \right]\Big\vert\, \bm\xi_{0} \right] \right].
\end{equation*}
Applying the interchangeability theorem of \citet{ref:Rockafellar-10} in the reverse direction, we can then move the second supremum out of its conditional expectation layer to obtain
\begin{equation*}
    \mathbb{E} \left[ \sup_{\bm{x}_0 \in \mathbb{R}^{n_0}} \ \sup_{\bm{x}_1 \in \mathcal{L}^{n_1}(\mathcal{F}_{[1,1]})} \ \mathbb{E} \left[ \sup_{\bm{x}_2 \in \mathbb{R}^{n_2}} \cdots \ \mathbb{E} \left[\sup_{\bm{x}_T \in \mathbb{R}^{n_T}} f(\bm{x}_0, \ldots, \bm{x}_T, \bm\xi_{[T]}) \,\Big\vert\, \bm\xi_{[T-1]} \right] \cdots \, \Big\vert\, \bm\xi_{0} \right] \right].
\end{equation*}
Iterating this argument from the outside to the inside, it is indeed possible to move all supremum operators into the outermost (unconditional) expectation layer and thus obtain
\begin{equation*}
    \mathbb{E} \left[ \sup_{\bm{x}_0 \in \mathbb{R}^{n_0}} \ \sup_{\bm{x}_1 \in \mathcal{L}^{n_1}(\mathcal{F}_{[1,1]})} \ \sup_{\bm{x}_2 \in \mathcal{L}^{n_2}(\mathcal{F}_{[1,2]})} \cdots \sup_{\bm{x}_T \in \mathcal{L}^{n_T}(\mathcal{F}_{[1,T]})} \mathbb{E} \bigg[ f(\bm{x}_0, \ldots, \bm{x}_T, \bm\xi_{[T]}) \,\Big\vert\, \bm\xi_{0} \bigg] \right].
\end{equation*}
Re-combining the stage-wise supremum operators finally yields problem~\eqref{eq:msp-conditional}.
\qed
\end{proof}
~\\[-10mm]

\begin{proof}{Proof of Proposition~\ref{prop:collective_bilevel_separation}}
Observe that if the reservoir filling levels $\{ \bm{w}_d\}_{d \in \mathcal{D}}$ are fixed, then the remaining decisions in problem~\eqref{opt:collective_bidding_model_dailywater2} are no longer coupled across days. This allows us to decompose the stochastic program~\eqref{opt:collective_bidding_model_dailywater2} into an outer maximization problem over the end-of-day reservoir filling levels $\{ \bm{w}_d\}_{d \in \mathcal{D}}$ and a series of (mutually independent) inner stochastic programs maximizing over the bidding decisions $\{ ( {s}_t, {u}_t, {v}_t ) \}_{t \in \mathcal{T}(d)}$ and the operational decisions $\{ (\bm{g}_t, \bm{p}_t, \bm{z}_t )\}_{t \in \mathcal{T}(d)}$, one for each day $d\in\mathcal D$. Formally, problem~\eqref{opt:collective_bidding_model_dailywater2} is thus equivalent to the outer stochastic program
\begin{equation*}
\begin{array}{c@{~~~}l@{\,}l}
    \text{sup} & \displaystyle \sum_{d \in \mathcal{D}}  \hat \Pi_d( \bm w_{d-1}, \bm w_{d}) \\[1ex]
    \text{s.t.} & \bm{w}_d \in \mathcal{L}^R(\mathcal{F}_{[\uparrow(d)]}) & \forall d \in \mathcal{D} \\
    & \underline{\bm{w}}_d \leq \bm{w}_d \leq \overline{\bm{w}}_d & \forall d \in \mathcal{D},~ \text{$\mathbb{P}$-a.s.},
    \end{array}
\end{equation*}
where $\hat \Pi_d( \bm w_{d-1}, \bm w_{d})$ stands for the optimal value of the inner (parametric) stochastic program
\begin{equation}
\label{opt:trader-preliminary}
    \begin{array}{c@{~~~}ll} \text{sup} & \displaystyle \mathbb{E} \Big[ \sum_{t \in \mathcal{T}(d)} \pi_t^{\rm{s}}{s}_t + \revision{\kappa^{\rm u}_t}(\pi_t^{\rm{u}} + \rho_t^{\rm{u}} \psi_t^{\rm{u}}) u_t + \revision{\kappa^{\rm v}_t}(\pi_t^{\rm{v}} + \rho_t^{\rm{v}} \psi_t^{\rm{v}}) v_t \Big] \\[3mm]
    \text{s.t.} & {s}_t, {u}_t, {v}_t \in \mathcal{L}(\mathcal{F}_{[\Downarrow(d)]}), ~\bm{g}_t, \bm{p}_t, \bm{z}_t \in \mathcal{L}^A(\mathcal{F}_{[t]}) & \hspace{-7mm} \forall t \in \mathcal{T}(d) \\[3mm]
    & {0} \leq {u}_t, ~{0} \leq {v}_t, ~\bm{0} \leq \bm{g}_t \leq \overline{\bm{g}}_t, ~\bm{0} \leq \bm{p}_t \leq \overline{\bm{p}}_t, ~\bm{0} \leq \bm{z}_t & \hspace{-7mm} \forall t \in \mathcal{T}(d),~\text{$\mathbb{P}$-a.s.}\\[3mm]
    & {s}_t + \revision{\kappa^{\rm u}_t}\rho_t^{\rm{u}} {u}_t- \revision{\kappa^{\rm v}_t}\rho_t^{\rm{v}} {v}_t = \bm{\eta}_t^\top \bm{g}_t - \bm{\zeta}_t^\top \bm{p}_t & \hspace{-7mm} \forall t \in \mathcal{T}(d),~\text{$\mathbb{P}$-a.s.}\\
    & \displaystyle \underline{\bm{w}}_t \leq \bm{w}_{d-1} + \sum_{\tau=\downarrow(d)}^t \bm{\phi}_\tau + \mathbf{M}(\bm{g}_\tau - \bm{p}_\tau + \bm{z}_\tau) \leq \overline{\bm{w}}_t & \hspace{-7mm} \forall t \in \mathcal{T}(d),~\text{$\mathbb{P}$-a.s.}\\[3mm]
    & \displaystyle \bm{w}_d \leq \bm{w}_{d-1} + \sum_{\tau \in \mathcal{T}(d)} \bm{\phi}_\tau + \mathbf{M}(\bm{g}_\tau - \bm{p}_\tau + \bm{z}_\tau) & \hspace{-7mm} \text{$\mathbb{P}$-a.s.},
    \end{array}
\end{equation}
which optimizes only over decisions pertaining to day $d$. Observe that problem~\eqref{opt:trader-preliminary} differs from the trader's problem~\eqref{opt:collective_trading} in that it maximizes an {\em un}conditional expectation of the profits earned on day $d$ and in that its decision variables may still adapt to information revealed {\em prior} to day~$d$.

We will now use Lemma~\ref{lem:maxexp2expmax} to show that $\hat \Pi_d( \bm w_{d-1}, \bm w_{d})$ coincides with the expected value of $\Pi_d( \bm w_{d-1}, \bm w_{d}, \bm\xi_{[\Downarrow(d)]})$. For notational convenience, we abbreviate the random variables revealed before the beginning of day $d$ as $\bm{\xi}_{0}^d = \bm{\xi}_{[\Downarrow(d)]}$ and those revealed up to each hour $h= 1, \ldots, H$ of day $d$ by $\bm{\xi}_{[h]}^d = \bm{\xi}_{[\Downarrow(d)+h]}$. Furthermore, we gather the hourly decisions of day $d$ in the vectors
\begin{equation*}
\begin{aligned}
& \bm{x}_0^d = ( \{ {s}_\tau, {u}_\tau, {v}_\tau \}_{\tau=\downarrow(d)}^{\uparrow(d)} ) \in \mathcal{L}^{3 H}(\mathcal{F}_{[\Downarrow(d)]}),\\
& \bm{x}_h^d = (\bm{g}_{\Downarrow(d)+h}, \bm{p}_{\Downarrow(d)+h}, \bm{z}_{\Downarrow(d)+h}) \in \mathcal{L}^{3 A}(\mathcal{F}_{[\Downarrow(d)+h]}) \hspace{2mm} &&\forall h = 1,\ldots,H.
\end{aligned}
\end{equation*}
Using this notation, the profit earned on day $d$ can be expressed more concisely as
\begin{equation*}
f^d(\bm{x}_0^d,\ldots, \bm{x}_H^d, \bm\xi_{[H]}^d) = \left\{ \begin{array}{l@{\,}l} \displaystyle \sum_{t \in \mathcal{T}(d)} \pi_t^{\rm{s}} s_t + \revision{\kappa^{\rm u}_t}(\pi_t^{\rm{u}} & + \rho_t^{\rm{u}} \psi_t^{\rm{u}}) u_t + \revision{\kappa^{\rm v}_t}(\pi_t^{\rm{v}} + \rho_t^{\rm{v}} \psi_t^{\rm{v}}) v_t  \\[1ex]
& \text{if $(\bm{x}_0^d,\ldots, \bm{x}_H^d)$ satisfies all $\mathbb P$-almost sure} \\ & \text{constraints of problem~\eqref{opt:trader-preliminary} in scenario $\bm\xi_{[H]}^d$},\\[2.2ex]
-\infty & \text{otherwise.}
\end{array} \right.
\end{equation*}
Problem~\eqref{opt:trader-preliminary} can thus be represented abstractly as
\begin{equation*}
    \hat \Pi_d( \bm w_{d-1}, \bm w_{d})=
    \left\{ \begin{array}{lll} \text{sup} & \mathbb{E} \big[ f^d(\bm{x}_0^d,\ldots, \bm{x}_H^d, \bm\xi_{[H]}^d) \big] \\[3mm]
    \text{s.t.} & \bm{x}_0^d \in \mathcal{L}^{3 H}(\mathcal{F}_{[\Downarrow(d)]}), ~\bm{x}_h^d \in \mathcal{L}^{3A}(\mathcal{F}_{[\Downarrow(d)+h]}) 
    & \hspace{2mm}\forall h = 1, \ldots, H.
    \end{array}\right.
\end{equation*}
By the extended interchangeability principle established in Lemma~\ref{lem:maxexp2expmax}, we thus find 
\begin{equation*}
    \hat \Pi_d( \bm w_{d-1}, \bm w_{d})=
    \mathbb{E} \left[ \begin{array}{lll} \text{sup} & \mathbb{E} \big[ f^d(\bm{x}_0^d,\ldots, \bm{x}_H^d, \bm\xi_{[H]}^d)  \,\vert\, \bm\xi_{0}^d \big] \\[3mm]
    \text{s.t.} & \bm{x}_0^d \in \mathbb{R}^{3 H}, ~\bm{x}_h^d \in \mathcal{L}^{3 A}(\mathcal{F}_{[\downarrow(d),\Downarrow(d)+h]}) & \hspace{2mm}\forall h = 1, \ldots, H
    \end{array} \right] .
\end{equation*}
Unravelling the abbreviations shows that the stochastic program inside the expectation coincides with the trader's problem~\eqref{opt:collective_trading}, whose optimal value is given by \revision{$\Pi_d( \bm w_{d-1}, \bm w_{d}, \bm\xi_{[\Downarrow(d)]})$}.
\qed
\end{proof}

~\\[-10mm]

\begin{proof}{Proof of Proposition~\ref{prop:rho_restriction_collective}}
Throughout this proof we will fix a realization of the contextual covariates~$\bm\xi_{[\Downarrow(d)]}$, which reflect the trader's information when solving problem~\eqref{opt:collective_trading} for day $d$. Approximation~\ref{apx:day_ahead_water_level} then implies that $\bm w_{d-1}$ and $\bm w_d$ reduce to deterministic constants. In addition, $\mathbb{P}_{\vert \bm\xi_{[\Downarrow(d)]}}$ becomes an {\em un}conditional distribution. Below we will denote the expectation with respect to this distribution by $\mathbb{E}_{\vert\bm\xi_{[\Downarrow(d)]}}[\cdot]$.

As problem~\eqref{opt:collective_trading_rho} is a restriction of problem~\eqref{opt:collective_trading}, it suffices to show that for every feasible solution $\{  (s_t, u_t, v_t, \bm{g}_t, \bm{p}_t, \bm{z}_t) \}_{t \in \mathcal{T}(d)}$ of \eqref{opt:collective_trading} there is a feasible solution $\{  (s_t^\prime, u_t^\prime, v_t^\prime, \bm{g}_t^\prime, \bm{p}_t^\prime, \bm{z}_t^\prime) \}_{t \in \mathcal{T}(d)}$ of \eqref{opt:individual_trading_rho} that attains the same objective value. Such a solution can readily be constructed as
\begin{equation*}
\begin{aligned}
(s_t^\prime, u_t^\prime, v_t^\prime) = (s_t, u_t, v_t) \quad \text{and} \quad (\bm{g}_t^\prime, \bm{p}_t^\prime, \bm{z}_t^\prime) = \mathbb{E}_{\vert\bm\xi_{[\Downarrow(d)]}} [ (\bm{g}_t, \bm{p}_t, \bm{z}_t) \,\vert\, \revision{\mathcal{F}_{[\downarrow(d),t]}^\rho}] \quad \forall t \in \mathcal{T}(d).
\end{aligned}
\end{equation*}
The equality of objective values is immediate because the objective function only depends on the bidding decisions, which are preserved. The non-anticipativity constraints $\bm{g}_t^\prime, \bm{p}_t^\prime, \bm{z}_t^\prime \in \mathcal{L}^A(\mathcal{F}_{[\downarrow(d),t]}^\rho)$ are also satisfied thanks to the defining properties of conditional expectations. To show that the almost sure constraints hold, we recall first that the reserve activations are serially independent and independent of all other sources of uncertainty. This implies that \revision{$\{(\rho_\tau^{\rm{u}}, \rho_\tau^{\rm{v}}) \}_{\tau = t+1}^{\uparrow(d)}$}  is independent of the non-anticipative flow decisions $\bm{g}_t, \bm{p}_t, \bm{z}_t\in \mathcal{L}^A(\mathcal{F}_{[\downarrow(d),t]} )$ under the distribution $\mathbb{P}_{\vert \bm\xi_{[\Downarrow(d)]}}$, that is 
\begin{equation}
\label{eq:rho-independence}
(\bm{g}_t^\prime, \bm{p}_t^\prime, \bm{z}_t^\prime) = \mathbb{E}_{\vert \bm\xi_{[\Downarrow(d)]}} [ (\bm{g}_t, \bm{p}_t, \bm{z}_t) \,\vert\, \revision{\mathcal{F}_{[\downarrow(d),\uparrow(d)]}^\rho}] \quad \mathbb{P}_{\vert \bm\xi_{[\Downarrow(d)]}}\text{-a.s.}
\end{equation}
for all $t \in \mathcal{T}(d)$. The feasibility of $\{  ({s}_t^\prime, {u}_t^\prime, {v}_t^\prime, \bm{g}_t^\prime, \bm{p}_t^\prime, \bm{z}_t^\prime) \}_{t \in \mathcal{T}(d)}$ in \eqref{opt:collective_trading_rho} therefore follows from the feasibility of $\{  ({s}_t, {u}_t, {v}_t, \bm{g}_t, \bm{p}_t, \bm{z}_t) \}_{t \in \mathcal{T}(d)}$ in~\eqref{opt:collective_trading} and the linearity of the almost sure constraints in~\eqref{opt:collective_trading_rho},~which implies that they all remain valid under conditional expectations. For example, the newly constructed solution $\{  ({s}_t^\prime, {u}_t^\prime, {v}_t^\prime, \bm{g}_t^\prime, \bm{p}_t^\prime, \bm{z}_t^\prime) \}_{t \in \mathcal{T}(d)}$ obeys the end-of-day reservoir bound in \eqref{opt:individual_trading_rho} because
\begin{align*}
	\bm w_d=&~ \textstyle  \mathbb{E}_{\vert \bm\xi_{[\Downarrow(d)]}} [ \bm w_d \,\vert\, \revision{\mathcal{F}_{[\downarrow(d),\uparrow(d)]}^\rho}] \\
	\leq &~ \textstyle \mathbb{E}_{\vert \bm\xi_{[\Downarrow(d)]}} [ \bm{w}_{d-1} + \sum_{\tau \in \mathcal T(d)}\bm{\phi}_\tau + \mathbf{M} 
	( \bm{g}_\tau - \bm{p}_\tau + \bm{z}_\tau ) \,\vert\, \revision{\mathcal{F}_{[\downarrow(d),\uparrow(d)]}^\rho}]  \\
	=& ~ \textstyle \bm{w}_{d-1} + \sum_{\tau \in \mathcal T(d)} \bm{\phi}_\tau + \mathbf{M} 
	( \bm{g}_\tau^\prime - \bm{p}^\prime_\tau + \bm{z}^\prime_\tau ) \quad \forall t\in\mathcal T(d),~\mathbb{P}_{\vert \bm\xi_{[\Downarrow(d)]}}\text{a.s.},
\end{align*}
where the first equality follows from Approximation~\ref{apx:day_ahead_water_level}, whereby $\bm w_d$ is deterministic when $\bm\xi_{[\Downarrow(d)]}$ is kept fixed, the inequality holds because $\{  ({s}_t, {u}_t, {v}_t, \bm{g}_t, \bm{p}_t, \bm{z}_t) \}_{t \in \mathcal{T}(d)}$ satisfies the last constraint of~\eqref{opt:collective_trading}, and the last equality follows from~\eqref{eq:rho-independence}. Thus, the claim follows.
\qed
\end{proof}

~\\[-10mm]

\begin{proof}{Proof of Proposition~\ref{prop:down_cut_collective}}
Any feasible solution of the stochastic program~\eqref{opt:collective_trading_rho} satisfies the energy delivery constraint ${s}_t + \revision{\kappa^{\rm u}_t}\rho^{\rm u}_t {u}_t - \revision{\kappa^{\rm v}_t}\rho^{\rm v}_t {v}_t  = \bm\eta_t^\top \bm{g}_t - \bm\zeta_t^\top \bm{p}_t$ for all $t \in \mathcal{T}(d)$ $\mathbb{P}_{\vert \bm\xi_{[\Downarrow(d)]}}$-almost surely. Similarly, for any $t \in \mathcal{T}(d)$, we have that $\bm{g}_t \geq \bm{0}$ and $\bm{p}_t \leq \overline{\bm{p}}_t$ $\mathbb{P}_{\vert \bm\xi_{[\Downarrow(d)]}}$-almost surely. Together, these inequalities imply 
\begin{equation*}
\begin{aligned}
{s}_t + \revision{\kappa^{\rm u}_t}\rho^{\rm u}_t {u}_t - \revision{\kappa^{\rm v}_t}\rho^{\rm v}_t {v}_t  \; = \; \bm\eta_t^\top \bm{g}_t - \bm\zeta_t^\top \bm{p}_t \; \geq \; - \bm\zeta_t^\top \bm{p}_t \; \geq \; - \bm\zeta_t^\top \overline{\bm{p}}_t \quad \forall t \in \mathcal{T}(d), ~ \text{$\mathbb{P}_{\vert \bm\xi_{[\Downarrow(d)]}}$-a.s.}
\end{aligned}
\end{equation*}
\revision{As the scenario $(0, 1, 0,\overline{\rho}^{\rm{v}}_t)$ belongs to the support of the random vector~$(\kappa_t^{\rm u}, \kappa_t^{\rm v}, \rho_t^{\rm u},\rho_t^{\rm{v}})$ under the conditional probability measure~$\mathbb{P}_{\vert \bm\xi_{[\Downarrow(d)]}}$, we thus find~${s}_t - \overline{\rho}^{\rm{v}}_t {v}_t \geq - \bm\zeta_t^\top \overline{\bm{p}}_t$ for every~$t \in \mathcal{T}(d)$.} \qed
\end{proof}

~\\[-8mm]

\begin{proof}{Proof of Proposition~\ref{prop:robusttodeterministic_collective1}}
We will prove the proposition by showing that every feasible solution of problem~\eqref{opt:collective_trading_cuts} corresponds to a feasible solution of problem~\eqref{opt:reduced_collective_trading} with the same objective value. To this end, fix any feasible solution $\left\{ ({s}_t, {u}_t, {v}_t, \bm{g}_t, \bm{p}_t, \bm{z}_t) \right\}_{t \in \mathcal{T}(d)}$ of problem~\eqref{opt:collective_trading_cuts} and an arbitrary realization of the contextual covariates $\bm\xi_{[\Downarrow(d)]}$. \revision{Since all non-dummy reservoirs and pumps have finite capacity, 
there is~$\overline{\bm z}\in\mathbb R^A$ such that~$\bm z_t\leq \overline{\bm z}$ $\mathbb{P}_{\vert \bm\xi_{[\Downarrow(d)]}}$-almost surely for all~$t\in\mathcal T(d)$. 
Then, there exists an event $\Omega_0\in \mathcal F$ with $\mathbb{P}_{\vert \bm\xi_{[\Downarrow(d)]}}[\Omega_0]=1$ such that all constraints of problem~\eqref{opt:collective_trading_cuts} are satisfied pointwise for all $\omega\in\Omega_0$. In addition, by our assumptions about the statistics of the reserve allocations and activations, there exists a sequence of samples~$\omega_n\in\Omega_0$, $n\in\mathbb N$, such that 
\[
    \lim_{n\rightarrow\infty}\kappa^{\rm u}_t(\omega_n) = 1,\quad \lim_{n\rightarrow\infty}\rho_t^{\rm{u}}(\omega_n) = \overline{\rho}^{\rm{u}}_t \quad \text{and}\quad \lim_{n\rightarrow\infty}\rho_t^{\rm{v}}(\omega_n) =0\quad \forall t \in \mathcal{T}(d). 
\]
Note that the sequence $\{(\bm{g}_t(\omega_n), \bm{p}_t(\omega_n), \bm{z}_t(\omega_n))\}_{n\in\mathbb N}$ is bounded. By passing to a sub-sequence if necessary, we may thus assume that this sequence converges for every $t \in \mathcal{T}(d)$, and we may define
\[
    \bm{g}_t'= \lim_{n\rightarrow\infty}\bm{g}_t(\omega_n),\quad \bm{p}_t'=\lim_{n\rightarrow\infty}\bm{p}_t(\omega_n) \quad \text{and}\quad \bm{z}_t'= \lim_{n\rightarrow\infty}\bm{z}_t(\omega_n) \quad \forall t \in \mathcal{T}(d). 
\]
One readily verifies that $\left\{ ({s}_t, {u}_t, {v}_t, \bm{g}_t', \bm{p}_t', \bm{z}_t') \right\}_{t \in \mathcal{T}(d)}$ is feasible in~\eqref{opt:reduced_collective_trading} because the constraints of the stochastic program~\eqref{opt:collective_trading_cuts} are continuous in the uncertain parameters and the uncertain flow decisions. 
Also, the objective value of $\left\{ ({s}_t, {u}_t, {v}_t, \bm{g}_t', \bm{p}_t', \bm{z}_t') \right\}_{t \in \mathcal{T}(d)}$ in~\eqref{opt:reduced_collective_trading} coincides with that of $\left\{ ({s}_t, {u}_t, {v}_t, \bm{g}_t, \bm{p}_t, \bm{z}_t) \right\}_{t \in \mathcal{T}(d)}$ in~\eqref{opt:collective_trading_cuts} because the two solutions involve identical market bids. The claim thus follows.
}
\qed
\end{proof}

~\\[-10mm]

\begin{proof}{Proof of Theorem~\ref{thm:all_collective}}
\revision{This is an immediate consequence of Propositions~\ref{prop:rho_restriction_collective}, \ref{prop:down_cut_collective} and~\ref{prop:robusttodeterministic_collective1}, as well as Proposition~\ref{prop:robusttodeterministic_collective2} from Appendix~\ref{sec:reduction_individual}.} \qed 
\end{proof}

~\\[-8mm]

\begin{proof}{Proof of Theorem~\ref{thm:all_collective_bidding}}
From Propositions~\ref{prop:collective_inequality_equivalence} and~\ref{prop:collective_bilevel_separation} we know that the bidding model~\eqref{opt:collective_bidding_model} has the same optimal value as the planner's problem~\eqref{opt:collective_bidding_model_dailywater_sep}, and it is easy to see that problem~\eqref{opt:reduced_collective_planning} is obtained by applying Approximation~\ref{apx:day_ahead_water_level} to problem~\eqref{opt:collective_bidding_model_dailywater_sep}. As this approximation consists in restricting the planner's information structure, the optimal value of problem~\eqref{opt:reduced_collective_planning} is no larger than that of problem~\eqref{opt:collective_bidding_model}. It remains to be shown that~\eqref{opt:reduced_collective_planning} and \eqref{opt:collective_bidding_model_dailywater_sep_apx} share the same optimal value. As Approximation~\ref{apx:day_ahead_water_level} is in force, we may conclude via Theorem~\ref{thm:all_collective} that the optimal value $\Pi_d( \bm w_{d-1}, \bm w_{d}, \bm\xi_{[\Downarrow(d)]})$ of the stochastic program~\eqref{opt:collective_trading} coincides with the optimal value of the linear program~\eqref{opt:reduced_collective_trading}. Substituting~\eqref{opt:reduced_collective_trading} into~\eqref{opt:reduced_collective_planning} and using Theorem~14.60 by \citet{ref:Rockafellar-10} to move the maximization over the trading and flow decisions out of the expectation and the sum finally yields~\eqref{opt:collective_bidding_model_dailywater_sep_apx}. \qed
\end{proof}

\newpage

\section{\revision{Proof that Problem~\eqref{opt:collective_trading} is a Restriction of Problem~\eqref{opt:reduced_collective_trading}}}
\label{sec:reduction_individual}

\revision{In this appendix we prove that problem~\eqref{opt:collective_trading} is a restriction of problem~\eqref{opt:reduced_collective_trading}, that is, that the arc from the node labelled~\eqref{opt:reduced_collective_trading} to the node labelled~\eqref{opt:collective_trading} in Figure~\ref{fig:proofplan} is justified. To this end, we introduce a variant of problem~\eqref{opt:collective_trading} that we henceforth refer to as the \emph{individual} trader's problem.}

\begin{equation*}
\tag{IT}
\label{opt:individual_trading}
\begin{aligned}
\begin{array}{c@{~~~}ll} \text{sup} & \displaystyle \mathbb{E} \Big[ \sum_{t \in \mathcal{T}(d)} \pi_t^{\rm{s}} \, \mathbf{1}^\top \bm{s}_t + \revision{\kappa^{\rm u}_t}(\pi_t^{\rm{u}} + \rho_t^{\rm{u}} \psi_t^{\rm{u}}) \, \mathbf{1}^\top \bm{u}_t + \revision{\kappa^{\rm v}_t}(\pi_t^{\rm{v}} + \rho_t^{\rm{v}} \psi_t^{\rm{v}}) \, \mathbf{1}^\top \bm{v}_t \,\big\vert\, \bm\xi_{[\Downarrow(d)]} \Big] \\[3mm]
    \text{s.t.} & \bm{s}_t, \bm{u}_t, \bm{v}_t \in \mathbb{R}^A, ~\bm{g}_t, \bm{p}_t, \bm{z}_t \in \mathcal{L}^A(\mathcal{F}_{[\downarrow(d),t]}) & \hspace{-20mm} \forall t \in \mathcal{T}(d) \\[3mm]
    & \bm{0} \leq \bm{u}_t, ~\bm{0} \leq \bm{v}_t, ~\bm{0} \leq \bm{g}_t \leq \overline{\bm{g}}_t, ~\bm{0} \leq \bm{p}_t \leq \overline{\bm{p}}_t, ~\bm{0} \leq \bm{z}_t & \hspace{-20mm} \forall t \in \mathcal{T}(d),~\text{$\mathbb{P}_{\vert \bm\xi_{[\Downarrow(d)]}}$-a.s.}\\[3mm]
    & \bm{s}_t + \revision{\kappa^{\rm u}_t}\rho_t^{\rm{u}} \bm{u}_t- \revision{\kappa^{\rm v}_t}\rho_t^{\rm{v}} \bm{v}_t = \bm{\eta}_t \circ \bm{g}_t - \bm{\zeta}_t \circ \bm{p}_t & \hspace{-20mm} \forall t \in \mathcal{T}(d),~\text{$\mathbb{P}_{\vert \bm\xi_{[\Downarrow(d)]}}$-a.s.}\\
    & \displaystyle \underline{\bm{w}}_t \leq \bm{w}_{d-1} + \sum_{\tau=\downarrow(d)}^t \bm{\phi}_\tau + \mathbf{M}(\bm{g}_\tau - \bm{p}_\tau + \bm{z}_\tau) \leq \overline{\bm{w}}_t & \hspace{-20mm} \forall t \in \mathcal{T}(d),~\text{$\mathbb{P}_{\vert \bm\xi_{[\Downarrow(d)]}}$-a.s.}\\[3mm]
    & \displaystyle \bm{w}_d \leq \bm{w}_{d-1} + \sum_{\tau \in \mathcal{T}(d)} \bm{\phi}_\tau + \mathbf{M}(\bm{g}_\tau - \bm{p}_\tau + \bm{z}_\tau) & \hspace{-20mm} \text{$\mathbb{P}_{\vert \bm\xi_{[\Downarrow(d)]}}$-a.s.}
    \end{array}
\end{aligned}
\end{equation*}

To understand how~\eqref{opt:individual_trading} differs from~\eqref{opt:collective_trading}, recall that energy is produced
by the generators and consumed by the pumps installed along the arcs \revision{of the reservoir topology}. Under problem~\eqref{opt:individual_trading}, the company places an individual bid for every arc of its reservoir system. We gather these individual (arc-specific) bids placed on the spot market, the reserve-up market and the reserve-down market in the vectors $\bm{s}_t, \bm{u}_t, \bm{v}_t \in \mathcal{L}^A(\mathcal{F}_{[\downarrow(t)]})$ [MWh], respectively. In contrast, the previous trader's \revision{problem}~\eqref{opt:collective_trading} contains aggregate decisions $s_t, u_t, v_t  \in \mathcal{L}(\mathcal{F}_{[\downarrow(t)]})$. Note that~\eqref{opt:collective_trading} is non-inferior
to~\eqref{opt:individual_trading} because it increases the producer’s flexibility in choosing the generation and
pumping decisions. We show below that, similarly to~\eqref{opt:collective_trading}, problem~\eqref{opt:individual_trading} can undergo several simplifications. \revision{However, we will show that~\eqref{opt:individual_trading} itself eventually reduces to the tractable linear program
\begin{equation*} \tag{IT${}^{\rm{r}}$}
\label{opt:reduced_individual_trading}
\begin{aligned}
    &\text{sup} && \sum_{t \in \mathcal{T}(d)} \tilde{\pi}_t^{\rm{s}} \, \bm{1}^\top\bm{s}_t + \tilde{\pi}_t^{\rm{u}} \, \bm{1}^\top\bm{u}_t + \tilde{\pi}_t^{\rm{v}} \, \bm{1}^\top\bm{v}_t \\
    &\text{s.t.} && \bm{s}_t, \bm{u}_t, \bm{v}_t \in \mathbb{R}^A, ~\bm{g}_t, \bm{p}_t, \bm{z}_t \in \mathbb{R}^A && \forall t \in \mathcal{T}(d)\\
    &	  && \bm{0} \leq \bm{u}_t, ~\bm{0} \leq \bm{v}_t, ~\bm{0} \leq \bm{g}_t \leq \overline{\bm{g}}_t, ~\bm{0} \leq \bm{p}_t \leq \overline{\bm{p}}_t, ~\bm{0} \leq \bm{z}_t && \forall t \in \mathcal{T}(d)\\
    &	  && \bm{s}_t + \revision{\overline{\rho}^{\rm{u}}_t} \bm{u}_t = \bm{\eta}_t \circ \bm{g}_t - \bm{\zeta}_t \circ \bm{p}_t && \forall t \in \mathcal{T}(d)\\
    & 	  && \bm{s}_t - \revision{\overline{\rho}^{\rm{v}}_t} \bm{v}_t \geq - \bm{\zeta}_t \circ \overline{\bm{p}}_t && \forall t \in \mathcal{T}(d)\\
    &	  && \underline{\bm{w}}_t \leq \bm{w}_{d-1} + \sum_{\tau = \downarrow(d)}^t \bm{\phi}_\tau + \mathbf{M} ( \bm{g}_\tau - \bm{p}_\tau + \bm{z}_\tau ) \leq \overline{\bm{w}}_t && \forall t \in \mathcal{T}(d)\\
    &	  && \bm{w}_d \leq \bm{w}_{d-1} + \sum_{\tau \in \mathcal{T}(d)} \bm{\phi}_\tau + \mathbf{M} ( \bm{g}_\tau - \bm{p}_\tau + \bm{z}_\tau )
\end{aligned}
\end{equation*} 
if we impose Approximation~\ref{apx:day_ahead_water_level}, whereby $\bm{w}_d$ is restricted to $\mathcal{L}^R(\mathcal{F}_{[\uparrow(d-1)]})$}. \revision{By construction, problem~\eqref{opt:reduced_collective_trading} is non-inferior to~\eqref{opt:reduced_individual_trading}. Proposition~\ref{prop:robusttodeterministic_collective2} below, which has previously been used to prove Theorem~\ref{thm:all_collective}, shows that~\eqref{opt:reduced_individual_trading} is non-inferior to~\eqref{opt:reduced_collective_trading} as well.}

\revision{In analogy to Proposition~\ref{prop:rho_restriction_collective}, we first establish that, conditional on $\bm\xi_{[\Downarrow(d)]}$, the wait-and-see decisions $\{ (\bm{g}_t, \bm{p}_t, \bm{z}_t) \}_{t \in \mathcal{T}(d)}$ in~\eqref{opt:individual_trading} may be restricted to measurable functions of the reserve allocations and the reserve activations without sacrificing optimality. This additional information restriction gives rise to the problem}
\begin{equation}
\label{opt:individual_trading_rho}
\begin{aligned}
    &\text{sup} && \sum_{t \in \mathcal{T}(d)} \tilde{\pi}_t^{\rm{s}} \, \bm{1}^\top\bm{s}_t + \tilde{\pi}_t^{\rm{u}} \, \bm{1}^\top\bm{u}_t + \tilde{\pi}_t^{\rm{v}} \, \bm{1}^\top\bm{v}_t \\
    &\text{s.t.} && \bm{s}_t, \bm{u}_t, \bm{v}_t \in \mathbb{R}^A, ~\bm{g}_t, \bm{p}_t, \bm{z}_t \in \mathcal{L}^A(\mathcal{F}_{[\downarrow(d),t]}^\rho) && \forall t \in \mathcal{T}(d)\\
    &	  && \bm{0} \leq \bm{u}_t, ~\bm{0} \leq \bm{v}_t, ~\bm{0} \leq \bm{g}_t \leq \overline{\bm{g}}_t, ~\bm{0} \leq \bm{p}_t \leq \overline{\bm{p}}_t, ~\bm{0} \leq \bm{z}_t && \forall t \in \mathcal{T}(d),~\text{$\mathbb{P}_{\vert \bm\xi_{[\Downarrow(d)]}}$-a.s.}\\
    &	  && \bm{s}_t + \revision{\kappa^{\rm u}_t}\rho_t^{\rm{u}} \bm{u}_t - \revision{\kappa^{\rm v}_t}\rho_t^{\rm{v}} \bm{v}_t = \bm{\eta}_t \circ \bm{g}_t - \bm{\zeta}_t \circ \bm{p}_t && \forall t \in \mathcal{T}(d),~\text{$\mathbb{P}_{\vert \bm\xi_{[\Downarrow(d)]}}$-a.s.}\\
    &	  && \underline{\bm{w}}_t \leq \bm{w}_{d-1} + \sum_{\tau = \downarrow(d)}^t \bm{\phi}_\tau + \mathbf{M} ( \bm{g}_\tau - \bm{p}_\tau + \bm{z}_\tau ) \leq \overline{\bm{w}}_t && \forall t \in \mathcal{T}(d),~\text{$\mathbb{P}_{\vert \bm\xi_{[\Downarrow(d)]}}$-a.s.}\\
    &	  && \bm{w}_d \leq \bm{w}_{d-1} + \sum_{\tau \in \mathcal{T}(d)} \bm{\phi}_\tau + \mathbf{M} ( \bm{g}_\tau - \bm{p}_\tau + \bm{z}_\tau ) && \text{$\mathbb{P}_{\vert \bm\xi_{[\Downarrow(d)]}}$-a.s.},
\end{aligned}
\end{equation}

\begin{proposition}
\label{prop:rho_restriction}
Under Approximation~\ref{apx:day_ahead_water_level}, the optimal values of problems~\eqref{opt:individual_trading} and \eqref{opt:individual_trading_rho} are equal.
\end{proposition}

\begin{proof}{Proof}
As in the proof of Proposition~\ref{prop:rho_restriction_collective}, one may condition any feasible solution of problem~\eqref{opt:individual_trading} on the history of reserve activations to construct a feasible solution of problem~\eqref{opt:individual_trading_rho} that adopts the same objective value. Details are omitted for brevity. \qed
\end{proof}

\revision{Proposition~\ref{prop:rho_restriction} asserts that the individual trader's problem~\eqref{opt:individual_trading} is equivalent to the stochastic program~\eqref{opt:individual_trading_rho}, which accommodates only~$H+1$ decision stages and whose wait-and-see decisions depend only on the reserve allocations and the reserve activations.
However, problem~\eqref{opt:individual_trading_rho} still constitutes an infinite-dimensional linear program. \revision{We now show that problem~\eqref{opt:individual_trading_rho} is indeed equivalent to an efficiently solvable finite linear program. As a first step towards this goal, we derive a family of valid inequalities that may be added to problem~\eqref{opt:individual_trading_rho} without affecting its feasible set.}}

\begin{proposition}
\label{prop:down_cut}
Any feasible solution of \eqref{opt:individual_trading_rho} satisfies $\bm{s}_t -\revision{\overline{\rho}^{\rm{v}}_t} \bm{v}_t \geq - \bm{\zeta}_t \circ \overline{\bm{p}}_t$
for all $t \in \mathcal{T}(d)$.
\end{proposition}

\begin{proof}{Proof}
The proof widely parallels that of Proposition~\ref{prop:down_cut_collective} and is thus omitted for brevity. \qed
\end{proof}

~\\[-10mm]

By Proposition~\ref{prop:down_cut}, the deterministic inequalities $\bm{s}_t - \revision{\overline{\rho}^{\rm{v}}_t} \bm{v}_t \geq - \bm{\zeta}_t \circ \overline{\bm{p}}_t$ are valid for all $t \in \mathcal{T}(d)$~and can therefore be appended to problem~\eqref{opt:individual_trading_rho} without affecting its feasible set. Observe that $\bm{s}_t + \bm{\zeta}_t \circ \overline{\bm{p}}_t$ comprises the arc-wise maximum bids in the reserve-down market on which the company can deliver in case of a call-off. To see this, note that $s_{t,a}$ is the amount of energy produced on arc $a$ for the spot market and that $\zeta_{t,a} \overline{p}_{t,a}$ represents the maximum amount of energy that can be absorbed on arc~$a$ by pumping. In case of a call-off on the reserve-down market, the energy production on arc $a$ can thus be reduced at most by $s_{t,a} + \zeta_{t,a} \overline{p}_a$. Appending the valid inequalities to~\eqref{opt:individual_trading_rho}, we obtain
\begin{equation}
\label{opt:individual_trading_cuts}
\begin{aligned}
    &\text{sup} && \sum_{t \in \mathcal{T}(d)} \tilde{\pi}_t^{\rm{s}} \, \bm{1}^\top\bm{s}_t + \tilde{\pi}_t^{\rm{u}} \, \bm{1}^\top\bm{u}_t + \tilde{\pi}_t^{\rm{v}}  \, \bm{1}^\top\bm{v}_t \\
    &\text{s.t.} && \bm{s}_t, \bm{u}_t, \bm{v}_t \in \mathbb{R}^A, ~\bm{g}_t, \bm{p}_t, \bm{z}_t \in \mathcal{L}^A(\mathcal{F}_{[\downarrow(d),t]}^\rho) && \forall t \in \mathcal{T}(d)\\
    &	  && \bm{0} \leq \bm{u}_t, ~\bm{0} \leq \bm{v}_t, ~\bm{0} \leq \bm{g}_t \leq \overline{\bm{g}}_t, ~\bm{0} \leq \bm{p}_t \leq \overline{\bm{p}}_t, ~\bm{0} \leq \bm{z}_t && \forall t \in \mathcal{T}(d),~\text{$\mathbb{P}_{\vert \bm\xi_{[\Downarrow(d)]}}$-a.s.}\\
    &	  && \bm{s}_t + \revision{\kappa^{\rm u}_t}\rho_t^{\rm{u}} \bm{u}_t - \revision{\kappa^{\rm v}_t}\rho_t^{\rm{v}} \bm{v}_t = \bm{\eta}_t \circ \bm{g}_t - \bm{\zeta}_t \circ \bm{p}_t && \forall t \in \mathcal{T}(d),~\text{$\mathbb{P}_{\vert \bm\xi_{[\Downarrow(d)]}}$-a.s.}\\
    & 	  && \bm{s}_t - \revision{\overline{\rho}^{\rm{v}}_t}\bm{v}_t \geq - \bm{\zeta}_t \circ \overline{\bm{p}}_t && \forall t \in \mathcal{T}(d)\\
    &	  && \underline{\bm{w}}_t \leq \bm{w}_{d-1} + \sum_{\tau = \downarrow(d)}^t \bm{\phi}_\tau + \mathbf{M} ( \bm{g}_\tau - \bm{p}_\tau + \bm{z}_\tau ) \leq \overline{\bm{w}}_t && \forall t \in \mathcal{T}(d),~\text{$\mathbb{P}_{\vert \bm\xi_{[\Downarrow(d)]}}$-a.s.}\\
    &	  && \bm{w}_d \leq \bm{w}_{d-1} + \sum_{\tau \in \mathcal{T}(d)} \bm{\phi}_\tau + \mathbf{M} ( \bm{g}_\tau - \bm{p}_\tau + \bm{z}_\tau ) && \text{$\mathbb{P}_{\vert \bm\xi_{[\Downarrow(d)]}}$-a.s.},
\end{aligned}
\end{equation}
which  is equivalent to problem~\eqref{opt:individual_trading_rho} by virtue of Proposition~\ref{prop:down_cut}.

\revision{Next, we will show that problem~\eqref{opt:individual_trading_cuts} is equivalent to the reduced individual trader's problem~\eqref{opt:reduced_individual_trading}. Our argument critically relies on the following technical yet intuitive lemma, which asserts that there is no benefit in simultaneous generation and pumping.}





\revision{
\begin{lemma}
\label{lem:noupanddown}
The optimal value of the reduced planner's problem~\eqref{opt:reduced_individual_trading} does not decrease if we append the complementarity constraints $\bm{g}_t \circ \bm{p}_t =\bm 0$ for all $t \in \mathcal{T}(d)$.
\end{lemma}

\begin{proof}{Proof}
Consider any feasible solution $\left\{ (\bm{s}_t, \bm{u}_t, \bm{v}_t, \bm{g}_t, \bm{p}_t, \bm{z}_t) \right\}_{t \in \mathcal{T}(d)}$ of problem~\eqref{opt:reduced_individual_trading}, and construct a new solution $\left\{ (\bm{s}_t, \bm{u}_t, \bm{v}_t, \bm{g}'_t, \bm{p}'_t, \bm{z}'_t) \right\}_{t \in \mathcal{T}(d)}$ with adjusted flow decisions
\begin{equation*}
	\bm{g}'_t  = \bm{g}_t - \bm\zeta_{t} \circ \bm\Delta_t,  \quad
	\bm{p}'_t = \bm{p}_t - \bm\eta_{t}  \circ \bm\Delta_t, \quad
	\bm{z}'_t = \bm{z}_t + (\bm\zeta_{t}  - \bm\eta_{t} ) \circ \bm\Delta_t,
\end{equation*}
where $\bm\Delta_t \in \mathbb{R}^A$ is defined through $\Delta_{t,a} = \min \{ g_{t,a}/\zeta_{t,a}, \, p_{t,a}/\eta_{t,a} \}$. This new solution preserves the market decisions and consequently the objective value of the original solution, and it is readily seen to satisfy the complementarity constraints $\bm{g}'_t \circ \bm{p}'_t = \bm 0$ for all $t \in \mathcal{T}(d)$. The claim thus follows if we can show that the new solution is feasible in~\eqref{opt:reduced_individual_trading}. To this end, note that $\bm\Delta_t \geq \bm{0}$, which implies that $\bm{g}'_t \leq \bm{g}_t  \leq \overline{\bm{g}}$, $\bm{p}'_t \leq \bm{p}_t  \leq \overline{\bm{p}}$, and $\bm{z}'_t \geq \bm{0}$, where the last inequality exploits our standing assumption that $\bm\zeta_t > \bm\eta_t$. Similarly, the inequalities $\bm\zeta_t \circ \bm\Delta_t \leq \bm{g}_t$ and $\bm\eta_t \circ \bm\Delta_t \leq \bm{p}_t$ ensure that $\bm{g}'_t \geq \bm{0}$ and $\bm{p}'_t \geq \bm{0}$. Finally, observe that $\bm{g}'_t - \bm{p}'_t + \bm{z}'_t = \bm{g}_t - \bm{p}_t + \bm{z}_t$ ({\em i.e.,} the net reservoir outflows remain unchanged) and $\bm\eta_t \circ \bm{g}'_t - \bm\zeta_t \circ \bm{p}'_t = \bm\eta_t \circ \bm{g}_t - \bm\zeta_t \circ \bm{p}_t$ ({\em i.e.,} the arc-wise net energy production quantities remain unchanged), which implies that the new solution satisfies the reservoir balance constraints and the energy delivery constraints, respectively. In summary, we have shown that the new solution is indeed feasible in~\eqref{opt:reduced_individual_trading}, and thus the claim follows. \qed
\end{proof}
}

\begin{proposition}
\label{prop:robusttodeterministic}
Under Approximation~\ref{apx:day_ahead_water_level}, the optimal values of~\eqref{opt:individual_trading_cuts} and \eqref{opt:reduced_individual_trading} are equal.
\end{proposition}

\begin{proof}{Proof}
We will prove the proposition by showing that every feasible solution of problem~\eqref{opt:individual_trading_cuts} corresponds to a feasible solution of problem~\eqref{opt:reduced_individual_trading} with the same objective value and vice versa. \revision{Since the first correspondence follows from a simple adaptation of the respective argument in the proof of Proposition~\ref{prop:robusttodeterministic_collective1}, we omit its detailed proof for the sake of brevity.}

\revision{To prove that  every feasible solution of problem~\eqref{opt:reduced_individual_trading} corresponds to a feasible solution of problem~\eqref{opt:individual_trading_cuts},} we fix any feasible solution $\left\{ (\bm{s}_t, \bm{u}_t, \bm{v}_t, \bm{g}_t, \bm{p}_t, \bm{z}_t) \right\}_{t \in \mathcal{T}(d)}$ of the reduced trader's problem~\eqref{opt:reduced_individual_trading}. \revision{By Lemma~\ref{lem:noupanddown},} we may assume without loss of generality that this solution satisfies the complementarity constraints $\bm g_{t}\circ \bm p_{t} = \bm 0$ for all $t \in \mathcal{T}(d)$. We \revision{claim} that one can systematically construct flow decisions $\{ (\bm{g}'_t, \bm{p}'_t) \}_{t \in \mathcal{T}(d)}$ that satisfy
\begin{subequations}
\label{eq:alleqs}
\begin{align}
	\label{eq:nonanticipative} & {g}'_{t,a}, {p}'_{t,a} \in \mathcal{L}(\mathcal{F}_{[\downarrow(d),t]}^\rho) \\  
	\label{eq:boundg}&  0 \leq g'_{t,a} \leq \overline{g}_{t,a} & \text{$\mathbb{P}_{\vert \bm\xi_{[\Downarrow(d)]}}$-a.s.}  \\
	\label{eq:boundp}&  0 \leq p'_{t,a} \leq \overline{p}_{t,a} & \text{$\mathbb{P}_{\vert \bm\xi_{[\Downarrow(d)]}}$-a.s.}  \\
	\label{eq:delivery} & \eta_{t,a}  g'_{t,a} - \zeta_{t,a}  p'_{t,a} = s_{t,a} + \revision{\kappa^{\rm u}_t}\rho^{\rm u}_t u_{t,a} - \revision{\kappa^{\rm v}_t}\rho^{\rm v}_t v_{t,a} & \text{$\mathbb{P}_{\vert \bm\xi_{[\Downarrow(d)]}}$-a.s.}\\ 
	\label{eq:boundgp} & g'_{t,a} - p'_{t,a} \leq g_{t,a} - p_{t,a} & \text{$\mathbb{P}_{\vert \bm\xi_{[\Downarrow(d)]}}$-a.s.}
\end{align}
\end{subequations}
for all $t \in \mathcal{T}(d)$ and $a \in \mathcal{A}$. As $g_{t,a} p_{t,a} = 0$, we can distinguish two cases for each hour-arc pair $(t,a)$, which necessitate two different constructions of the corresponding flows $g'_{t,a}$ and $p'_{t,a}$. \revision{Specifically, we first assume that $g_{t,a}=0$ (Case~1), and then we assume that $g_{t,a}>0$ (Case~2). From now on all equalities and inequalities involving random variables are understood to hold $\mathbb{P}_{\vert \bm\xi_{[\Downarrow(d)]}}$-almost surely.}

\begin{paragraph}{\bf Case 1 ($g_{t,a} = 0$):} We set $g'_{t,a} = 0$ and $p'_{t,a} = -(s_{t,a} + \revision{\kappa^{\rm u}_t}\rho^{\rm u}_t u_{t,a} - \revision{\kappa^{\rm v}_t}\rho^{\rm v}_t v_{t,a})/\zeta_{t,a}$. It is easy to verify that $(g'_{t,a},p'_{t,a})$ satisfies~\eqref{eq:boundg} and~\eqref{eq:delivery}. The non-anticipativity constraints \eqref{eq:nonanticipative} are also met because \revision{the reserve allocations $\kappa^{\rm u}_t$ and $\kappa^{\rm v}_t$ are revealed in hour $\downarrow(t)$, whereas} the reserve activations~$\rho^{\rm u}_t$ and~$\rho^{\rm v}_t$ are revealed in hour $t$. In order to establish~\eqref{eq:boundp} and~\eqref{eq:boundgp}, we first observe that the given feasible solution of the reduced trader's problem~\eqref{opt:reduced_individual_trading} satisfies
\begin{equation*}
	s_{t,a} + \revision{\overline{\rho}^{\rm{u}}_t} u_{t,a} = \eta_{t,a} g_{t,a} - \zeta_{t,a} p_{t,a} \quad \Longrightarrow \quad p_{t,a} = -(s_{t,a} + \revision{\overline{\rho}^{\rm{u}}_t} u_{t,a})/\zeta_{t,a},
\end{equation*}
where the implication holds because $g_{t,a} = 0$. \revision{As $\rho^{\rm u}_t\le \overline{\rho}^{\rm{u}}_t$ and $\kappa^{\rm u}_t \leq 1$ while $\rho^{\rm v}_t$, $\kappa^{\rm v}_t$, $u_{t,a}$ and $v_{t,a}$ are non-negative}, we may thus conclude that $p_{t,a} \leq p'_{t,a}$. Hence, the constructed pumping decision $p'_{t,a}$ meets requirement~\eqref{eq:boundgp}. Finally, requirement~\eqref{eq:boundp} is satisfied because $0\leq p_{t,a} \leq p'_{t,a}$ and because
\begin{equation*}
	p'_{t,a} \; = \; -(s_{t,a} + \revision{\kappa^{\rm u}_t}\rho^{\rm u}_t u_{t,a} - \revision{\kappa^{\rm v}_t}\rho^{\rm v}_t v_{t,a})/\zeta_{t,a} \;\leq\; -(s_{t,a} - \revision{\overline{\rho}^{\rm{v}}_t} v_{t,a})/\zeta_{t,a} \; \leq \; \overline{p}_{t,a},
\end{equation*}
where the second inequality follows from the valid cut derived in Proposition~\ref{prop:down_cut}, which constitutes one of the constraints of the reduced trader's problem~\eqref{opt:reduced_individual_trading}. Thus, $g'_{t,a}$ and $p'_{t,a}$ satisfy~\eqref{eq:alleqs}.
\end{paragraph}

\begin{paragraph}{\bf Case 2 ($g_{t,a} > 0$):}
We set $g'_{t,a} = (s_{t,a} + \revision{\kappa^{\rm u}_t}\rho^{\rm u}_t u_{t,a} - \revision{\kappa^{\rm v}_t}\rho^{\rm v}_t v_{t,a})^+/\eta_{t,a}$ and $p'_{t,a} = (s_{t,a} + \revision{\kappa^{\rm u}_t}\rho^{\rm u}_t u_{t,a} - \revision{\kappa^{\rm v}_t}\rho^{\rm v}_t v_{t,a})^-/\zeta_{t,a}$. These flow decisions manifestly satisfy the non-anticipativity constraints~\eqref{eq:nonanticipative}. 

If $s_{t,a} + \revision{\kappa^{\rm u}_t}\rho^{\rm u}_t u_{t,a} - \revision{\kappa^{\rm v}_t}\rho^{\rm v}_t v_{t,a} \leq 0$, then we have $g'_{t,a}=0$ and $p'_{t,a}= -(s_{t,a} + \revision{\kappa^{\rm u}_t}\rho^{\rm u}_t u_{t,a} - \revision{\kappa^{\rm v}_t}\rho^{\rm v}_t v_{t,a})/\zeta_{t,a}$, and one can proceed as in Case~1 to show that the requirements~\eqref{eq:boundg}--\eqref{eq:boundgp} are met. From now on assume that $s_{t,a} + \revision{\kappa^{\rm u}_t}\rho^{\rm u}_t u_{t,a} - \revision{\kappa^{\rm v}_t}\rho^{\rm v}_t v_{t,a} > 0$, in which case $g'_{t,a}= (s_{t,a} + \revision{\kappa^{\rm u}_t}\rho^{\rm u}_t u_{t,a} - \revision{\kappa^{\rm v}_t}\rho^{\rm v}_t v_{t,a})/\eta_{t,a}$ and $p'_{t,a}=0$. Note first that the requirements~\eqref{eq:boundp} and~\eqref{eq:delivery} are trivially met. We further find that
\begin{equation}
	\label{eq:boundp'}
	g'_{t,a} \;=\; (s_{t,a} + \revision{\kappa^{\rm u}_t}\rho^{\rm u}_t u_{t,a} - \revision{\kappa^{\rm v}_t}\rho^{\rm v}_t v_{t,a})/\eta_{t,a} \;\leq\; (s_{t,a} + \revision{\overline{\rho}^{\rm{u}}_t} u_{t,a})/\eta_{t,a} \;=\; g_{t,a} \; \leq \;
	\overline{g}_{t,a},
\end{equation}
where the second equality follows from the constraint~$s_{t,a} + \revision{\overline{\rho}^{\rm{u}}_t} u_{t,a} = \eta_{t,a} g_{t,a} - \zeta_{t,a} p_{t,a}$ of problem~\eqref{opt:reduced_individual_trading} and the complementarity condition $g_{t,a} p_{t,a} = 0$, which implies that $p_{t,a} = 0$. Hence, requirement~\eqref{eq:boundg} is satisfied. Finally, as $p'_{t,a} =p_{t,a} = 0$, the inequality~\eqref{eq:boundp'} also implies that requirement~\eqref{eq:boundgp} is met.
Thus, $g'_{t,a}$ and $p'_{t,a}$ satisfy again all of the requirements listed in~\eqref{eq:alleqs}.
\end{paragraph}

Given the flow decisions $\{ (\bm{g}'_t, \bm{p}'_t) \}_{t \in \mathcal{T}(d)}$ constructed above, we are now ready to introduce compatible spill decisions $\bm{z}'_t = \bm{z}_t  + (\bm{g}_t - \bm{p}_t)- (\bm{g}'_t - \bm{p}'_t)$ for all $t\in\mathcal T(d)$. In the remainder of the proof we will demonstrate that the constructed solution $\left\{ (\bm{s}_t, \bm{u}_t, \bm{v}_t, \bm{g}'_t, \bm{p}'_t, \bm{z}'_t) \right\}_{t \in \mathcal{T}(d)}$ is feasible in~\eqref{opt:individual_trading_cuts}.

Note first that we need not be concerned with the constraints that only involve the market decisions $\left\{ (\bm{s}_t, \bm{u}_t, \bm{v}_t) \right\}_{t \in \mathcal{T}(d)}$, which are trivially  satisfied because $\{ (\bm{s}_t, \bm{u}_t, \bm{v}_t, \bm g_t, \bm p_t, \bm z_t)\}_{t \in \mathcal{T}(d)}$ is feasible in problem~\eqref{opt:reduced_individual_trading}. Moreover, all constraints of problem~\eqref{opt:individual_trading_cuts} that do {\em not} involve the spill decisions are satisfied because of~\eqref{eq:alleqs}. It remains to verify that the spill decisions are non-anticipative as well as non-negative and that the reservoir balance constraints are satisfied. To this end, we note first that $\bm z'_t$ inherits non-anticipativity from $\bm g'_t$ and $\bm p'_t$. Similarly, $\bm z'_t$ inherits non-negativity from~$\bm z_t$ thanks to~\eqref{eq:boundgp}. 
Finally, we highlight that $(\bm{g}'_t, \bm{p}'_t, \bm{z}'_t)$ impacts the reservoir balance constraints only through the net water outflows $\bm{g}'_t - \bm{p}'_t + \bm{z}'_t$, which coincide with $\bm{g}_t - \bm{p}_t + \bm{z}_t$ by the construction of $\bm z'_t$. This guarantees via the feasibility of $\left\{ (\bm{s}_t, \bm{u}_t, \bm{v}_t, \bm{g}_t, \bm{p}_t, \bm{z}_t) \right\}_{t \in \mathcal{T}(d)}$  in~\eqref{opt:reduced_individual_trading} that the reservoir balance constraints are satisfied. Therefore, $\left\{ (\bm{s}_t, \bm{u}_t, \bm{v}_t, \bm{g}'_t, \bm{p}'_t, \bm{z}'_t) \right\}_{t \in \mathcal{T}(d)}$ is indeed feasible in~\eqref{opt:individual_trading_cuts}.

The claim now follows because $\left\{ (\bm{s}_t, \bm{u}_t, \bm{v}_t, \bm{g}_t, \bm{p}_t, \bm{z}_t) \right\}_{t \in \mathcal{T}(d)}$ and $\left\{ (\bm{s}_t, \bm{u}_t, \bm{v}_t, \bm{g}'_t, \bm{p}'_t, \bm{z}'_t) \right\}_{t \in \mathcal{T}(d)}$ share the same market decisions, which implies that these two solutions attain the same objective values in their respective optimization problems. Thus, we can always find a feasible solution of \eqref{opt:individual_trading_cuts} that attains the same objective value as any feasible solution of the reduced trader's problem~\eqref{opt:reduced_individual_trading}. \qed
\end{proof}
~\\[-10mm]

In summary, the results of this section show that, as long as the end-of-day reservoir levels are fixed a day in advance, the infinite-dimensional trader's problem~\eqref{opt:individual_trading} is equivalent to the tractable linear program~\eqref{opt:reduced_individual_trading}, whose size scales linearly with the number $A$ of arcs in the reservoir system and the number $H$ of hours per day. This key insight is formalized in the following theorem.
\begin{theorem}
\label{thm:all_individual}
Under Approximation~\ref{apx:day_ahead_water_level}, the optimal values of~\eqref{opt:individual_trading} and \eqref{opt:reduced_individual_trading} are equal.
\end{theorem}

\begin{proof}{Proof}
This is an immediate consequence of Propositions~\ref{prop:rho_restriction},~\ref{prop:down_cut} and \ref{prop:robusttodeterministic}. \qed
\end{proof}
~\\[-10mm]

\revision{The relations between the various optimization problems studied so far in this appendix are illustrated in the upper row of Figure~\ref{fig:proofplan2}, which uses the same conventions as Figure~\ref{fig:proofplan}. To conclude the appendix, it remains to be shown that the optimal objective value of \eqref{opt:reduced_collective_trading} is upper bounded by that of~\eqref{opt:reduced_individual_trading}. This is established in the next and final proposition.} 

\begin{figure}[h]
\begin{center}
\begin{tikzpicture}[scale=0.9, every node/.style={scale=0.9}, squarednode/.style={rectangle, draw=black, fill=gray!5, very thick, minimum width=15mm, minimum height=10mm}, node distance=4.7cm]
	\node[squarednode] (I1) {\eqref{opt:individual_trading}};
	\node[squarednode] (I2) [right of=I1] {\eqref{opt:individual_trading_rho}};
	\node[squarednode] (I3) [right of=I2] {\eqref{opt:individual_trading_cuts}};
	\node[squarednode] (I4) [right of=I3] {\eqref{opt:reduced_individual_trading}};
	
	\node[squarednode] (C1) [below of=I1] {\eqref{opt:collective_trading}};
	\node[squarednode] (C2) [below of=I2] {\eqref{opt:collective_trading_rho}};
	\node[squarednode] (C3) [below of=I3] {\eqref{opt:collective_trading_cuts}};
	\node[squarednode] (C4) [below of=I4] {\eqref{opt:reduced_collective_trading}};
	
	\draw[-{Latex[length=2mm]}] ($(I1.east)+(0mm,1mm)$) -- node[above]{Proposition~\ref{prop:rho_restriction}} ($(I2.west)+(0mm,1mm)$);
	\draw[-{Latex[length=2mm]}] ($(I2.east)+(0mm,1mm)$) -- node[above]{Proposition~\ref{prop:down_cut}} ($(I3.west)+(0mm,1mm)$);
	
	\draw[{Latex[length=2mm]}-{Latex[length=2mm]}] ($(I3.east)+(0mm,0mm)$) -- node[above]{Proposition~\ref{prop:robusttodeterministic}} ($(I4.west)+(0mm,0mm)$);
	
	\draw[{Latex[length=2mm]}-,dashed] ($(I1.east)+(0mm,-1mm)$) --  ($(I2.west)+(0mm,-1mm)$);
	\draw[{Latex[length=2mm]}-,dashed] ($(I2.east)+(0mm,-1mm)$) --  ($(I3.west)+(0mm,-1mm)$);	
	
	\draw[-{Latex[length=2mm]}, dashed] (I1.south) -- (C1.north);
	\draw[-{Latex[length=2mm]}, dashed] (I2.south) -- (C2.north);
	\draw[-{Latex[length=2mm]}, dashed] (I3.south) -- (C3.north);
	
	\draw[-{Latex[length=2mm]}, dashed] ($(I4.south)+(+1mm,0mm)$) -- ($(C4.north)+(+1mm,0mm)$);
	
	\draw[{Latex[length=2mm]}-] ($(I4.south)+(-1mm,0mm)$) --node[left]{Proposition~\ref{prop:robusttodeterministic_collective2}} ($(C4.north)+(-1mm,0mm)$);
	
	\draw[-{Latex[length=2mm]}] ($(C1.east)+(0mm,1mm)$) -- node[above]{Proposition~\ref{prop:rho_restriction_collective}} ($(C2.west)+(0mm,1mm)$);
	\draw[-{Latex[length=2mm]}] ($(C2.east)+(0mm,1mm)$) -- node[above]{Proposition~\ref{prop:down_cut_collective}} ($(C3.west)+(0mm,1mm)$);
	\draw[-{Latex[length=2mm]}] ($(C3.east)+(0mm,0mm)$) -- node[above]{Proposition~\ref{prop:robusttodeterministic_collective1}} ($(C4.west)+(0mm,0mm)$);
	
	\draw[{Latex[length=2mm]}-,dashed] ($(C1.east)+(0mm,-1mm)$) --  ($(C2.west)+(0mm,-1mm)$);
	\draw[{Latex[length=2mm]}-,dashed] ($(C2.east)+(0mm,-1mm)$) --  ($(C3.west)+(0mm,-1mm)$);	
\end{tikzpicture}
\end{center}
	\caption{Illustration of the relations between different variants of the trader's problem. Dashed arcs represent trivial relaxations, and solid arcs represent non-trivial implications proved in the referenced propositions.}
	\label{fig:proofplan2}
\end{figure}
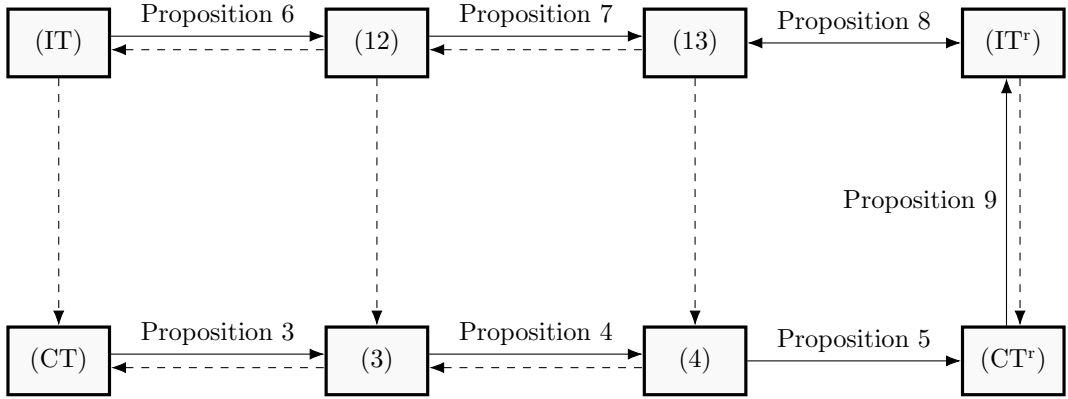

~\\[-11mm]

\revision{
\begin{proposition}
\label{prop:robusttodeterministic_collective2}
Under Approximation~\ref{apx:day_ahead_water_level}, the optimal value of problem~\eqref{opt:reduced_individual_trading} is larger than or equal to that of problem~\eqref{opt:reduced_collective_trading}.
\end{proposition}
}

\begin{proof}{Proof}
The claim follows if we can show that for every feasible solution of problem~\eqref{opt:reduced_collective_trading} there exists a feasible solution of~\eqref{opt:reduced_individual_trading} that attains the same objective function value. To achieve this, we select an arbitrary feasible solution $\{ (s_t, u_t, v_t, \bm{g}_t, \bm{p}_t, \bm{z}_t) \}_{t \in \mathcal{T}(d)}$ of problem~\eqref{opt:reduced_collective_trading} and aim to show that problem~\eqref{opt:reduced_individual_trading} admits a feasible solution $\{ (\bm{s}'_t, \bm{u}'_t, \bm{v}'_t, \bm{g}_t, \bm{p}_t, \bm{z}_t) \}_{t \in \mathcal{T}(d)}$ with the same flow decisions that satisfies $\bm{1}^\top \bm{s}'_t = s_t$, $\bm{1}^\top \bm{u}'_t = u_t$ and $\bm{1}^\top \bm{v}'_t = v_t$ for all $t \in \mathcal{T}(d)$. These identities ensure that the two solutions adopt the same objective values in their respective optimization problems.

As the flow decisions $\{ (\bm{g}_t, \bm{p}_t, \bm{z}_t )\}_{t \in \mathcal{T}(d)}$ are preserved, their upper and lower bounds as well as the reservoir level constraints are trivially satisfied. It thus suffices to show that there exist individual market bids $\{ (\bm{s}'_t, \bm{u}'_t, \bm{v}'_t )\}_{t \in \mathcal{T}(d)}$ that are consistent with the prescribed collective market bids $\{ (s_t, u_t, v_t) \}_{t \in \mathcal{T}(d)}$ and that satisfy all remaining constraints of problem~\eqref{opt:reduced_individual_trading}, that is, the non-negativity constraints, the energy delivery constraints and the valid inequalities from Proposition~\ref{prop:down_cut_collective}. Formally, such individual market bids exist if and only if the optimal value of the feasibility problem
\begin{equation}
\label{eq:primal-feasibility-problem}
\begin{aligned}
    &\text{min} && 0\\
    &\text{s.t.} && \bm{s}'_t \in \mathbb{R}^A,~\bm{u}'_t \in \mathbb{R}_+^A,~\bm{v}'_t \in \mathbb{R}_+^A \\
    &	  && \bm{1}^\top \bm{s}'_t = s_t,~ \bm{1}^\top \bm{u}'_t = u_t,~ \bm{1}^\top \bm{v}'_t = v_t \\
    &	  && \bm{s}'_t + \revision{\overline{\rho}^{\rm u}_t} \bm{u}'_t = \bm\eta_t \circ \bm{g}_t - \bm\zeta_t \circ \bm{p}_t \\
    &	  && \bm{s}'_t - \revision{\overline{\rho}^{\rm v}_t} \bm{v}'_t \geq -\bm\zeta_t \circ \overline{\bm{p}}_t
\end{aligned}
\end{equation}
vanishes for each hour $t \in \mathcal{T}(d)$, provided that $\{ (s_t, u_t, v_t, \bm{g}_t, \bm{p}_t, \bm{z}_t) \}_{t \in \mathcal{T}(d)}$ is feasible in~\eqref{opt:reduced_collective_trading}.  Assigning dual variables $\alpha, \beta, \gamma \in \mathbb{R}$ to the consistency constraints for the market bids, $\bm\lambda \in \mathbb{R}^A$ to the energy delivery constraints and $\bm\mu \in \mathbb{R}^A_+$ to the valid inequalities from Proposition~\ref{prop:down_cut_collective}, the linear program dual to the above feasibility problem can be represented as
\begin{equation}
\label{eq:dual-feasibility-problem}
\begin{aligned}
    &\text{max} && \textstyle \alpha s_t + \beta u_t + \gamma v_t + \bm\lambda^\top (\bm\eta_t \circ \bm{g}_t - \bm\zeta_t \circ \bm{p}_t) - \bm\mu^\top ( \bm\zeta_t \circ \overline{\bm{p}}_t )\\
    &\text{s.t.} && \alpha, \beta, \gamma \in \mathbb{R},~\bm\lambda \in \mathbb{R}^A,~\bm\mu \in \mathbb{R}^A_+ \\
    & && \alpha \bm{1} + \bm\lambda + \bm\mu = \bm 0 \\
    & && \beta \bm{1} + \revision{\overline{\rho}^{\rm u}_t} \bm\lambda \leq \bm 0\\
    & && \gamma \bm{1} - \revision{\overline{\rho}^{\rm v}_t} \bm\mu \leq \bm 0.
\end{aligned}
\end{equation}
Strong duality holds because the feasible set of the dual problem contains the origin. In the remainder we will argue that the objective value of any feasible solution $(\alpha, \beta, \gamma, \bm\lambda, \bm\mu)$ of the dual linear program~\eqref{eq:dual-feasibility-problem} is bounded above by $0$. To this end, we define $\underline\mu = \min_{a \in \mathcal{A}} \mu_a \ge 0$ and note that combining the first two constraints in~\eqref{eq:dual-feasibility-problem} yields $(\beta - \revision{\overline{\rho}^{\rm u}_t} \alpha)\bm{1} - \revision{\overline{\rho}^{\rm u}_t}\bm\mu \leq \bm{0}$, which in turn implies that $\beta - \revision{\overline{\rho}^{\rm u}_t} \alpha - \revision{\overline{\rho}^{\rm u}_t} \underline\mu \leq 0$. We thus find that the objective value in~\eqref{eq:dual-feasibility-problem} satisfies
\begin{equation*}
\begin{aligned}
	&\alpha s_t + \beta u_t + \gamma v_t + \bm\lambda^\top (\bm\eta_t \circ \bm{g}_t - \bm\zeta_t \circ \bm{p}_t) - \bm\mu^\top ( \bm\zeta_t \circ \overline{\bm{p}}_t ) \\
	= \quad &\alpha s_t + \beta u_t + \gamma v_t + (\bm\lambda + \bm\mu)^\top (\bm\eta_t \circ \bm{g}_t - \bm\zeta_t \circ \bm{p}_t) - \bm\mu^\top (\bm\eta_t \circ \bm{g}_t + \bm\zeta_t \circ (\overline{\bm{p}}_t - \bm{p}_t)) \\
	\leq \quad &\alpha s_t + \beta u_t + \gamma v_t - \alpha (\bm\eta_t^\top \bm{g}_t - \bm\zeta_t^\top \bm{p}_t) - \underline\mu (\bm\eta_t^\top \bm{g}_t + \bm\zeta_t^\top (\overline{\bm{p}}_t - \bm{p}_t))  \\
	= \quad &\alpha s_t + \beta u_t + \gamma v_t - \alpha (s_t + \revision{\overline{\rho}^{\rm u}_t} u_t) - \underline\mu (s_t + \revision{\overline{\rho}^{\rm u}_t} u_t + \bm\zeta_t^\top \overline{\bm{p}}_t) \\
	= \quad &(\beta - \revision{\overline{\rho}^{\rm u}_t} \alpha - \revision{\overline{\rho}^{\rm u}_t}\underline\mu)u_t  + \gamma v_t  - \underline\mu (s_t + \bm\zeta_t^\top \overline{\bm{p}}_t) \\
	\leq \quad &(\beta - \revision{\overline{\rho}^{\rm u}_t} \alpha - \revision{\overline{\rho}^{\rm u}_t} \underline\mu)u_t  + \underline\mu (\revision{\overline{\rho}^{\rm v}_t} v_t - s_t - \bm\zeta_t^\top \overline{\bm{p}}_t) ~\leq ~ 0,
\end{aligned}
\end{equation*}
where the first inequality exploits the relations $\bm\lambda + \bm\mu = -\alpha\bm{1}$ (by the feasibility of $\alpha$, $\bm \lambda$ and $\bm \mu$ in~\eqref{eq:dual-feasibility-problem}) and $\underline\mu \bm{1} \leq \bm\mu$ (by the construction of $\underline\mu$) together with the non-negativity of $\bm{g}_t$ and $\overline{\bm{p}}_t - \bm{p}_t$ (by the  feasibility of $\bm g_t$ and $\bm p_t$ in~\eqref{opt:reduced_collective_trading}). The second equality follows from the energy delivery constraints in~\eqref{opt:reduced_collective_trading}, and the second inequality holds because $\gamma \leq \revision{\overline{\rho}^{\rm v}_t}\underline\mu$ (by the feasibility of $\gamma$ in~\eqref{eq:dual-feasibility-problem} and the construction of $\underline\mu$) and $v_t \geq 0$ (by the feasibility of $v_t$ in~\eqref{opt:reduced_collective_trading}). The last inequality, finally, follows from our earlier observations that $\beta - \revision{\overline{\rho}^{\rm u}_t} \alpha - \revision{\overline{\rho}^{\rm u}_t} \underline\mu \leq 0$ and $\underline\mu \geq 0$ combined with the relations $u_t \geq 0$ and $\revision{\overline{\rho}^{\rm v}_t}v_t - s_t - \bm\zeta_t^\top \overline{\bm{p}}_t \leq 0$ (by the feasibility of $s_t$, $u_t$ and $v_t$ in~\eqref{opt:reduced_collective_trading}). 

In conclusion, we have demonstrated the that optimal value of the dual feasibility problem~\eqref{eq:dual-feasibility-problem}---and thus also that of its primal counterpart---must vanish. As our arguments hold for any feasible solution $\{ (s_t, u_t, v_t, \bm{g}_t, \bm{p}_t, \bm{z}_t) \}_{t \in \mathcal{T}(d)}$ of~\eqref{opt:reduced_collective_trading} and for any $t\in\mathcal T(d)$, the claim follows. \qed
\end{proof}

\end{APPENDIX}

\end{document}